\documentclass[11pt]{amsart}
\usepackage{amssymb}
\usepackage{amscd}
\usepackage{amsmath}
\usepackage{amsmath}
\usepackage[all]{xy}
\usepackage{mathrsfs}

\usepackage{amsfonts}
\usepackage{color}

\theoremstyle{plain}
\newtheorem{theorem}{Theorem}[section]
\newtheorem{lemma}[theorem]{Lemma}
\newtheorem{corollary}[theorem]{Corollary}
\newtheorem{proposition}[theorem]{Proposition}
\newtheorem{definition}[theorem]{Definition}

\newtheorem{conjecture}[theorem]{Conjecture}

\theoremstyle{definition}
\newtheorem{remark}[theorem]{Remark}

  1

\DeclareMathOperator{\Cl}{Cl}

\DeclareMathOperator{\Ext}{Ext}
\DeclareMathOperator{\Gal}{Gal}

\DeclareMathOperator{\Hom}{Hom}
\DeclareMathOperator{\End}{End}

\DeclareMathOperator{\Spec}{Spec}
\DeclareMathOperator{\N}{N}

\DeclareMathOperator{\im}{im}
\DeclareMathOperator{\coker}{coker}
\DeclareMathOperator{\Fit}{Fitt}
\DeclareMathOperator{\Fitt}{Fitt}

\newcommand{\CC}{\mathbb{C}}

\newcommand{\bbC}{\mathbb{C}}

\newcommand{\GG}{\mathbb{G}}

\newcommand{\QQ}{\mathbb{Q}}
\newcommand{\Q}{\mathbb{Q}}
\newcommand{\RR}{\mathbb{R}}
\newcommand{\br}{\mathbb{R}}
\newcommand{\bbR}{\mathbb{R}}
\newcommand{\ZZ}{\mathbb{Z}}

\newcommand{\Ann}{\mathrm{Ann}}
\newcommand{\bz}{\mathbb{Z}}
\newcommand{\bbZ}{\mathbb{Z}}

\newcommand{\bbQ}{\mathbb{Q}}
\newcommand{\Log}{\mathrm{Log}}
\newcommand{\ord}{\mathrm{ord}}
\newcommand{\Ord}{\mathrm{Ord}}
\newcommand{\Rec}{\mathrm{Rec}}
\newcommand{\rec}{\mathrm{rec}}

\begin{document}

\title[]{On zeta elements for $\GG_{m}$}

\author{David Burns, Masato Kurihara and Takamichi Sano}

\address{King's College London,
Department of Mathematics,
London WC2R 2LS,
U.K.}
\email{david.burns@kcl.ac.uk}

\address{Keio University,
Department of Mathematics,
3-14-1 Hiyoshi\\Kohoku-ku\\Yokohama\\223-8522,
Japan}
\email{kurihara@math.keio.ac.jp}

\address{Keio University,
Department of Mathematics,
3-14-1 Hiyoshi\\Kohoku-ku\\Yokohama\\223-8522,
Japan}
\email{tkmc310@a2.keio.jp}


\begin{abstract}
 In this paper, we present a unifying approach to the general theory of abelian Stark conjectures. 
To do so we define natural notions of `zeta element', of `Weil-\'etale cohomology complexes' and of `integral Selmer groups' for the multiplicative group $\GG_m$ over finite abelian extensions of number fields. We then conjecture a precise connection between zeta elements and  Weil-\'etale cohomology complexes, we show this conjecture is equivalent to a special case of the equivariant Tamagawa number conjecture and we give an unconditional proof of the analogous statement for global function fields. We also show that the conjecture entails much detailed information about the arithmetic properties of generalized Stark elements including a new family of integral congruence relations between Rubin-Stark elements (that refines recent conjectures of Mazur and Rubin and of the third author) and explicit formulas in terms of these elements for the higher Fitting ideals of the integral Selmer groups of $\mathbb{G}_m$, thereby obtaining a clear and very general approach to the theory of abelian Stark conjectures. As first applications of this approach, we derive, amongst other things, a proof of (a refinement of) a conjecture of Darmon concerning cyclotomic units, a proof of (a refinement of) Gross's `Conjecture for Tori' in the case that the base field is $\QQ$, a proof of new cases of the equivariant Tamagawa number conjecture in situations in which the relevant $p$-adic $L$-functions have trivial zeroes, explicit conjectural formulas for both annihilating elements and, in certain cases, the higher Fitting ideals (and hence explicit structures) of ideal class groups, a reinterpretation of the $p$-adic Gross-Stark Conjecture in terms of the properties of zeta elements and a strong refinement of many previous results (of several authors) concerning abelian Stark conjectures.

\end{abstract}

\maketitle

%
\section{Introduction} \label{Intro}

The study of the special values of zeta functions
and, more generally, of $L$-functions is a central theme in number theory that has a long tradition stretching back to Dirichlet and Kummer in the nineteenth century. In particular, much work has been done concerning the arithmetic properties of the
special values of $L$-functions and their incarnations in appropriate arithmetic cohomology groups, or `zeta elements' as they are commonly known.

The aim of our project is to systematically study the fine arithmetic properties of such zeta elements and thereby to obtain both generalizations and refinements of a wide range of well-known results and conjectures in the area. 

In this first article we shall concentrate, for primarily pedagogical reasons, on the classical and very concrete case of the $L$-functions that are attached to the multiplicative group $\GG_m$ over a finite
abelian extension $K/k$ of global fields. In subsequent articles we will then investigate the key Iwasawa-theoretic aspects of our approach (see \cite{aze2}) and also explain how the conjectures and results presented here naturally extend both to the case of Galois extensions that are not abelian and to the case of the zeta elements that are associated (in general conjecturally) to a wide class of motives over number fields.

The main results of the present article are given below as Theorems \ref{MT1}, \ref{MT2}, \ref{MC4} and
\ref{MT4}. In the rest of this introduction we state these results and also discuss a selection of interesting consequences.

\subsection{}To do this we fix a finite abelian extension of global fields $K/k$ with Galois group $G=\Gal(K/k)$. We then fix a finite non-empty set of places $S$ of $k$ containing both the set $S_{\rm ram}(K/k)$ of places which ramify in $K/k$ and the set $S_\infty(k)$ of archimedean places (if any).
 Lastly we fix an auxiliary finite non-empty set of places $T$ of $k$ which is disjoint from $S$ and such that the group $\mathcal{O}^\times_{K,S,T}$ of $S$-units of $K$ that are endowed with a trivialization at each place of $K$ above a place in $T$ is $\ZZ$-torsion-free
(for the precise definition of $\mathcal{O}^\times_{K,S,T}$, see \S\ref{notation section}).

In this article, we shall first formulate (as Conjecture \ref{ltc}) a precise `leading term conjecture' ${\rm LTC}(K/k)$ for the extension $K/k$. This conjecture predicts that the canonical {\it zeta element} $z_{K/k,S,T}$ that interpolates the leading terms at $s=0$ of the ($S$-truncated $T$-modified) $L$-functions $L_{k,S,T}(\chi, s)$ should generate the determinant module over $G$ of a canonical `$T$-modified Weil-\'etale cohomology' complex for $\GG_{m}$ that we introduce here (for details see \S\ref{wec}).

The main result of the first author in \cite{gmcrc} implies that {\rm LTC}$(K/k$) is valid if $k$ is a global function field.

In the number field case our formulation of ${\rm LTC}(K/k)$ is motivated by the `Tamagawa Number Conjecture' formulated by
Bloch and Kato in \cite{BK} and by the `generalized Iwasawa main conjecture' studied by Kato in \cite{K1} and \cite{K2}. In particular, we shall show that for extensions $K/k$ of number fields ${\rm LTC}(K/k)$ is equivalent to
the relevant special case of the `equivariant Tamagawa number conjecture' formulated in the article \cite{BF Tamagawa} of Flach and the first author. Taken in conjunction with previous work of several authors, this fact implies that ${\rm LTC}(K/k)$ is also unconditionally valid for several important families of number fields.

We assume now that $S$ contains a subset $V=\{v_{1},\ldots,v_{r}\}$ of places which split completely in $K$.
In this context,  one can use the values at $s=0$ of the $r$-th
derivatives of $S$-truncated $T$-modified $L$-functions to define a canonical element
$$\epsilon_{K/k,S,T}^V$$
in the exterior power module $\bigwedge_{\ZZ[G]}^r
\mathcal{O}_{K,S,T}^\times \otimes \RR$ (for the precise definition see \S \ref{secrs}).

As a natural generalization of a classical conjecture of Stark (dealing with the case $r=1$) Rubin conjectured in \cite{R} that the elements
 $\epsilon_{K/k,S,T}^V$ should always satisfy certain precise integrality conditions (for more details see Remark \ref{MT2 remark}). As in now common in the literature, in the sequel we shall refer to $\epsilon_{K/k,S,T}^V$ as the `Rubin-Stark element' (relative to the given data) and to the central conjecture of Rubin in \cite{R} as the `Rubin-Stark Conjecture'.


In some very special cases $\epsilon_{K/k,S,T}^V$ can be explicitly computed and the Rubin-Stark Conjecture verified. For example, this is the case if $r=0$ (so $V=\emptyset$) when $\epsilon_{K/k,S,T}^V$ can be described in terms of Stickelberger elements and if $k=\QQ$ and $V=\{\infty\}$ when $\epsilon_{K/k,S,T}^V$ can be described in terms of cyclotomic units.

As a key step in our approach we show that in all cases the validity of LTC$(K/k)$ implies that  $\epsilon_{K/k,S,T}^V$ can be computed as `the canonical projection' of the zeta element $z_{K/k,S,T}$.

This precise result is stated as Theorem \ref{zetars} and its proof will also incidentally show that LTC$(K/k)$ implies the validity of the Rubin-Stark conjecture for $K/k$. The latter implication was in fact already observed by the first author in \cite{burns} (and the techniques developed in loc. cit. have since been used by several other authors) but we would like to point out that the proof presented here is very much simpler than that given in \cite{burns} and is therefore much more amenable to subsequent generalization.

\subsection{}The first consequence of Theorem \ref{zetars} that we record here concerns a refined version of
a conjecture that was recently formulated independently by Mazur and Rubin in \cite{MR2} and
by the third author in \cite{sano}.

To discuss this we fix an intermediate field $L$ of $K/k$ and a subset $V'=\{v_{1},\ldots,v_{r'}\}$ of $S$ which contains $V$ and is such that every place in $V'$ splits completely in $L$.

In this context it is known that the elements $\epsilon_{K/k,S,T}^V$ naturally constitute an Euler system of rank $r$ and the elements
$\epsilon_{L/k,S,T}^{V'}$ an Euler system of rank $r'$. If $r<r'$, then the image of $\epsilon_{K/k,S,T}^V$ under the map induced by the field theoretic norm $K^\times\to L^\times$ vanishes. However, in this case Mazur and Rubin (see \cite[Conj. 5.2]{MR2}) and
the third author (see \cite[Conj. 3]{sano}) independently observed that the reciprocity maps
of local class field theory lead to an important conjectural relationship between the elements $\epsilon_{K/k,S,T}^V$ and $\epsilon_{L/k,S,T}^{V'}$.

We shall here formulate an interesting refinement ${\rm MRS}(K/L/k,S,T)$ of the central conjectures of \cite{MR2} and \cite{sano} (see Conjecture \ref{mrsconj} and the discussion of Remark \ref{refined sano}) and we shall then prove the following result.

\begin{theorem} \label{MT1} ${\rm LTC}(K/k)$ implies the validity of ${\rm MRS}(K/L/k,S,T)$.
\end{theorem}

This result is both a generalization and strengthening of the main result of the third author in \cite[Th. 3.22]{sano} and provides strong evidence for
${\rm MRS}(K/L/k,S,T)$.

As already remarked earlier, if $k$ is a global function field, then the validity of ${\rm LTC}(K/k)$ is a consequence of the main result of \cite{gmcrc}. In addition, if $k = \QQ$, then the validity of ${\rm LTC}(K/k)$ follows from the work of Greither and the first author in \cite{bg} and of Flach in \cite{fg}.

Theorem \ref{MT1} therefore has the following consequence.

\begin{corollary} \label{MC1}
${\rm MRS}(K/L/k,S,T)$ is valid if $k=\QQ$ or if $k$ is a global function
field. \end{corollary}

This result is of particular interest since it verifies the conjectures of Mazur and Rubin \cite{MR2} and of the third author \cite{sano} even in cases for which one has $r>1$.

A partial converse to Theorem \ref{MT1} will be given in Theorem \ref{MT4} below and this is then used in Corollary \ref{IntroConsequencesOfVentullo} to derive further new evidence in support of the conjecture  ${\rm LTC}(K/k)$, and hence also ${\rm MRS}(K/L/k,S,T)$.

Next we recall that in \cite{D} Darmon used the theory of cyclotomic units to formulate a refined version
of the class number formula for the class groups of real quadratic fields. We further recall that Mazur and Rubin in \cite{MR}, and later the third author in \cite{sano},
have proved the validity of the central conjecture of \cite{D} but only after inverting the prime $2$.

We shall formulate in \S \ref{dargro} a natural refinement of
Darmon's conjecture. By using Corollary \ref{MC1} we shall then give a full proof of
our refined version of Darmon's conjecture, thereby obtaining the following result (for a precise version of which see Theorem \ref{darconj}).

\begin{corollary} \label{MC2} A natural refinement of Darmon's conjecture in \cite{D} is valid.
\end{corollary}

Let now $K/k$ be an abelian extension as above and choose intermediate fields $L$ and $\widetilde L$ with $[L:k]=2$, $L \cap  \widetilde L=k$ and
$K=L  \widetilde L$. In this context Gross has formulated in \cite{G} a `conjecture for tori' regarding the value of the canonical Stickelberger element associated to $K/k$ modulo a certain ideal constructed from class numbers and a canonical integral regulator map. This conjecture has been widely studied in the literature, perhaps most notably by Hayward in \cite{Hay} and by Greither and Ku\v cera in \cite{GK0,GK}.

We shall formulate (as Conjecture \ref{tori conj}) a natural refinement of Gross's conjecture for tori and we shall then prove (in  Theorem \ref{mrstori}) that the validity of this refinement is a consequence of ${\rm MRS}(K/L/k,S,T)$.

As a consequence of Corollary \ref{MC1} we shall therefore obtain the following result.

\begin{corollary} \label{MC3}
A natural refinement of Gross's conjecture for tori is valid if
 $k=\QQ$ or if $k$ is a global function field.
\end{corollary}

This result is a significant improvement of the main results of Greither and Ku\v cera in \cite{GK0,GK}. In particular, whilst the latter articles only study the case that $k=\QQ$, $L$ is an imaginary quadratic field, and $\widetilde L/\QQ$ is
an abelian extension satisfying several technical conditions
(see Remark \ref{GreitherKucera}), Corollary \ref{MC3} now proves Gross's conjecture completely in the case $k=\QQ$
 and with no assumption on either $L$ or $\widetilde L$.

\subsection{}
In order to state our second main result, we introduce two new Galois modules which
are each finitely generated abelian groups and will play a key role in the arithmetic theory of zeta elements.

The first of these is
a canonical `($\Sigma$-truncated $T$-modified) integral dual Selmer group' $\mathcal{S}_{\Sigma,T}(\GG_{m/K})$ for the multiplicative group over $K$
for each finite non-empty set of places $\Sigma$ of $K$ that contains $S_\infty(K)$ and each finite set of places $T$ of $K$ that is disjoint from $\Sigma$.

If $\Sigma = S_\infty(K)$ and $T$ is empty, then $\mathcal{S}_{\Sigma,T}(\GG_{m/K})$ is simply defined to be the cokernel of the map
$$\prod_{w}\ZZ \longrightarrow \Hom_\ZZ(K^\times,\ZZ), \,\,\,\,\,\,(x_w)_w \mapsto (a \mapsto \sum_{w}\ord_w(a)x_w),$$
where in the product and sum $w$ runs over all finite places of $K$, and in this case constitutes a canonical integral structure on the Pontryagin dual
 of the Bloch-Kato Selmer group $H^{1}_{f}(K, \QQ/\ZZ(1))$ (see Remark \ref{notation2}(i)).

In general, the group $\mathcal{S}_{\Sigma,T}(\GG_{m/K})$ is defined to be a
 natural analogue for $\mathbb{G}_m$ of the `integral Selmer group' that was introduced for abelian varieties by Mazur and Tate in \cite{mt} and, in particular, lies in a canonical exact sequence of $G$-modules of the form
\begin{equation}\label{selseq}0 \longrightarrow \Hom_\ZZ({\rm Cl}_{\Sigma}^T(K), \QQ/\ZZ)
\longrightarrow
\mathcal{S}_{\Sigma,T}(\GG_{m/K}) \longrightarrow
\Hom_{\ZZ}(\mathcal{O}^\times_{K,\Sigma,T},\ZZ) \longrightarrow 0
\end{equation}
where ${\rm Cl}_{\Sigma}^T(K)$ is the ray class group of $\mathcal{O}_{K,\Sigma}$ modulo the product of all
places of $K$ above $T$ (see \S\ref{notation section}).

This Selmer group is also philosophically related to the theory of Weil-\'etale cohomology that is conjectured to exist by Lichtenbaum in \cite{L3}, and in this direction we show that in all cases there is a natural identification
$$\mathcal{S}_{\Sigma,T}(\GG_{m/K})=H^2_{\!c,T}((\mathcal{O}_{\!K,\Sigma})_{\mathcal{W}}, \ZZ)$$
where the right hand group denotes the cohomology in degree two of a canonical `$T$-modified compactly supported Weil-\'etale cohomology complex' that we introduce in \S\ref{wec}.

The second module $\mathcal{S}^{{\rm tr}}_{\Sigma,T}(\GG_{m/K})$ that we introduce is a canonical `transpose'
(in the sense of Jannsen's homotopy theory of modules \cite{jannsen}) for $\mathcal{S}_{\Sigma,T}(\GG_{m/K})$.

In terms of the complexes introduced in \S\ref{wec} this module can be described as a certain `$T$-modified Weil-\'etale cohomology group' of $\mathbb{G}_m$
$$\mathcal{S}^{{\rm tr}}_{\Sigma,T}(\GG_{m/K})=H^1_{T}((\mathcal{O}_{\!K,\Sigma})_{\mathcal{W}},\mathbb{G}_m)$$
and can also be shown to lie in a canonical exact sequence of $G$-modules of the form
\begin{equation}\label{dual ses}
0 \longrightarrow {\rm Cl}_{\Sigma}^T(K) \longrightarrow
\mathcal{S}^{{\rm tr}}_{\Sigma,T}(\GG_{m/K})
\longrightarrow X_{K,\Sigma} \longrightarrow 0.
\end{equation}
Here $X_{K,\Sigma}$ denotes the subgroup of the free abelian group
on the set $\Sigma_{K}$ of places of $K$ above $\Sigma$ comprising elements whose coefficients sum to zero.


We can now state our second main result.

In this result we write $\Fitt^{r}_{G}(M)$ for the $r$-th Fitting ideal of a finitely generated $G$-module $M$,
though the usual notation is $\Fitt_{r, \ZZ[G]}(M)$, in order to make the notation consistent with the exterior power $\bigwedge_{\ZZ[G]}^{r} M$.
 Note that we will review the definition of higher Fitting ideals in \S \ref{DefinitionRelativeFitt} and also introduce there for each finitely generated $G$-module $M$ and each pair of non-negative integers $r$ and $i$ a natural notion of `higher relative Fitting ideal'
\[ \Fitt^{(r,i)}_{G}(M)=\Fitt^{(r,i)}_{G}(M,M_{\rm tors}).\]
%

We write $x \mapsto x^{\#}$ for the $\CC$-linear involution of $\CC[G]$ which
inverts elements of $G$.

\begin{theorem} \label{MT2}
Let $K/k, S, T, V$ and $r$ be as above, and assume that ${\rm LTC}(K/k)$ is valid. Then all of the following claims are also valid.
\begin{itemize}
\item[(i)]{One has
$$\Fit_G^r(\mathcal{S}_{S,T}(\GG_{m/K}))=\{  \Phi(\epsilon_{K/k,S,T}^V)^{\#}
:  \Phi \in \bigwedge_{\ZZ[G]}^r \Hom_{\ZZ[G]}(\mathcal{O}_{K,S,T}^\times,\ZZ[G])\}.$$
}
\item[(ii)]{Let $\mathcal{P}_k(K)$ be
the set of all places which split completely in $K$. Fix a non-negative integer $i$ and set
$$\mathcal{V}_{i}=\{V' \subset \mathcal{P}_k(K)
: |V'|=i \ \mbox{and} \ V' \cap (S \cup T)=\emptyset \}.$$
Then one has
$$\Fitt_{G}^{(r,i)}(\mathcal{S}^{{\rm tr}}_{S,T}(\GG_{m/K}))
=
\{  \Phi(\epsilon_{K/k,S \cup V',T}^{V \cup V'})  :  V' \in \mathcal{V}_{i} \
\mbox{and} \
\Phi \in \bigwedge_{\ZZ[G]}^{r+i} \Hom_{\ZZ[G]}(\mathcal{O}_{K,S \cup V',T}^\times,\ZZ[G])\}.$$
In particular, if $i=0$, then one has
$$\Fit_G^r(\mathcal{S}^{{\rm tr}}_{S,T}(\GG_{m/K}))=\{  \Phi(\epsilon_{K/k,S,T}^V)  :  \Phi \in \bigwedge_{\ZZ[G]}^r \Hom_{\ZZ[G]}(\mathcal{O}_{K,S,T}^\times,\ZZ[G])\}.$$
}
\end{itemize}
\end{theorem}


\begin{remark}\label{MT2 remark} In terms of the notation of Theorem \ref{MT2}, the Rubin-Stark Conjecture asserts that $\Phi(\epsilon_{K/k,S,T}^V)$ belongs to $\ZZ[G]$ for every $\Phi$ in $\bigwedge_{\ZZ[G]}^r \Hom_{\ZZ[G]}(\mathcal{O}_{K,S,T}^\times,\ZZ[G])$. The property described in Theorem \ref{MT2} is deeper in that it shows the ideal generated by
$\Phi(\epsilon_{K/k,S,T}^V)$ as $\Phi$ runs over $\bigwedge_{\ZZ[G]}^r \Hom_{\ZZ[G]}(\mathcal{O}_{K,S,T}^\times,\ZZ[G])$
 should encode significant arithmetic information relating to integral (dual) Selmer groups. (See also Remark \ref{alb} in this regard.)\end{remark}

\subsection{} \label{IntroHigherFitt}
In this subsection we shall discuss some interesting consequences of Theorem \ref{MT2}
concerning the explicit Galois structure of ideal class groups.

To do this we fix an odd prime $p$ and suppose that $K/k$ is any finite abelian extension of global fields.
We write $L$ for the (unique) intermediate field of $K/k$ such that $K/L$ is a $p$-extension
and $[L:k]$ is prime to $p$. Then the group $\Gal(K/k)$ decomposes as a direct product $\Gal(L/k)\times \Gal(K/L)$ and we fix a non-trivial faithful character $\chi$ of $\Gal(L/k)$. We set ${\rm Cl}^{T}(K):={\rm Cl}_{\emptyset}^T(K)$ and define its `$(p,\chi)$-component' by setting
\[ A^{T}(K)^{\chi}:=({\rm Cl}^T(K) \otimes \ZZ_{p}) \otimes_{\ZZ_{p}[\Gal(L/k)]}
\mathcal{O}_{\chi}.\]
Here we write $\mathcal{O}_{\chi}$ for the module $\ZZ_{p}[\im (\chi)]$ upon which $\Gal(L/k)$ acts via $\chi$ so that $A^{T}(K)^{\chi}$ has an induced action of the group ring $R_{K}^{\chi}:=\mathcal{O}_{\chi}[\Gal(K/L)]$.

Then in Theorem \ref{ThCharacterComponent1} we shall derive the following results about the structure of $A^{T}(K)^{\chi}$ from the final assertion of Theorem \ref{MT2}(ii).

In this result we write `$\chi(v) \neq 1$' to mean that the decomposition group of $v$ in $\Gal(L/k)$ is non-trivial.

\begin{corollary} \label{ThCharacterComponent01}
Let $r$ be the number of archimedean places of $k$ that split completely in $K$.
Assume that
any ramifying place $v$ of $k$ in $K$ satisfies $\chi(v) \neq 1$. Assume also that the equality of ${\rm LTC}(K/k)$ is valid after applying the functor $-\otimes_{\ZZ_{p}[\Gal(L/k)]}
\mathcal{O}_{\chi}$.

Then for any non-negative integer $i$ one has an equality
$$\Fitt_{R_{K}^{\chi}}^{i}(A^T(K)^{\chi})
=
\{  \Phi(\epsilon_{K/k,S \cup V',T}^{V \cup V', \chi})  :  V' \in \mathcal{V}_{i} \
\mbox{and} \
\Phi \in \bigwedge_{R_{K}^{\chi}}^{r+i} \mathcal{H}_{\chi} \}$$
where we set $S:=S_{\infty}(k) \cup S_{\rm ram}(K/k)$ and $\mathcal{H}_{\chi}:=
\Hom_{R_{K}^{\chi}}((\mathcal{O}_{K,S \cup V',T}^\times
\otimes \ZZ_{p})^{\chi},R_{K}^{\chi})$.
\end{corollary}

We remark that Corollary \ref{ThCharacterComponent01} specializes to give refinements of several results in the literature.

For example, if $k =\QQ$ and $K$ is equal to the maximal totally real subfield $\QQ(\zeta_m)^+$ of $\QQ(\zeta_{m})$ where $\zeta_{m}$
is a fixed choice of primitive $m$-th root of unity for some natural number $m$, then ${\rm LTC}(K/k)$ is known to be valid and so Corollary \ref{ThCharacterComponent01} gives an explicit description of the higher Fitting ideals of ideal class groups in terms of
cyclotomic units (which are the relevant Rubin-Stark elements in this case). In particular, if $m=p^{n}$ for any non-negative integer $n$, then the necessary condition
on $\chi$ is satisfied for all non-trivial $\chi$ and Corollary \ref{ThCharacterComponent01}
gives a strong refinement of Ohshita's theorem in \cite{ohshita} for the field $K=\QQ(\zeta_{p^{n}})^{+}$.

The result is also stronger than that of Mazur and Rubin in \cite[Theorem 4.5.9]{Koly} since the latter describes structures over a discrete valuation ring whilst Corollary \ref{ThCharacterComponent01} describes structures over the group ring $R_{K}^{\chi}$.

In addition, if $K$ is a CM extension of a totally real field $k$, then Corollary \ref{ThCharacterComponent01} constitutes a generalization of
the main results of the second author in both \cite{Ku3} and \cite{Ku5}. To explain this we suppose that $K/k$ is a CM-extension and that
$\chi$ is an odd character. Then classical Stickelberger elements can be used to define for each non-negative integer $i$ a `higher Stickelberger ideal'
\[ \Theta^{i}(K/k) \subseteq \ZZ_{p}[\Gal(K/k)]\]
(for details see \S \ref{strclassgroup}). By taking $T$ to be empty we can consider the $(p,\chi)$-component of the usual ideal class group

\[ A(K)^{\chi} := ({\rm Cl}^T(K) \otimes \ZZ_{p}) \otimes_{\ZZ_{p}[\Gal(L/k)]}
\mathcal{O}_{\chi}.\]
Then, by using both Theorem \ref{MT1} and Corollary \ref{ThCharacterComponent01} we shall derive the following result as a consequence of the more general Theorem \ref{ThHigherFittCM}.

In this result we write $\omega$ for the Teichm\"{u}ller character giving the Galois
action on the group of $p$-th roots of unity.

\begin{corollary} \label{ThCharacterComponent02}
Let $K$ be a CM-field, $k$ totally real, and $\chi$ an odd
character with $\chi \neq \omega$.
We assume that
any ramifying place $v$ of $k$ in $K$ satisfies $\chi(v) \neq 1$ and that
${\rm LTC}(F/k)$ is valid for certain extensions $F$ of $K$
(see Theorem \ref{ThHigherFittCM} for the precise conditions on $F$).

Then for any non-negative integer $i$ one has an equality
\[ \Fitt_{R_{K}^{\chi}}^{i}(A(K)^{\chi})
= \Theta^{i}(K/k)^{\chi}.\]
\end{corollary}

In the notation of Corollary \ref{ThCharacterComponent02} suppose that $K$ is the $n$-th layer of the cyclotomic $\ZZ_{p}$-extension
of $L$ for some non-negative integer $n$ and that every place $\mathfrak p$ above $p$
satisfies $\chi(\mathfrak p) \neq 1$. Then the conditions on $\chi(v)$ and ${\rm LTC}(F/k)$ that are stated in Corollary \ref{ThCharacterComponent02} are automatically satisfied and so Corollary \ref{ThCharacterComponent02} generalizes the main results of the second author in \cite{Ku5}.

To get a better feeling for Corollary \ref{ThCharacterComponent02} consider the simple case that $[K:k]$ is prime to $p$. In this case $K=L$, the ring $\ZZ_{p}[\Gal(K/k)]$ is semi-local and
$A(K)^{\chi}$ is a module over the discrete valuation ring
$\mathcal{O}_{\chi}=R_{K}^\chi$. Then the conclusion in Corollary \ref{ThCharacterComponent02} in the case $i=0$ implies that
\begin{equation} \label{OrderIntro}
|A(K)^{\chi}| =
|\mathcal{O}_{\chi}/L_{k}(\chi^{-1},0)|
\end{equation}
where $L_{k}(\chi^{-1},s)$ is the usual Artin $L$-function.
If every place $\mathfrak p$ above $p$ satisfies
$\chi(\mathfrak p) \neq 1$, then this equality is known to be a consequence of
the main conjecture for totally real fields proved by Wiles \cite{Wiles}. However, without any such restriction on the values $\chi(\mathfrak p)$, the equality (\ref{OrderIntro}) is as yet unproved.

In addition, the result of Corollary \ref{ThCharacterComponent02} is much stronger than (\ref{OrderIntro}) in that it shows the explicit structure of
$A(K)^{\chi}$ as a Galois module to be completely determined (conjecturally at least) by Stickelberger elements via the obvious (non-canonical) isomorphism of $\mathcal{O}_{\chi}$-modules
$$A(K)^{\chi} \simeq \bigoplus_{i \geq 1}
\Fitt_{\mathcal{O}_{\chi}}^{i}(A(K)^{\chi})/\Fitt_{\mathcal{O}_{\chi}}^{i-1}(A(K)^{\chi}).$$

Next we note that the proof of Corollary \ref{ThCharacterComponent02} also combines with the result of Theorem \ref{MT4} below to give the following result (which does not itself assume the validity of LTC$(K/k)$).

This result will be proved in Corollaries \ref{CorCharacterComponent3} and \ref{CorHigherFittCM2}. In it we write $\mu_{p^\infty}(k(\zeta_{p}))$ for the $p$-torsion subgroup of $k(\zeta_{p})^\times$.

\begin{corollary}
Assume that $K/k$ is a CM-extension, that the degree $[K:k]$ is prime to $p$, and that
$\chi$ is an odd character of $G$ such that
there is at most one $p$-adic place $\mathfrak{p}$ of $k$
with $\chi(\mathfrak p)=1$. Assume also that the $p$-adic $\mu$-invariant of $K_{\infty}/K$ vanishes.

Then one has both an equality
$$|A(K)^{\chi}| =
\left\{
\begin{array}{ll}
|\mathcal{O}_{\chi}/L_{k}(\chi^{-1},0)|  & \mbox{if} \ \chi \neq \omega, \\
|\mathcal{O}_{\chi}/(|\mu_{p^\infty}(k(\zeta_{p}))|\cdot L_{k}(\chi^{-1},0)) &
\mbox{if} \ \chi=\omega
\end{array}\right.
$$
and a (non-canonical) isomorphism of $\mathcal{O}_{\chi}$-modules

$$A(K)^{\chi} \simeq \bigoplus_{i \geq 1} \Theta^{i}(K/k)^{\chi}/\Theta^{i-1}(K/k)^{\chi}.$$

\end{corollary}

This result is a generalization of the main theorem of the second author in \cite{Ku3} where it is assumed that $\chi(\mathfrak p) \neq 1$ for all $p$-adic places $\mathfrak{p}$. It also generalizes the main result of Rubin in \cite{RubinGaussSum} which deals only with the special case $K=\QQ(\mu_{p})$ and $k=\QQ$.

To end this subsection we note that Remark \ref{Remark021} below explains why Theorem \ref{MT2}(ii) also generalizes and refines the main result of Cornacchia and Greither in \cite{CorGrei}.

\subsection{}In this subsection we discuss further connections between Rubin-Stark elements and the structure of class groups of the form ${\rm Cl}_{\Sigma}^T(K)$ for `small' sets $\Sigma$ which do not seem to follow from Theorem \ref{MT2}. In particular, we do not assume here that $\Sigma$ contains $S_{\rm ram}(K/k)$.


To do this we denote the annihilator ideal of a $G$-module $M$ by $\Ann_{G}(M)$.

\begin{theorem} \label{MC4}
Assume that ${\rm LTC}(K/k)$ is valid. Fix $\Phi$ in
$\bigwedge_{\ZZ[G]}^r \Hom_{\ZZ[G]}(\mathcal{O}_{K,S,T}^\times,\ZZ[G])$
and any place $v$ in $S \setminus V$.

Then one has
$$\Phi(\epsilon_{K/k,S,T}^V) \in \Ann_{G}({\rm Cl}_{V \cup \{v\}}^T(K))$$
and, if $G$ is cyclic, also
$$\Phi(\epsilon_{K/k,S,T}^V) \in \Fitt_{G}^{0}({\rm Cl}_{V \cup \{v\}}^T(K)).$$
\end{theorem}

\begin{remark} \label{Remark02}
\begin{rm}
The first assertion of Theorem \ref{MC4} provides a common refinement
and wide-ranging generalization (to $L$-series of arbitrary order of vanishing) of several well-known
conjectures and results in the literature. To discuss this we write ${\rm Cl}^T(K)$ for the full ray class group modulo $T$
(namely, ${\rm Cl}^T(K)={\rm Cl}_{\emptyset}^T(K)$, see \S\ref{notation section}).

(i) We first assume that $r=0$ (so $V$ is empty) and that $k$ is a number field.
Then, without loss of generality (for our purposes), we can assume that $k$ is totally real and $K$ is a CM field.
In this case $\epsilon_{K/k,S,T}^{\emptyset}$ is the Stickelberger element
$\theta_{K/k,S,T}(0)$ of the extension $K/k$
(see \S \ref{L-functionsubsection}).
We take $v$ to be an archimedean place in $S$.
Then ${\rm Cl}_{\{v\}}^T(K)={\rm Cl}^T(K)$ and so the first assertion of Theorem \ref{MC4} shows that
${\rm LTC}(K/k)$ implies the classical Brumer-Stark Conjecture,
$$\theta_{K/k,S,T}(0)\cdot{\rm Cl}^T(K)=0.$$

(ii) We next consider the case that $K$ is totally real and take $V$ to be $S_{\infty}(k)$ so that $r = |V| = [k : \QQ]$. In this case Corollary \ref{MC4} implies that for any non-archimedean place $v$ in $S$, any element $\sigma_{v}$ of the decomposition subgroup $G_v$ of $v$ in $G$ and any element $\Phi$ of
$\bigwedge_{\ZZ[G]}^{[k:\Q]}\Hom_{\ZZ[G]}(\mathcal{O}_{K,S,T}^\times,\ZZ[G])$,
one has
\begin{equation} \label{Remark022}
(1-\sigma_{v})\cdot\Phi(\epsilon_{K/k,S,T}^{S_{\infty}})
\in \Ann_{G}({\rm Cl}^T(K)).
\end{equation}

To make this containment even more explicit we further specialize to the case that $k = \QQ$
and that $K$ is equal to $\QQ(\zeta_m)^+$ for some natural number $m$. We recall that ${\rm LTC}(K/k)$ has been verified in this case.
We take $S$ to be the
set comprising the (unique) archimedean place $\infty$ and all prime divisors of $m$, and
$V$ to be $S_\infty = \{\infty\}$ (so $r = 1$).
For a set $T$ which contains an odd prime, we set $\delta_{T}:=\prod_{v \in T}(1-{\N}v {\rm Fr}_{v}^{-1})$, where ${\rm Fr}_v\in G$ denotes the Frobenius automorphism at a place of $K$ above $v$.
In this case, we have
$$\epsilon_{K/k,S,T}^{\{\infty\}}=\epsilon_{m,T}:=
(1-\zeta_m)^{\delta_T} \in \mathcal{O}_{K,S,T}^\times$$
(see, for example, \cite[p.79]{tate} or \cite[\S 4.2]{P}) and so (\ref{Remark022}) implies that for any $\sigma_{v} \in G_{v}$ and any
$\Phi \in \Hom_{\ZZ[G]}(\mathcal{O}_{K,S,T}^\times,\ZZ[G])$ one has
$$
(1-\sigma_{v})\cdot\Phi(\epsilon_{m,T})
\in \Ann_{G}({\rm Cl}^T(K)).
$$
Now the group $G$ is generated by the decomposition subgroups $G_{v}$ of each
prime divisor $v$ of $m$, and so for any $\sigma \in G$
one has an equality $\sigma-1=\Sigma_{v \mid m} x_{v}$ for suitable elements $x_{v}$ of the ideal $I({G_{v}})$ of $\ZZ[G]$ that is generated by
 $\{ \sigma_v-1: \sigma_v \in G_{v}\}$. Hence, since $\epsilon_{m,T}^{\sigma-1}$ belongs to $\mathcal{O}_{K}^{\times}$ one finds that
for any $\varphi \in \Hom_{\ZZ[G]}(\mathcal{O}_{K}^\times,\ZZ[G])$ one has
$\varphi(\epsilon_{m,T}^{\sigma-1})=\Sigma_{v \mid m} x_{v} \widetilde \varphi(\epsilon_{m,T})$
where $\widetilde \varphi$ is any lift of $\varphi$ to
$\Hom_{\ZZ[G]}(\mathcal{O}_{K,S,T}^\times,\ZZ[G])$.
Therefore, for any $\varphi$ in $\Hom_{\ZZ[G]}(\mathcal{O}_{K}^\times,\ZZ[G])$
and any $\sigma$ in $G$, one has
\begin{equation} \label{RubinInv1987}
\varphi(\epsilon_{m,T}^{\sigma-1}) \in \Ann_{G}({\rm Cl}^T(K)).
\end{equation}
This containment is actually finer than the annihilation result proved by Rubin in
\cite[Th. 2.2 and the following Remark]{rubininv} since it deals with the group ${\rm Cl}^T(K)$ rather than ${\rm Cl}(K)$.
\end{rm}
\end{remark}

\begin{remark} \label{Remark03}
\begin{rm} We next consider the case that $K/k$ is a cyclic CM-extension and $V$ is empty. As remarked above, in this case
the Rubin-Stark element $\epsilon_{K/k,S,T}^{\emptyset}$ coincides with the Stickelberger element
$\theta_{K/k,S,T}(0)$.

The second assertion of Theorem \ref{MC4} therefore combines with the argument in Remark \ref{Remark02}(i) to show that  ${\rm LTC}(K/k)$ implies a containment
$$\theta_{K/k,S,T}(0) \in \Fitt_{G}^0({\rm Cl}^T(K)).$$
This is a strong refinement of the Brumer-Stark conjecture. To see this note that
${\rm Cl}^{T}(K)$ is equal to the ideal class group ${\rm Cl}(K)$ of $K$ when $T$ is empty. The above containment thus combines with \cite[Chap. IV, Lem. 1.1]{tate} to imply that if $G$ is cyclic, then one has
$$\theta_{K/k,S}(0)\cdot\Ann_{G}(\mu(K)) \subset \Fitt_{G}^0({\rm Cl}(K))$$
where $\mu(K)$ denotes the group of roots of unity in $K$. It is known that this inclusion is not in general valid without the hypothesis that $G$ is cyclic (see \cite{GreiKuri}). The possibility of such an explicit refinement of Brumer's Conjecture
was discussed by the second author in \cite{Ku8} and \cite{Ku9M}. In fact, in the terminology of \cite{Ku8}, the above argument shows that both properties (SB) and (DSB) follow from ${\rm LTC}(K/k)$ whenever $G$ is cyclic. For further results in the case that $G$ is cyclic see Corollary \ref{MC11}.
\end{rm}
\end{remark}

\begin{remark} \label{Remark021}
\begin{rm} Following the discussion of Remark \ref{Remark02}(ii) we can also now consider Theorem \ref{MT2} further in the case that $k = \QQ$ and $K=\QQ(\zeta_{p^{n}})^+$ for an odd prime $p$ and natural number $n$. Setting $S := \{\infty, p\}$ one has $\Cl^T_S(K) = \Cl^T(K)$ and the $G$-module $X_{K,S}$ is free of rank one and so the exact sequence (\ref{dual ses}) combines with the final assertion of Theorem \ref{MT2}(ii) (with $r=1$) to give equalities

\begin{align*}
\Fit_G^0(\Cl^T(K)) &=
\Fit_G^1(\mathcal{S}^{{\rm tr}}_{S,T}(\GG_{m/K}))
\\
&=
\{\Phi(\epsilon_{p^n,T}) \mid \Phi \in \Hom_{G}
(\mathcal{O}_{K,S,T}^\times,\ZZ[G])\}\\
&= \Fitt^{0}_{G}(\mathcal{O}_{K,S,T}^\times/(\ZZ[G]\cdot \epsilon_{p^n,T}))
\end{align*}
where the last equality follows from the fact that $G$ is cyclic. By a standard argument, this equality implies that
$\Fit_G^0(\Cl(K))=\Fitt^{0}_{G}(\mathcal{O}_{K}^\times/\mathcal{C}_{K})$ with $\mathcal{C}_{K}$ denoting the group $\ZZ[G]\cdot\{1-\zeta_{p^{n}}, \zeta_{p^{n}}\}\cap \mathcal{O}_{K}^\times$
 of cyclotomic units of $K$, and this is the main result of Cornacchia and Greither in \cite{CorGrei}.
\end{rm}
\end{remark}

For any finite group $\Gamma$ and any $\Gamma$-module $M$ we write $M^\vee$ for its
Pontryagin dual $\Hom_\ZZ(M,\QQ/\ZZ)$, endowed with the natural contragredient action of $\Gamma$.  

In \S\ref{ProofofRemark} we show that the proof of Theorem \ref{MC4} also implies the following result.

In this result we fix an odd prime $p$ and set ${\rm Cl}^T(K)^\vee_p :=
{\rm Cl}^T(K)^\vee\otimes\ZZ_p$.

\begin{corollary} \label{CorRemark03}
Let $K/k$ be any finite abelian CM-extension and $p$ any odd prime. If ${\rm LTC}(K/k)$ is valid, then one has a containment
$$\theta_{K/k,S,T}(0)^{\#} \in \Fitt_{\ZZ_p[G]}^0({\rm Cl}^T(K)^\vee_p).$$
\end{corollary}

\begin{remark}\begin{rm}\

(i) In \cite{GreitherPopescu} Greither and Popescu have proved the containment in Corollary \ref{CorRemark03} in the special case that $S$ contains all places above $p$ and the $p$-adic $\mu$-invariant of $K$ vanishes. By combining Corollary \ref{CorRemark03} with the result of Corollary \ref{IntroConsequencesOfVentullo} below one can now
prove this containment in cases that are not considered in \cite{GreitherPopescu}. For details see \cite{aze2}.

(ii) For any odd prime $p$ the group ${\rm Cl}(K)^\vee_p := {\rm Cl}(K)^\vee \otimes \ZZ_{p}$ is not a quotient of
${\rm Cl}^T(K)^\vee_p$ and so Corollary \ref{CorRemark03} does not imply $\theta_{K/k,S,T}(0)$ belongs to $\Fitt_{\ZZ_p[G]}^0({\rm Cl}(K)_p^\vee)$.

(iii) For any odd prime $p$ write ${\rm Cl}(K)_p^{\vee,-}$ for the submodule of ${\rm Cl}(K)\otimes \ZZ_p$ upon which complex conjugation acts as multiplication by $-1$. Then, under a certain technical hypothesis on $\mu(K)$, the main result of Greither in \cite{CG} shows that LTC$(K/k)$ also implies an explicit description of the Fitting ideal $\Fitt_{\ZZ_p[G]}^0({\rm Cl}(K)_p^{\vee,-})$ in terms of suitably normalized Stickelberger elements. By replacing the role of `Tate sequences for small $S$' in the argument of Greither by the `$T$-modified Weil-\'etale cohomology' complexes that we introduce in \S\ref{wec} one can in fact prove the same sort of result without any hypothesis on $\mu(K)$ and with ${\rm Cl}(K)$ replaced by the finer group ${\rm Cl}^T(K)$.
\end{rm}\end{remark}

\subsection{}We believe that, in certain cases, the conjecture ${\rm MRS}(K/L/k,S,T)$ could be easier to study than
 LTC$(K/k)$. It is therefore natural to investigate the converse implication to that given in Theorem \ref{MT1}.

The last of our main results that we record here both provides a converse to Theorem \ref{MT1} in a special case and also expresses an interesting connection between the conjectures LTC$(K/k)$ and ${\rm MRS}(K/L/k,S,T)$ and a well-known leading term formula for $p$-adic $L$-series that has been conjectured by Gross.

To state this we assume that
$k$ is totally real, that $K$ is a CM-field and that $p$ is an odd prime.
We consider the conjectures and the properties over
the minus part $\ZZ_{p}[G]^-$ of $\ZZ_p[G]$. We define an $\mathcal{L}$-invariant for arbitrary rank, and then
formulate in Conjecture \ref{GrossStark}
the Gross-Stark conjecture on the leading term at $s=0$ of a $p$-adic
$L$-function by using this $\mathcal{L}$-invariant.

In the following result we write $K_{\infty}$ for the cyclotomic $\ZZ_{p}$-extension of $K$ and $K_{n}$ for the the $n$-th
layer of this extension.

\begin{theorem} \label{MT4}
We assume that for any odd finite order character $\chi$ of
$\Gal(K_{\infty}/k)$ the $p$-adic $L$-function of $\chi$ has the order of vanishing at $s=0$ conjectured by Gross. We also assume that the $p$-adic $\mu$-invariant of $K_{\infty}/K$ vanishes.

Then the following assertions are equivalent.
\begin{itemize}
\item[(i)]{The minus part of the $p$-part of ${\rm LTC}(K_{n}/k)$ is valid for all non-negative integers $n$.}
\item[(ii)]{The minus part of the $p$-part
of ${\rm MRS}(F_{n}/F/k,S,T,\emptyset,V')$ is valid for all intermediate CM-fields $F$ of $K_{\infty}/k$ with $[F:k]<\infty$
and for all non-negative integers $n$, where $F_{n}$ is the $n$-th layer of the cyclotomic $\ZZ_{p}$-extension of $F$
and $V'$ consists of all completely splitting places in $F$ of $S$.}
\item[(iii)]{The leading term part
of the Gross-Stark conjecture (see Conjecture \ref{GrossStark}) is valid for all
intermediate CM-fields $F$ of $K_{\infty}/k$ with $[F:k]<\infty$.}
\end{itemize}
\end{theorem}


By combining this result with an argument of Gross (involving Brumer's $p$-adic version of Baker's Theorem) and important recent work of Darmon, Dasgupta and Pollack \cite{DDP}
and of Ventullo \cite{ventullo} we shall then obtain the following new evidence in support of the conjectures ${\rm LTC}(K/k)$ and ${\rm MRS}(K/L/k,S,T)$.

\begin{corollary} \label{IntroConsequencesOfVentullo}
If the $p$-adic Iwasawa $\mu$-invariant of $K$ vanishes and at most
one $p$-adic place splits in $K/K^{+}$, then the minus parts of the
$p$-parts of both ${\rm LTC}(K_{n}/k)$ and
${\rm MRS}(K_{n}/K/k,S,T)$ are valid for all non-negative integers $n$.
\end{corollary}

In a sequel to this article we shall make a much fuller investigation of the Iwasawa-theoretic aspects of our approach (see \cite{aze2}). By doing so we shall show, in particular, that, without any restriction to CM extensions (or to the `minus parts' of conjectures), under the assumed validity of a natural main conjecture of Iwasawa theory,  the validity of
the $p$-part of ${\rm MRS}(L/K/k,S,T)$ for all finite abelian extensions $L/k$ implies the $p$-part of ${\rm LTC}(L/k)$. Such a result provides both an important generalization of Theorem \ref{MT4} and a very useful partial converse to the result of Theorem \ref{MT1}.

\subsection{}\label{notation section} In this final subsection of the Introduction we collect together some important notation which will be used in the article.

For an abelian group $G$, a $\ZZ[G]$-module is simply called a $G$-module. Tensor products, Hom, exterior powers, and determinant modules over $\ZZ[G]$ are denoted by $\otimes_G$, $\Hom_G$, $\bigwedge_G$, and ${\det}_G$, respectively. We use similar notation for Ext-groups, Fitting ideals, etc. The augmentation ideal of $\ZZ[G]$ is denoted by $I(G)$. For any $G$-module $M$ and any subgroup $H \subset G$, we denote $M^H$ for the submodule of $M$ comprising elements fixed by $H$.
The norm element of $H$ is denoted by $\N_H$, namely,
$$\N_H=\sum_{\sigma \in H}\sigma \in \ZZ[G].$$

Let $E$ denote either $\QQ$, $\RR$ or $\CC$. For an abelian group $A$, we denote $E\otimes_\ZZ A$ by $EA$. The maximal $\ZZ$-torsion subgroup of $A$ is denoted by $A_{\rm tors}$. We write $A/A_{\rm tors}$ by $A_{\rm tf}$. The Pontryagin dual $\Hom_\ZZ(A,\QQ/\ZZ)$ of $A$ is
denoted by $A^\vee$ for discrete $A$.

Fix an algebraic closure $\overline \QQ$ of $\QQ$. For a positive integer $n$, we denote by $\mu_n$ the group of $n$-th roots of unity in $\overline \QQ ^\times$.

Let $k$ be a global field. The set of all infinite places of $k$ is denoted by $S_\infty(k)$ or simply by $S_\infty$ when $k$ is clear from the context. (If $k$ is a function field, then $S_\infty(k)$ is empty.) Consider a finite Galois extension $K/k$, and denote its Galois group by $G$. The set of all places of $k$ which ramify in $K$ is denoted by $S_{\rm ram}(K/k)$ or simply by $S_{\rm ram}$ when $K/k$ is clear from the context. For any non-empty finite set $S$ of places of $k$, we denote by $S_K$ the set of places of $K$ lying above places in $S$. The ring of $S$-integers of $K$ is defined by
$$\mathcal{O}_{K,S}:=\{ a\in K : \ord_w(a) \geq 0 \text{ for all finite places $w$ of $K$ not contained in $S_K$} \},$$
where $\ord_w$ denotes the normalized additive valuation at $w$.
The unit group of $\mathcal{O}_{K,S}$ is called the $S$-unit group of $K$.
Let $T$ be a finite set of finite places of $k$, which is disjoint from $S$. The $(S,T)$-unit group of $K$ is defined by
$$\mathcal{O}_{K,S,T}^\times:=\{ a\in \mathcal{O}_{K,S}^\times : a \equiv 1 \text{ (mod }w) \text{ for all $w\in T_K$} \}.$$
The ideal class group of $\mathcal{O}_{K,S}$ is denoted by ${\rm Cl}_S(K)$. This is called the $S$-class group of $K$.
The $(S,T)$-class group of $K$, which we denote by ${\rm Cl}_S^T(K)$, is defined to be the ray class group of $\mathcal{O}_{K,S}$ modulo $\prod_{w\in T_K}w$
(namely, the quotient of the group of fractional ideals
whose supports are coprime to all places above $S\cup T$ by the subgroup
of principal ideals with a generator congruent to $1$ modulo all
places in $T_{K}$).
When $S$ is in $S_{\infty}$, we omit $S$ and write ${\rm Cl}^T(K)$ for
${\rm Cl}_S^T(K)$.
When $S \subset S_{\infty}$ and $T=\emptyset$, we write ${\rm Cl}(K)$
which is the class group of the integer ring $\mathcal{O}_{K}$.

We denote by $X_{K,S}$ the augmentation kernel of the divisor group $Y_{K,S}:=\bigoplus_{w\in S_K}\ZZ w$. If $S$ contains $S_\infty(k)$, then the Dirichlet regulator map
$$\lambda_{K,S} : \RR \mathcal{O}_{K,S}^\times \longrightarrow \RR X_{K,S},$$
defined by $\lambda_{K,S}(a):=-\sum_{w \in S_K}\log |a|_w w$, is an isomorphism.

For a place $w$ of $K$, the decomposition subgroup of $w$ in $G$ is denoted by $G_w$. If $w$ is finite, the residue field of $w$ is denoted by $\kappa(w)$. The cardinality of $\kappa(w)$ is denoted by ${\N} w$. If the place $v$ of $k$ lying under $w$ is unramified in $K$, then the Frobenius automorphism at $w$ is denoted by ${\rm Fr}_{w} \in G_w$. When $G$ is abelian, then $G_w$ and ${\rm Fr}_w$ depend only on $v$, so in this case we often denote them by $G_v$ and ${\rm Fr}_v$ respectively.
The $\CC$-linear involution $\CC[G] \rightarrow \CC[G]$ induced by
$\sigma \mapsto \sigma^{-1}$ with $\sigma \in G$ is denoted by $x \mapsto x^{\#}$.

A complex of $G$-modules is said to be `perfect' if it is quasi-isomorphic to a bounded complex of finitely generated projective $G$-modules.

We denote by $D(\ZZ[G])$ the derived category of $G$-modules, and by $D^{\rm p}(\ZZ[G])$ the full subcategory of $D(\ZZ[G])$ consisting of perfect complexes.

\section{Canonical Selmer groups and complexes for $\mathbb{G}_m$} \label{ccsm}

In this section, we give a definition of `integral dual Selmer groups for $\GG_m$', as analogues of Mazur-Tate's `integral Selmer groups' defined for abelian varieties in \cite{mt}. We shall also review the construction of certain natural arithmetic complexes, which are used for the formulation of the leading term conjecture.

\subsection{Integral dual Selmer groups}\label{ism}
Let $K/k$ be a finite Galois extension of global fields with Galois group $G$. Let $S$ be a non-empty finite set of places which contains $S_\infty(k)$. Let $T$ be a finite set of places of $k$ which is disjoint from $S$.

\begin{definition}
\begin{rm}
We define the `($S$-relative $T$-trivialized) integral dual Selmer group for $\GG_m$' by setting
\begin{eqnarray}
\mathcal{S}_{S,T}(\GG_{m/K}) :=
\coker(\prod_{w \notin S_K\cup T_K}\ZZ \longrightarrow \Hom_\ZZ(K_T^\times,\ZZ) ), \nonumber
\end{eqnarray}
where $K_T^\times$ is the subgroup of $K^\times$ defined by
$$K_T^\times:=\{ a\in K^\times : \ord_w(a-1)>0 \text{ for all } w\in T_K \},$$
and the homomorphism on the right hand side is defined by
$$ (x_w)_w \mapsto (a \mapsto \sum_{w\notin S_K\cup T_K}\ord_w(a)x_w).$$
When $T$ is empty, we omit the subscript $T$ from this notation.
\end{rm}
\end{definition}

By the following proposition we see that our integral dual Selmer groups are actually like usual dual Selmer groups (see also Remark \ref{notation2} below).

\begin{proposition}
There is a canonical exact sequence
\begin{eqnarray*}
0 \longrightarrow {\rm Cl}_S^T(K)^\vee \longrightarrow \mathcal{S}_{S,T}(\GG_{m/K}) \longrightarrow \Hom_\ZZ(\mathcal{O}_{K,S,T}^\times, \ZZ) \longrightarrow 0\end{eqnarray*}
of the form (\ref{selseq}).
\end{proposition}

\begin{proof}
Consider the commutative diagram
$$
\xymatrix{
0 \ar[r] & \prod_{w \notin S_K\cup T_K} \ZZ \ar[r] \ar[d] & \prod_{w \notin S_K\cup T_K}\QQ \ar[r] \ar[d] & \prod_{w \notin S_K\cup T_K}\QQ/\ZZ \ar[r] \ar[d] & 0 \\
0 \ar[r] &\Hom_\ZZ(K_T^\times,\ZZ) \ar[r]& \Hom_\ZZ(K_T^\times,\QQ) \ar[r]& (K_T^\times)^\vee,  & \\
}
$$
where each row is the natural exact sequence, and each vertical arrow is given by $(x_w)_w \mapsto (a \mapsto \sum_{w\notin S_K\cup T_K}\ord_w(a)x_w)$. Using the exact sequence
$$0\longrightarrow \mathcal{O}_{K,S,T}^\times \longrightarrow K_T^\times \stackrel{\bigoplus \ord_w}{\longrightarrow} \bigoplus_{w\notin S_K\cup T_K}\ZZ \longrightarrow {\rm Cl}_S^T(K) \longrightarrow 0$$
and applying the snake lemma to the above commutative diagram, we obtain the exact sequence
$$0\longrightarrow {\rm Cl}_S^T(K)^\vee \longrightarrow \mathcal{S}_{S,T}(\GG_{m/K}) \longrightarrow \Hom_\ZZ(\mathcal{O}_{K,S,T}^\times,\QQ) \longrightarrow (\mathcal{O}_{K,S,T}^\times)^\vee.$$
Since the kernel of the last map is $\Hom_\ZZ(\mathcal{O}_{K,S,T}^\times, \ZZ)$, we obtain the desired conclusion.
\end{proof}

\begin{remark}\label{notation2}\

(i) The Bloch-Kato Selmer group $H^1_f(K, \QQ/\ZZ(1))$ is defined to be the kernel of the diagonal map
\[ K^{\times} \otimes \QQ/\ZZ \longrightarrow \bigoplus_{w} K_{w}^{\times}/\mathcal{O}_{K_{w}}^{\times}
\otimes \QQ/\ZZ\]
where $w$ runs over all finite places, and so lies in a canonical exact sequence
$$0 \longrightarrow \mathcal{O}_{K}^\times  \otimes \QQ/\ZZ \longrightarrow H^1_f(K, \QQ/\ZZ(1)) \longrightarrow
{\rm Cl}(K) \longrightarrow 0.$$
Its Pontryagin dual $H^1_f(K, \QQ/\ZZ(1))^{\vee}$ is a finitely generated $\hat{\ZZ}$-module and our integral dual Selmer group
$\mathcal{S}_{S_{\infty}}(\GG_{m/K})$ provides a canonical finitely generated $\ZZ$-structure on
$H^1_f(K, \QQ/\ZZ(1))^{\vee}$.

(ii) In general, the exact sequence (\ref{selseq}) also means that $\mathcal{S}_{S,T}(\GG_{m/K})$ is a natural analogue (relative to $S$ and $T$) for $\mathbb{G}_m$ over $K$ of the `integral Selmer group' that is defined for abelian varieties by Mazur and Tate in \cite[p.720]{mt}.

%
%
\end{remark}

In the next subsection we shall give a natural cohomological interpretation of the group $\mathcal{S}_{S,T}(\GG_{m/K})$ (see Proposition \ref{new one}(iii)) and also show that it has a canonical `transpose' (see Definition \ref{defseltr}). 

\subsection{`Weil-\'etale cohomology' complexes}\label{wec} In the following,
we construct two canonical complexes of $G$-modules, and use them to show that there exists a canonical transpose of the module $\mathcal{S}_{S,T}(\GG_{m/K})$. The motivation for our choice of notation (and terminology) for these complexes is explained in Remark \ref{notation} below.

We fix data $K/k,G,S,T$ as in the previous subsection.
We also write ${\mathbb F}_{T_K}^\times$
for the direct sum $\bigoplus_{w \in T_{K}} \kappa(w)^{\times}$ of the multiplicative groups of the residue fields of all places in $T_K$.

\begin{proposition}\label{new one} There exist canonical complexes of $G$-modules $R\Gamma_{\!c}((\mathcal{O}_{\!K,S})_{\mathcal{W}},\ZZ)$ and $R\Gamma_{\!c,T}((\mathcal{O}_{\!K,S})_{\mathcal{W}},\ZZ)$ which satisfy all of the following conditions.

\begin{itemize}
\item[(i)] There exists a  canonical commutative diagram of exact triangles in $D(\ZZ[G])$ 

\begin{equation}\label{diagram}\minCDarrowwidth1em\begin{CD}
\Hom_\ZZ(\mathcal{O}_{K,S}^\times,\QQ)[-3] @>\theta >> R\Gamma_{\!c}(\mathcal{O}_{K,S},\ZZ) @> >>  R\Gamma_{\!c}((\mathcal{O}_{\!K,S})_{\mathcal{W}},\ZZ)@> >> \\
@V VV @\vert @V VV \\
(\Hom_\ZZ(\mathcal{O}_{K,S}^\times,\QQ)\oplus (\mathbb{F}_{T_K}^\times)^\vee)[-3] @>\theta' >> R\Gamma_{\!c}(\mathcal{O}_{K,S},\ZZ) @> >> R\Gamma_{\!c,T}((\mathcal{O}_{\!K,S})_{\mathcal{W}},\ZZ)@> >> \\
@V VV @. @V VV\\
(\mathbb{F}^\times_{T_K})^\vee[-3]  @. @. (\mathbb{F}^\times_{T_K})^\vee[-2]\\
@V VV @. @VV\theta'' V\end{CD}\end{equation}
in which the first column is induced by the obvious exact sequence
\[ 0 \longrightarrow \Hom_\ZZ(\mathcal{O}_{K,S}^\times,\QQ) \longrightarrow \Hom_\ZZ(\mathcal{O}_{K,S}^\times,\QQ)\oplus (\mathbb{F}^\times_{T_K})^\vee \longrightarrow (\mathbb{F}^\times_{T_K})^\vee \longrightarrow 0\]
and $H^2(\theta'')$ is the Pontryagin dual of the natural injective homomorphism
\[ H^3(R\Gamma_{\!c}((\mathcal{O}_{\!K,S})_{\mathcal{W}},\ZZ))^\vee = \mathcal{O}_{K,{\rm tors}}^\times \longrightarrow \mathbb{F}^\times_{T_K}. \]
\item[(ii)] If $S'$ is a set of places of $k$ which contains $S$ and is disjoint from $T$, then there is a canonical exact triangle in $D(\ZZ[G])$
$$R\Gamma_{\!c,T}((\mathcal{O}_{\!K,S'})_{\mathcal{W}},\ZZ) \longrightarrow R\Gamma_{\!c,T}((\mathcal{O}_{\!K,S})_{\mathcal{W}},\ZZ) \longrightarrow \bigoplus_{w \in S'_K\setminus S_K} R\Gamma((\kappa(w))_\mathcal{W},\ZZ),$$
where $R\Gamma((\kappa(w))_\mathcal{W},\ZZ)$ is the complex of $G_w$-modules which lies in the exact triangle
$$\QQ[-2] \longrightarrow R\Gamma(\kappa(w),\ZZ) \longrightarrow R\Gamma((\kappa(w))_\mathcal{W},\ZZ) \longrightarrow,$$
where the $H^2$ of the first arrow is the natural map
$$\QQ \longrightarrow \QQ/\ZZ=H^2(\kappa(w),\ZZ).$$

\item[(iii)] The complex $R\Gamma_{\!\!c,T}((\mathcal{O}_{K,S})_{\mathcal{W}},\ZZ)$ is acyclic outside degrees one, two and three, and there are canonical isomorphisms
$$H^i(R\Gamma_{\!\!c,T}((\mathcal{O}_{K,S})_{\mathcal{W}},\ZZ))\simeq
\begin{cases}
Y_{K,S}/\Delta_S(\ZZ) &\text{ if }i=1,\\
 \mathcal{S}_{S,T}(\GG_{m/K}) &\text{ if }i=2,\\
(K_{T, {\rm tors}}^\times)^\vee &\text{ if } i=3,
\end{cases}$$
where $\Delta_S$ is the natural diagonal map.
\item[(iv)] If $S$ contains $S_{\rm ram}(K/k)$, then $R\Gamma_{\!c}((\mathcal{O}_{\!K,S})_{\mathcal{W}},\ZZ)$ and $R\Gamma_{\!c,T}((\mathcal{O}_{\!K,S})_{\mathcal{W}},\ZZ)$ are both perfect complexes of $G$-modules.
\end{itemize}
\end{proposition}

\begin{proof} In this argument we use the fact that morphisms in $D(\ZZ[G])$ between bounded above complexes $K_1^\bullet$ and $K_2^\bullet$ can be computed by means of the spectral sequence
\begin{equation}\label{ss} E_2^{p,q} = \prod_{a\in\ZZ}
{\rm Ext}^p_{G}(H^a(K_1^\bullet),H^{q+a}(K_2^\bullet)) \Rightarrow H^{p+q}(R\Hom _G(K_1^\bullet,K_2^\bullet))\end{equation}
constructed by Verdier in \cite[III, 4.6.10]{verdier}.

Set $C^\bullet =C^\bullet_S:= R\Gamma_{c}(\mathcal{O}_{K,S},\ZZ)$ and $W := \Hom_\ZZ(\mathcal{O}_{K,S}^\times,\QQ)$ for simplicity.
 Then we recall first that $C^\bullet$ is acyclic outside degrees one, two and three, that there are canonical isomorphisms
\begin{equation}\label{isos} H^i(C^\bullet) \cong \begin{cases} Y_{K,S}/\Delta_S(\ZZ) &\text{ if $i =1$,}\\
{\rm Cl}_S(K)^\vee &\text{ if $i=2$,}\\
(\mathcal{O}_{K,S}^\times)^\vee &\text{ if $i=3$,}\end{cases}\end{equation}
where $\Delta_S$ is the map that occurs in the statement of claim (iii) and that, when $S$ contains $S_{\rm ram}(K/k)$, $C^\bullet$ is isomorphic to a bounded complex of cohomologically-trivial $G$-modules.

It is not difficult to see that the groups $\Ext^i_G(W,H^{3-i}(C^\bullet))$ vanish for all $i > 0$, and so the spectral sequence (\ref{ss}) implies that the `passage to cohomology' homomorphism
\[ H^0(R\Hom _G(W[-3],C^\bullet)) =
\Hom _{D(\ZZ[G])}(W[-3],C^\bullet) \longrightarrow \Hom_{G}(W,(\mathcal{O}_{K,S}^\times)^\vee)\]
is bijective. We may therefore define $\theta$ to be the unique morphism in $D(\ZZ[G])$ for which $H^3(\theta)$ is equal to the natural map
\[ W = \Hom_\ZZ(\mathcal{O}_{K,S}^\times,\QQ) \longrightarrow \Hom_\ZZ(\mathcal{O}_{K,S}^\times,\QQ/\ZZ) = (\mathcal{O}_{K,S}^\times)^\vee\]
and then take $C_\mathcal{W}^\bullet:= R\Gamma_{\!c}((\mathcal{O}_{\!K,S})_{\mathcal{W}},\ZZ)$ to be any complex which lies in an exact triangle of the form that occurs in the upper row of (\ref{diagram}). An analysis of the long exact cohomology sequence of this triangle then shows that $C_\mathcal{W}^\bullet$ is acyclic outside degrees one, two and three, that $H^1(C_\mathcal{W}^\bullet) = H^1(C^\bullet)$, that $H^2(C_\mathcal{W}^\bullet)_{\rm tors} = H^2(C^\bullet)$, that $H^2(C_\mathcal{W}^\bullet)_{\rm tf} = \Hom_\ZZ(\mathcal{O}_{K,S}^\times,\ZZ)$ and that $H^3(C_\mathcal{W}^\bullet) = (\mathcal{O}_{K,{\rm tors}}^\times)^\vee$. In particular, when $S$ contains $S_{\rm ram}(K/k)$, since each of these groups is finitely generated and both of the complexes $W[-3]$ and $C^\bullet$ are represented by bounded complexes of cohomologically-trivial $G$-modules, this implies that $C^\bullet_\mathcal{W}$ is perfect.

To define the morphism $\theta'$ we first choose a finite set $S''$ of places of $k$ which is disjoint from
$S\cup T$ and such that ${\rm Cl}_{S'}(K)$ vanishes for $S':= S\cup S''$. Note that (\ref{isos}) with $S$ replaced by $S'$ implies $C_{S'}^\bullet$ is acyclic outside degrees one and three. We also note that, since each place in $T$ is unramified in $K/k$, there is also an exact sequence of $G$-modules
\begin{equation}\label{T seq} 0 \longrightarrow \bigoplus_{v\in T}\ZZ [G]
\stackrel{(1-{\N} v{\rm Fr}_{w})_v}{\longrightarrow} \bigoplus_{v\in T}\ZZ
[G] \longrightarrow ({\mathbb F}_{T_K}^\times)^\vee \longrightarrow 0\end{equation}
where $w$ is any choice of place of $K$ above $v$. This sequence shows both that $({\mathbb F}_{T_K}^\times)^\vee[-3]$ is a perfect complex of $G$-modules and also that the functor $\Ext^i_G(({\mathbb F}_{T_K}^\times)^\vee,-)$ vanishes for all $i > 1$. In particular, the spectral sequence (\ref{ss}) implies that in this case the passage to cohomology homomorphism
\[ \Hom _{D(\ZZ[G])}(({\mathbb F}_{T_K}^\times)^\vee[-3],C_{S'}^\bullet) \longrightarrow \Hom_{G}(({\mathbb F}_{T_K}^\times)^\vee,(\mathcal{O}_{K,S'}^\times)^\vee)\]
is bijective. We may therefore define $\theta'$ to be the morphism which restricts on $W[-3]$ to give $\theta$ and on
 $({\mathbb F}_{T_K}^\times)^\vee[-3]$ to give the composite morphism
 \[ ({\mathbb F}_{T_K}^\times)^\vee[-3]\stackrel{\theta_1'}{\longrightarrow} R\Gamma_{\!\!c}(\mathcal{O}_{K,S'},\ZZ) \stackrel{\theta_2'}{\longrightarrow}   R\Gamma_{\!\!c}(\mathcal{O}_{K,S},\ZZ)\]
 where $\theta_1'$ is the unique morphism for which $H^3(\theta_1')$ is the Pontryagin dual of the natural map $\mathcal{O}_{K,S'}^\times \rightarrow {\mathbb F}_{T_K}^\times$ and $\theta_2'$ occurs in the canonical exact triangle
 \begin{equation}\label{milne tri} R\Gamma_{\!\!c}(\mathcal{O}_{K,S'},\ZZ) \stackrel{\theta_2'}{\longrightarrow} R\Gamma_{\!\!c}(\mathcal{O}_{K,S},\ZZ) \longrightarrow \bigoplus_{w \in S''_K}R\Gamma(\kappa(w),\ZZ) \longrightarrow \end{equation}
constructed by Milne in \cite[Chap. II, Prop. 2.3 (d)]{milne}.

We now take $C_{\mathcal{W},T}^\bullet:= R\Gamma_{\!c,T}((\mathcal{O}_{\!K,S})_{\mathcal{W}},\ZZ)$ to be any complex which lies in an exact triangle of the form that occurs in the second row of (\ref{diagram}) and then, just as above, an analysis of this triangle shows that $C_{\mathcal{W},T}^\bullet$ is a perfect complex of $G$-modules when $S$ contains $S_{\rm ram}(K/k)$. Note also that since for this choice of $\theta'$ the upper left hand square of (\ref{diagram}) commutes the diagram can then be completed to give the right hand vertical exact triangle. The claim (ii) follows easily from the above  constructions.

It only remains to prove claim (iii). It is easy to see that the groups $H^i(R\Gamma_{\!\!c,T}((\mathcal{O}_{K,S})_{\mathcal{W}},\ZZ))$ for $i=1$ and $3$ are as described in claim (iii), so we need only prove that there is a natural isomorphism
$$H^2(R\Gamma_{\!\!c,T}((\mathcal{O}_{K,S})_{\mathcal{W}},\ZZ))\simeq \mathcal{S}_{S,T}(\GG_{m/K}).$$
To do this we first apply claim (ii) for a set $S'$ that is large enough to ensure that ${\rm Cl}_{S'}^T(K)$ vanishes. Since in this case $$H^2(R\Gamma_{\!\!c,T}((\mathcal{O}_{K,S'})_{\mathcal{W}},\ZZ))=\Hom_\ZZ(\mathcal{O}_{K,S',T}^\times,\ZZ),$$
we obtain in this way a canonical isomorphism
\begin{eqnarray}
H^2(R\Gamma_{\!\!c,T}((\mathcal{O}_{K,S})_{\mathcal{W}},\ZZ))\simeq \coker(\bigoplus_{w \in S'_K\setminus S_K}\ZZ \longrightarrow \Hom_\ZZ(\mathcal{O}_{K,S',T}^\times,\ZZ)). \label{selisom1}
\end{eqnarray}
Consider next the commutative diagram
$$
\xymatrix{
0 \ar[r] & \prod_{w \notin S_K'\cup T_K} \ZZ \ar[r] \ar@{=}[d] & \prod_{w \notin S_K\cup T_K}\ZZ \ar[r] \ar[d] & \bigoplus_{w \in S_K'\setminus S_K}\ZZ \ar[r] \ar[d] & 0 \\
0 \ar[r] &\prod_{w \notin S_K'\cup T_K}\ZZ \ar[r]& \Hom_\ZZ(K_T^\times,\ZZ) \ar[r]& \Hom_\ZZ(\mathcal{O}_{K,S',T}^\times,\ZZ) \ar[r] & 0 \\
}
$$
with exact rows, where the first exact row is the obvious one, the second is the dual of the exact sequence
$$0 \longrightarrow \mathcal{O}_{K,S',T}^\times \longrightarrow K_T^\times \stackrel{\bigoplus\ord_w}{\longrightarrow} \bigoplus_{w \notin S_K'\cup T_K}\ZZ \longrightarrow 0,$$
and the vertical arrows are given by $(x_w)_w \mapsto (a \mapsto \sum_{w}\ord_w(a)x_w)$. From this we have the canonical isomorphism
\begin{eqnarray}
 \mathcal{S}_{S,T}(\GG_{m/K})\simeq\coker(\bigoplus_{w \in S'_K\setminus S_K}\ZZ \longrightarrow \Hom_\ZZ(\mathcal{O}_{K,S',T}^\times,\ZZ)). \label{selisom2}
 \end{eqnarray}
From (\ref{selisom1}) and (\ref{selisom2}) our claim follows.
%
%
\end{proof}

Given the constructions in Proposition \ref{new one}, in each degree $i$ we set
%
%
%
\[ H^i_{\!c,T}((\mathcal{O}_{\!K,S})_{\mathcal{W}},\ZZ) := H^i(R\Gamma_{\!c,T}((\mathcal{O}_{\!K,S})_{\mathcal{W}},\ZZ)).\]

We also define 
a complex
\[ R\Gamma_T((\mathcal{O}_{K,S})_{\mathcal{W}},\GG_m):=R\Hom_\ZZ(R\Gamma_{c,T}((\mathcal{O}_{K,S})_{\mathcal{W}}, \ZZ),\ZZ)[-2].\]
We endow this complex with the natural contragredient action of $G$ and then in each degree $i$ set
$$H^i_T((\mathcal{O}_{K,S})_\mathcal{W},\GG_m):=H^i(R\Gamma_T((\mathcal{O}_{K,S})_{\mathcal{W}},\GG_m)).$$

\begin{remark}\label{notation} Our notation for the above cohomology groups and complexes is motivated by the following facts.

(i) Assume that $k$ is a function field. Write $C_k$ for the corresponding smooth projective curve, $C_{k,W{\rm\acute e t}}$ for the Weil-\'etale site on $C_k$ that is defined by Lichtenbaum in \cite[\S2]{L2} and $j$ for the open immersion ${\rm Spec}(\mathcal{O}_{k,S}) \longrightarrow C_k$. Then the group $H^i_{\!c}((\mathcal{O}_{\!K,S})_{\mathcal{W}},\ZZ)$ defined above is canonically isomorphic to the Weil-\'etale cohomology group $H^i(C_{k,W{\rm\acute e t}},j_{!}\ZZ)$.

(ii) Assume that $k$ is a number field. In this case there has as yet been no construction of a `Weil-\'etale topology' for $Y_S:= {\rm Spec}(\mathcal{O}_{K,S})$ with all of the properties that are conjectured by Lichtenbaum in \cite{L3}. However, if $\overline{Y_S}$ is a compactification of $Y_S$ and $\phi$ is the natural inclusion $Y_S\subset \overline{Y_S}$, then the approach of \cite{pnp} can be used to show that, should such a topology exist with all of the expected properties, then the groups $H^i_{\!c}((\mathcal{O}_{\!K,S})_{\mathcal{W}},\ZZ)$ defined above would be canonically isomorphic to the group $H^i_c(Y_S,\ZZ) := H^i(\overline{Y_S},\phi_!\ZZ)$ that is discussed in \cite{L3}.

(iii) The definition of $R\Gamma_T((\mathcal{O}_{K,S})_{\mathcal{W}},\GG_m)$ as the (shifted) linear dual of the complex $R\Gamma_{c,T}((\mathcal{O}_{K,S})_{\mathcal{W}}, \ZZ)$ is motivated by \cite[Rem. 3.8]{pnp} and hence by the duality theorem in Weil-\'etale
cohomology for curves over finite fields that is proved by Lichtenbaum in \cite{L2}.
\end{remark}


%
%

An analysis of the complex $R\Gamma_{\!c,T}((\mathcal{O}_{\!K,S})_{\mathcal{W}},\ZZ)$ as in the proof of Lemma \ref{dual fitting} below then leads us to give the following definition. In this definition we use the notion of `transpose' in the sense of Jannsen's homotopy theory of modules \cite{jannsen}.

\begin{definition} \label{defseltr}
\begin{rm}
The `transpose' of $\mathcal{S}_{S,T}(\GG_{m/K})$ is the group
\[  \mathcal{S}^{{\rm tr}}_{S,T}(\GG_{m/K}) := H^1_T((\mathcal{O}_{K,S})_\mathcal{W},\GG_m) = H^{-1}(R\Hom_\ZZ(R\Gamma_{c,T}((\mathcal{O}_{K,S})_{\mathcal{W}}, \ZZ),\ZZ)).\]
When $T$ is empty, we omit the subscript $T$ from this notation.
\end{rm}
\end{definition}

\begin{remark}\label{defseltr rem} By using the spectral sequence
$$E_2^{p,q}=\Ext_\ZZ^p(H_{c,T}^{-q}((\mathcal{O}_{K,S})_\mathcal{W},\ZZ),\ZZ) \Rightarrow H^{p+q+2}_{T}((\mathcal{O}_{K,S})_\mathcal{W},\GG_m),$$
which is obtained from (\ref{ss}), one can check that $R\Gamma_T((\mathcal{O}_{K,S})_{\mathcal{W}},\GG_m)$ is acyclic outside degrees zero and one, that there is canonical isomorphism
$$H^0_T((\mathcal{O}_{K,S})_\mathcal{W},\GG_m)\simeq \mathcal{O}_{K,S,T}^\times, $$
and that there is a canonical exact sequence
\begin{equation*}
0 \longrightarrow {\rm Cl}_{S}^T(K) \longrightarrow
\mathcal{S}^{{\rm tr}}_{S,T}(\GG_{m/K})
\longrightarrow X_{K,S} \longrightarrow 0
\end{equation*}
of the form (\ref{dual ses}).
\end{remark}

In the sequel we shall say that a $G$-module $M$ has a `locally-quadratic presentation' if it lies in an exact sequence of finitely generated $G$-modules of the form
\[ P \rightarrow P' \rightarrow M \rightarrow 0\]
in which $P$ and $P'$ are projective and the $\QQ[G]$-modules $\QQ P$ and $\QQ P'$ are isomorphic. 

\begin{lemma}\label{dual fitting}
Assume that $G$ is abelian, that $S$ contains $S_{\infty}(k)\cup S_{\rm ram}(K/k)$,
and that $\mathcal{O}_{K,S,T}^\times$ is $\ZZ$-torsion-free.
Then each of the groups $\mathcal{S}_{S,T}(\GG_{m/K})$ and
$\mathcal{S}^{{\rm tr}}_{S,T}(\GG_{m/K})$ have locally-quadratic presentations, and for each non-negative integer $i$ one has an equality
\[ {\rm Fitt}_G^i(\mathcal{S}^{{\rm tr}}_{S,T}(\GG_{m/K})) = {\rm Fitt}_G^i(\mathcal{S}_{S,T}(\GG_{m/K}))^\#.\]
\end{lemma}

\begin{proof} Set $C^\bullet := R\Gamma_{\!c,T}((\mathcal{O}_{\!K,S})_{\mathcal{W}},\ZZ)$ and $C^{\bullet,*}:= R\Hom_\ZZ(R\Gamma_{c,T}((\mathcal{O}_{\!K,S})_{\mathcal{W}},\ZZ),\ZZ)$.



From Proposition \ref{new one} we also know that $C^\bullet$ is a perfect complex of $G$-modules that is acyclic outside degree one and two and $\ZZ$-torsion-free in degree one. This implies, by a standard argument, that $C^\bullet$ can be represented by a complex $P\xrightarrow{\delta} P'$ of $G$-modules, where $P$ and $P'$ are finitely generated and projective and the first term is placed in degree one, and hence that there is a tautological exact sequence of $G$-modules
\begin{equation}\label{taut seq} 0 \longrightarrow H^1(C^\bullet) \longrightarrow P \stackrel{\delta}{\longrightarrow}P' \longrightarrow H^2(C^\bullet)\longrightarrow 0.\end{equation}
The descriptions in Proposition \ref{new one}(iii) imply that the linear dual of the Dirichlet regulator map $\lambda_{K,S}$ induces an isomorphism of $\RR[G]$-modules
\begin{equation}\label{dual reg} \lambda_{K,S}^*: \RR H^1(C^\bullet) \cong \RR H^2(C^\bullet).\end{equation}
Taken in conjunction with the sequence (\ref{taut seq}) this isomorphism implies that the $\QQ[G]$-modules $\QQ P$ and $\QQ P'$ are isomorphic and hence that $\mathcal{S}_{S,T}(\GG_{m/K}) = H^2(C^\bullet)$ has a locally-quadratic presentation, as claimed.

The complex $C^{\bullet,*}[-2]$ is represented by $\Hom_\ZZ(P',\ZZ)\stackrel{\delta^*}{\rightarrow}\Hom_\ZZ(P,\ZZ)$ where the linear duals are endowed with  contragredient action of $G$, the first term is placed in degree zero and $\delta^*$ is the map induced by $\delta$. There is therefore a tautological exact sequence
\begin{equation}\label{taut seq2} 0 \longrightarrow H^0(C^{\bullet,*}[-2]) \longrightarrow \Hom_\ZZ(P',\ZZ)\stackrel{\delta^*}{\longrightarrow}\Hom_\ZZ(P,\ZZ) \longrightarrow H^1(C^{\bullet,*}[-2])\longrightarrow 0,\end{equation}
and, since the above observations imply that $\Hom_\ZZ(P',\ZZ)$ and $\Hom_\ZZ(P,\ZZ)$ are projective $G$-modules that span isomorphic $\QQ[G]$-spaces, this sequence implies that the module $\mathcal{S}^{{\rm tr}}_{S,T}(\GG_{m/K}) = H^1(C^{\bullet,*}[-2])$ has a locally-quadratic presentation.

It now only remains to prove the final claim and it is enough to prove this after completion at each prime $p$. We shall denote for any abelian group $A$ the $p$-completion $A\otimes \ZZ_p$ of $A$ by $A_p$. By Swan's Theorem (cf. \cite[(32.1)]{curtisr}) one knows that for each prime $p$ the $\ZZ_p[G]$-modules $P_p$ and $P'_p$ are both free of rank, $d$ say, that is independent of $p$. In particular, after fixing bases of $P_p$ and $P'_p$ the homomorphism $P_p \stackrel{\delta}{\to} P_p'$ corresponds to a matrix $A_{\delta,p}$ in ${\rm M}_d(\ZZ_p[G])$ and the sequence (\ref{taut seq}) implies that the ideal ${\rm Fitt}_{G}^i(H^2(C^\bullet))_p$ is generated over $\ZZ_p[G]$ by the determinants of all $(d-i)\times(d-i)$ minors of $A_{\delta,p}$. The corresponding dual bases induce identifications of both $\Hom_\ZZ(P',\ZZ)_p$ and $\Hom_\ZZ(P,\ZZ)_p$ with $\ZZ_p[G]^{\oplus d}$, with respect to which the homomorphism $\Hom_\ZZ(P',\ZZ)_p \stackrel{\delta^\ast}{\to} \Hom_\ZZ(P,\ZZ)_p$ is represented by the matrix $A_{\delta,p}^{{\rm tr},\#}$ that is obtained by applying the involution $\#$ to each entry of the transpose of $A_{\delta,p}$. The exact sequence (\ref{taut seq2}) therefore implies that ${\rm Fitt}_{G}^i(H^1(C^{\bullet,*}[-2]))_p$ is generated over $\ZZ_p[G]$ by the determinants of all $(d-i)\times(d-i)$ minors of $A_{\delta,p}^{{\rm tr},\#}$. Hence one has an equality
$${\rm Fitt}_{G}^i(H^2(C^\bullet))_p={\rm Fitt}_{G}^i(H^1(C^{\bullet,*}[-2]))^{\#}_p,$$
as required.
\end{proof}





\subsection{Tate sequences}
In this subsection we review the construction of Tate's exact sequence, which is used in the formulation of the leading term conjecture in the next section.
Let $K/k,G,S$ be as in the previous subsection. We assume that $S_{\rm ram}(K/k) \subset S$. We assume only in this subsection
that $S$ is large enough so that ${\rm Cl}_S(K)$ vanishes.

In this setting, Tate constructed a `fundamental class' $\tau_{K/k,S}\in \Ext_G^2( X_{K,S}, \mathcal{O}_{K,S}^\times)$ using the class field theory \cite{tate-nag}. This class $\tau_{K/k,S}$ has the following property: if we regard $\tau_{K/k,S}$ as an element of $H^2(G,\Hom_\ZZ(X_{K,S},\mathcal{O}_{K,S}^\times))$ via the canonical isomorphism
$$\Ext_G^2(X_{K,S},\mathcal{O}_{K,S}^\times) \simeq \Ext_G^2(\ZZ, \Hom_\ZZ(X_{K,S},\mathcal{O}_{K,S}^\times))=H^2(G,\Hom_\ZZ(X_{K,S},\mathcal{O}_{K,S}^\times)),$$
then, for every integer $i$, the map between Tate cohomology groups
$$\widehat H^i(G,X_{K,S}) \stackrel{\sim}{\longrightarrow} \widehat H^{i+2}(G,\mathcal{O}_{K,S}^\times)$$
that is defined by taking cup product with $\tau_{K/k,S}$ is bijective.

The Yoneda extension class of $\tau_{K/k,S}$ is therefore represented by an exact sequence of the following sort:
\begin{eqnarray}
0\longrightarrow \mathcal{O}_{K,S}^\times \longrightarrow A \longrightarrow B \longrightarrow X_{K,S} \longrightarrow 0, \label{tateseqS}
\end{eqnarray}
where $A$ and $B$ are finitely generated cohomologically-trivial $G$-modules (see \cite[Chap. II, Th. 5.1]{tate}). We call this sequence a `Tate sequence' for $K/k$.
%
%
%
%

\begin{proposition} \label{proptateseq}
The complex $R\Gamma((\mathcal{O}_{K,S})_\mathcal{W},\GG_m)$ defines an element of
$$\Ext_{G}^2(\mathcal{S}^{{\rm tr}}_S(\GG_{m/K}),\mathcal{O}_{K,S}^\times).$$
This element is equal to Tate's fundamental class $\tau_{K/k,S}$.
\end{proposition}

\begin{proof} The first assertion follows directly from the discussion of Remark \ref{defseltr rem}.

The assumed vanishing of ${\rm Cl}_S(K)$ combines with the exact sequence (\ref{dual ses}) to imply that  $\mathcal{S}^{{\rm tr}}_S(\GG_{m/K})=X_{K,S}$. Given this, the second claim is proved by the first author in \cite[Prop. 3.5(f)]{pnp}
\end{proof}

\section{Zeta elements and the leading term conjecture} \label{secltc}

In this section, we suppose that
$K/k$ is a finite abelian extension of global fields with Galois group $G$.

We fix a finite non-empty set of places $S$ of $k$ which contains both $S_\infty(k)$ and $S_{\rm ram}(K/k)$ and an auxiliary finite set of places $T$ of $k$ that is disjoint from $S$.

\subsection{$L$-functions} \label{L-functionsubsection}
We recall the definition of (abelian) $L$-functions of global fields.
For any linear character $\chi \in \widehat G:=\Hom(G,\CC^\times)$, we define the $S$-truncated $T$-modified $L$-function for $K/k$ and $\chi$ by setting
$$L_{k,S,T}(\chi,s):=\prod_{v \in T}(1-\chi({\rm Fr}_v){\N} v^{1-s}) \prod_{v\notin S}(1-\chi({\rm Fr}_v){\N} v^{-s})^{-1}.$$
This is a complex function defined on ${\rm Re}(s)>1$ and is well-known to have a meromorphic continuation on $\CC$ and to be holomorphic at $s=0$. We denote by $r_{\chi,S}$ the order of vanishing of $L_{k,S,T}(\chi,s)$ at $s=0$ (this clearly does not depend on $T$). We denote the leading coefficient of the Taylor expansion of $L_{k,S,T}(\chi,s)$ at $s=0$ by
\[ L_{k,S,T}^\ast(\chi,0) := \lim_{s\to 0}s^{-r_{\chi,S}}L_{k,S,T}(\chi,s).\]
We then define the $S$-truncated $T$-modified equivariant $L$-function for $K/k$ by setting
$$\theta_{K/k,S,T}(s):=\sum_{\chi \in \widehat G}L_{k,S,T}(\chi^{-1},s)e_{\chi},$$
where $e_\chi:=\frac{1}{|G|}\sum_{\sigma \in G}\chi(\sigma)\sigma^{-1}$, and we define its leading term to be
\[ \theta_{K/k,S,T}^\ast(0):=\sum_{\chi\in \widehat G}L_{k,S,T}^\ast(\chi^{-1},0)e_\chi.\]
It is then easy to see that $\theta_{K/k,S,T}^\ast(0)$ belongs to $\RR[G]^\times$.

When $T=\emptyset$, we simply denote $L_{k,S,\emptyset}(\chi,s)$, $\theta_{K/k,S,\emptyset}(s)$, etc., by $L_{k,S}(\chi,s)$, $\theta_{K/k,S}(s)$, etc., respectively, and refer to them as the $S$-truncated $L$-function for $K/k$, $S$-truncated equivariant $L$-function for $K/k$, etc., respectively.

\subsection{The leading term lattice} \label{LTCandZeta}
 In this section we recall the explicit formulation of a conjectural description of the lattice $\theta_{K/k,S,T}^\ast(0)\cdot \ZZ[G]$ which involves Tate sequences.
In particular, up until Remark \ref{knownltc}, we always assume (without further explicit comment) that $S$ is large enough to ensure the group ${\rm Cl}_{S}(K)$ vanishes.

At the outset we also note that, as observed by Knudsen and Mumford in \cite{KM}, to avoid certain technical difficulties regarding signs, determinant modules must be regarded as graded invertible modules. Nevertheless, for simplicity of notation, in the following we have preferred to omit explicit reference to the grading of any graded invertible modules. Thus, for a finitely generated projective $G$-module $P$, we have abbreviated the graded invertible $G$-module $(\det_G(P), {\rm rk}_G(P))$ to ${\det}_G(P)$, where ${\rm rk}_G(P)$ is the rank of $P$. Since the notation ${\det}_G(P)$ explicitly indicates $P$, which in turn determines ${\rm rk}_G(P)$, we feel that this abbreviation should not cause difficulties.

We shall also use the following general notation. Suppose that we have a perfect complex $C^\bullet$ of $G$-modules, which is concentrated in degree $i$ and $i+1$ with some integer $i$, and an isomorphism $\lambda : \RR H^i(C^\bullet) \stackrel{\sim}{\rightarrow} \RR H^{i+1}(C^\bullet)$. Then we define an isomorphism
$$\vartheta_\lambda :  \RR{\det}_{G}(C^\bullet) \stackrel{\sim}{\longrightarrow} \RR[G]$$
as follows:

\begin{eqnarray}
 \RR{\det}_G(C^\bullet)&\stackrel{\sim}{\longrightarrow}& \bigotimes_{j\in \ZZ}{\det}_{\RR[G]}^{(-1)^{j}}(\RR C^{j})  \nonumber \\
&\stackrel{\sim}{\longrightarrow}&   \bigotimes_{j\in\ZZ}{\det}_{\RR[G]}^{(-1)^{j}}(\RR H^j(C^\bullet))\nonumber \\
&=& {\det}_{\bbR[G]}^{(-1)^i}(\bbR H^i(C^\bullet))\otimes_{\bbR[G]}{\det}_{\bbR[G]}^{(-1)^{i+1}}(\RR H^{i+1}(C^\bullet)) \nonumber \\
&\stackrel{\sim}{\longrightarrow}&  {\det}_{\bbR[G]}^{(-1)^i}(\bbR H^{i+1}(C^\bullet))\otimes_{\bbR[G]}{\det}_{\bbR[G]}^{(-1)^{i+1}}(\bbR H^{i+1}(C^\bullet)) \nonumber \\
&\stackrel{\sim}{\longrightarrow}& \bbR[G], \nonumber
\end{eqnarray}
where the fourth isomorphism is induced by $\lambda^{(-1)^i}$.

Let $A$ and $B$ be the $G$-modules which appear in the Tate sequence (\ref{tateseqS}). Since we have the regulator isomorphism
$$\lambda_{K,S}: \RR \mathcal{O}_{K,S}^\times \stackrel{\sim}{\longrightarrow} \RR X_{K,S},$$
the above construction for $C^\bullet=(A\rightarrow B)$, where $A$ is placed in degree $0$, gives the isomorphism
$$\vartheta_{\lambda_{K,S}} :  \RR{\det}_G(A) \otimes_{\RR[G]} \RR{\det}_G^{-1}(B) \stackrel{\sim}{\longrightarrow} \RR[G].$$


We study the following conjecture.

\begin{conjecture}\label{ltc0} In $\RR[G]$ one has
$$\vartheta_{\lambda_{K,S}}({\det}_G(A)\otimes_G{\det}_G^{-1}(B))=\theta_{K/k,S}^\ast(0)\cdot\ZZ[G].$$
\end{conjecture}

\begin{remark} \label{remltc}
\begin{rm}
 This conjecture coincides with the conjecture C($K/k$) stated in \cite[\S6.3]{burns}. The observations made in \cite[Rem. 6.2]{burns} therefore imply that Conjecture \ref{ltc0} is equivalent in the number field case to the `equivariant Tamagawa number conjecture' \cite[Conj. 4 (iv)]{BF Tamagawa} for the pair $(h^0(\Spec K),\ZZ[G])$, that the validity of Conjecture \ref{ltc0} is independent of $S$ and of the choice of Tate sequence and that its validity for the extension $K/k$ implies its validity for all extensions $F/E$ with $k \subseteq E \subseteq F\subseteq K$.
\end{rm}
\end{remark}

\begin{remark} \label{knownltc}
\begin{rm}
Conjecture \ref{ltc0} is known to be valid in each of the following cases:
\begin{itemize}
\item[(i)]{$K$ is an abelian extension of $\QQ$ (by Greither and the first author \cite{bg} and Flach \cite{fg}),}
\item[(ii)]{$k$ is a global function field (by the first author \cite{gmcrc}),}
\item[(iii)]{$[K:k]\leq 2$ (by Kim \cite[\S2.4, Rem. i)]{kim}).}
\end{itemize}
\end{rm}
\end{remark}

In the following result we do not assume that the group ${\rm Cl}_S(K)$ vanishes and we interpret the validity of Conjecture \ref{ltc0} in terms of the `Weil-\'etale cohomology' complexes $R\Gamma_{\!c,T}((\mathcal{O}_{\!K,S})_{\mathcal{W}},\ZZ)$ and $R\Gamma_T((\mathcal{O}_{K,S})_\mathcal{W},\GG_m)$ defined in \S\ref{wec}.

We note at the outset that $R\Gamma_{\!c,T}((\mathcal{O}_{\!K,S})_{\mathcal{W}},\ZZ)$ (resp. $R\Gamma_T((\mathcal{O}_{K,S})_\mathcal{W},\GG_m)$) is represented by a complex which is concentrated in degrees one and two (resp. zero and one), and so we can define the isomorphism
$$\vartheta_{\lambda_{K,S}^\ast} : \RR{\det}_G(R\Gamma_{\!c,T}((\mathcal{O}_{\!K,S})_{\mathcal{W}},\ZZ)) \stackrel{\sim}{\longrightarrow} \RR[G]$$
$$(\text{resp. } \vartheta_{\lambda_{K,S}} : \RR{\det}_G( R\Gamma_T((\mathcal{O}_{K,S})_\mathcal{W},\GG_m)) \stackrel{\sim}{\longrightarrow} \RR[G]).$$

%

\begin{proposition}\label{ltc prop2} Let $S$ be any finite non-empty set of places of $k$ containing both $S_\infty(k)$ and $S_{\rm ram}(K/k)$ and let $T$ be any finite set of places of $k$ that is disjoint from $S$. Then the following conditions on $K/k$ are equivalent.

\begin{itemize}
\item[(i)] Conjecture \ref{ltc0} is valid.
\item[(ii)] In $\mathbb{R}[G]$ one has an equality
\[ \vartheta_{\lambda_{K,S}^*}({\rm det}_G(R\Gamma_{c,T}((\mathcal{O}_{K,S})_\mathcal{W},\ZZ))) = {\theta_{K/k,S,T}^*(0)^{-1}}^\#\cdot\ZZ[G]. \]
\item[(iii)] In $\mathbb{R}[G]$ one has an equality
\[ \vartheta_{\lambda_{K,S}}({\rm det}_G(R\Gamma_{T}((\mathcal{O}_{K,S})_{\mathcal{W}},\GG_m))) = \theta_{K/k,S,T}^*(0)\cdot\ZZ[G]. \]
\end{itemize}
\end{proposition}

\begin{proof}

For any finitely generated projective $G$-module $P$ of (constant) rank $d$ there is a natural identification
\[ \bigwedge^d_{G}\Hom_\ZZ(P,\ZZ) \cong  \bigwedge^d_{G}\Hom_{G}(P,\ZZ[G])^{\#} \cong \Hom_{G}(\bigwedge^d_{G}P,\ZZ[G])^{\#},\]
where $G$ acts on $\Hom_\ZZ(P,\ZZ)$ contragrediently and on $\Hom_{G}(P,\ZZ[G])$ via right multiplication. The equivalence of the equalities in claims (ii) and (iii) is therefore a consequence of the fact that for any element $\Delta$ of the mutliplicative group of invertible $\ZZ[G]$-lattices in $\RR[G]$ the evaluation pairing identifies $\Hom_{G}(\Delta,\ZZ[G])^{\#}$ with the image under the involution $\#$ of the inverse lattice $\Delta^{-1}$.

To relate the equalities in claims (ii) and (iii) to Conjecture \ref{ltc0} we note first that the third column of (\ref{diagram}) implies that
\[ \vartheta_{\lambda_{K,S}^*}({\rm det}_G(R\Gamma_{c,T}((\mathcal{O}_{K,S})_\mathcal{W},\ZZ))) = {\rm det}_G((\mathbb{F}_{T_{K}}^\times)^\vee[-2])\cdot\vartheta_{\lambda_{K,S}^*}({\rm det}_G(R\Gamma_{c}((\mathcal{O}_{K,S})_\mathcal{W},\ZZ))),\]
whilst the resolution (\ref{T seq}) implies that

\begin{align*}{\rm det}_G((\mathbb{F}_{T_{K}}^\times)^\vee[-2]) &= (\prod_{v \in T}(1-{\N}v \mathrm{Fr}_w))^{-1}\cdot\ZZ[G]\\
&= {(\theta_{K/k,S,T}^*(0)/\theta_{K/k,S}^*(0))^{-1}}^{\#}\cdot\ZZ[G].\end{align*}
The equality in claim (ii) is therefore equivalent to an equality
\begin{equation}\label{no T} \vartheta_{\lambda_{K,S}^*}({\rm det}_G(R\Gamma_{c}((\mathcal{O}_{K,S})_\mathcal{W},\ZZ))) = {\theta_{K/k,S}^*(0)^{-1}}^\#\cdot\ZZ[G]. \end{equation}

We now choose an auxiliary set of places $S''$ as in the proof of Proposition \ref{new one} and set $S':= S\cup S''$. By Chebotarev density theorem we can even assume that all places in $S''$ split completely in $K/k$ and, for simplicity, this is what we shall do. Then, in this case, the exact triangle (\ref{milne tri}) combines with the upper triangle in (\ref{diagram}) to give an exact triangle in $D(\ZZ[G])$ of the form
\begin{equation}\label{milne tri2} Y_{K,S''}[-1]\oplus Y_{K,S''}[-2] \stackrel{\alpha}{\longrightarrow} R\Gamma_{\!\!c}((\mathcal{O}_{K,S'})_\mathcal{W},\ZZ) \stackrel{\beta}{\longrightarrow} R\Gamma_{\!\!c}((\mathcal{O}_{K,S})_\mathcal{W},\ZZ)
\longrightarrow . \end{equation}
After identifying the cohomology groups of the second and the
third occurring complexes by using Proposition \ref{new one}(iii) the long exact cohomology sequence of this triangle induces (after scaler extension) the sequence
\begin{multline*} 0 \longrightarrow  \QQ Y_{K,S''} \longrightarrow \QQ Y_{K,S'}/\Delta_{S'}(\QQ) \longrightarrow \QQ Y_{K,S}/\Delta_S(\QQ)
\\ \stackrel{0}{\longrightarrow} \QQ Y_{K,S''} \stackrel{{\ord}_{S''}^*}{\longrightarrow}  \Hom_\ZZ(\mathcal{O}^\times_{K,S'},\QQ) \xrightarrow{\pi_{\!S''}} \Hom_\ZZ(\mathcal{O}^\times_{K,S},\QQ)\longrightarrow 0.\end{multline*}
Here ${\ord}_{S''}^*$ is induced by the linear dual of the map $\mathcal{O}^\times_{K,S'}\rightarrow Y_{K,S''}$ induced by taking valuations at each place in $S''_K$ and $\pi_{S''}$ by the linear dual of the inclusion $\mathcal{O}^\times_{K,S}\subseteq \mathcal{O}^\times_{K,S'}$ and all other maps are obvious. This sequence implies that there is an exact commutative diagram
\[\minCDarrowwidth1em\begin{CD}
0 @> >> \RR Y_{K,S''} @>  H^1(\alpha)>> \mathbb{R} H^1_{c}((\mathcal{O}_{K,S'})_\mathcal{W},\ZZ) @>  H^1(\beta)>> \mathbb{R} H^1_{c}((\mathcal{O}_{K,S})_\mathcal{W},\ZZ) @> >> 0\\
@. @V \eta_{S''}VV @V \lambda_{K,S'}^*VV @V \lambda_{K,S}^*VV \\
0 @> >> \RR Y_{K,S''} @>  H^2(\alpha)>> \mathbb{R}H^2_{c}((\mathcal{O}_{K,S'})_\mathcal{W},\ZZ) @>  H^2(\beta)>> \mathbb{R} H^2_{c}((\mathcal{O}_{K,S})_\mathcal{W},\ZZ) @> >> 0\end{CD}\]
where $\eta_{S''}$ sends each sum $\sum_{v \in S''}\sum_{w\mid v}x_w w$ to $\sum_{v \in S''}\sum_{w\mid v}{\rm log}({\rm N}v) x_w w$.

This diagram combines with the triangle (\ref{milne tri2}) to imply that
\begin{align*}\vartheta_{\lambda_{K,S'}^*}({\rm det}_G(R\Gamma_{c}((\mathcal{O}_{K,S'})_\mathcal{W},\ZZ))) &= {\rm det}_{\RR[G]}(\eta_{S''})^{-1}  \vartheta_{\lambda_{K,S}^*}({\rm det}_G(R\Gamma_{c}((\mathcal{O}_{K,S})_\mathcal{W},\ZZ)))\\
&= \bigl(\prod_{v\in S''}{\rm log}({\rm N}v)\bigr)^{-1}  \vartheta_{\lambda_{K,S}^*}({\rm det}_G(R\Gamma_{c}((\mathcal{O}_{K,S})_\mathcal{W},\ZZ))).\end{align*}
Since $\theta_{K/k,S'}^*(0) = \bigl(\prod_{v\in S''}{\rm log}({\rm N}v)\bigr)\theta_{K/k,S}^*(0)$ this equality shows that (after changing $S$ if necessary) we may assume that ${\rm Cl}_{S}(K)$ vanishes when verifying (\ref{no T}). Given this, the proposition follows from Proposition \ref{proptateseq}.
%
%
\end{proof}

\subsection{`Zeta elements'}\label{zeta els}
We now use the above results to reinterpret Conjecture \ref{ltc0} in terms of the existence of a canonical `zeta element'. This interpretation will then play a key role in the proofs of Theorem \ref{ltcrs}, \ref{ltcmrs} and \ref{ltcfit} given below.

The following definition of zeta element is in the same spirit as that used by Kato in \cite{K1} and \cite{K2}.

\begin{definition}
\begin{rm}
The `zeta element' $z_{K/k,S,T}$ of $\mathbb{G}_m$ relative to the data $K/k, S$ and $T$ is the unique element of
\[ \RR{\det}_G(R\Gamma_T((\mathcal{O}_{K,S})_\mathcal{W},\GG_m)) \cong {\det}_{\RR[G]}(\RR\mathcal{O}_{K,S}^\times) \otimes_{\RR[G]} {\det}_{\RR[G]}^{-1}(\RR X_{K,S})\]
which satisfies $\vartheta_{\lambda_{K,S}}(z_{K/k,S,T})=\theta_{K/k,S,T}^\ast(0).$
\end{rm}
\end{definition}

The following `leading term conjecture' is then our main object of study.

\begin{conjecture}[${\rm LTC}(K/k)$]\label{ltc} In $\RR{\det}_G(R\Gamma_T((\mathcal{O}_{K,S})_\mathcal{W},\GG_m))$ one has an equality
\[\ZZ[G]\cdot z_{K/k,S,T}={\det}_G(R\Gamma_T((\mathcal{O}_{K,S})_{\mathcal{W}},\GG_m)).\]
\end{conjecture}

Given the definition of $z_{K/k,S,T}$, Proposition \ref{ltc prop2} implies immediately that this conjecture is equivalent to Conjecture \ref{ltc0} and hence is independent of the choices of $S$ and $T$.


\section{Preliminaries concerning exterior powers} \label{algpreliminary}
In this section, we recall certain useful constructions concerning exterior powers and also prove algebraic results that are to be used in later sections.

\subsection{Exterior powers}

Let $G$ be a finite abelian group. For a $G$-module $M$ and $f \in\Hom_G(M,\ZZ[G])$, there is a $G$-homomorphism
$$\bigwedge_G^rM \longrightarrow \bigwedge_G^{r-1}M$$
for all $r\in\ZZ_{\geq1}$, defined by
$$m_1\wedge\cdots\wedge m_r \mapsto \sum_{i=1}^r (-1)^{i-1}f(m_i)m_1\wedge\cdots\wedge m_{i-1}\wedge m_{i+1}\wedge\cdots \wedge m_r.$$
This morphism is also denoted by $f$.

This construction gives a homomorphism
\begin{eqnarray}
\bigwedge_G^s\Hom_G(M,\ZZ[G]) \longrightarrow \Hom_G(\bigwedge_G^rM, \bigwedge_G^{r-s}M) \label{extmap}
\end{eqnarray}
for all $r, s\in \ZZ_{\geq0}$ such that $r\geq s$, defined by
$$f_1\wedge\cdots\wedge f_s \mapsto (m \mapsto f_s \circ \cdots \circ f_1(m)).$$
By using this homomorphism we often regard an element of $\bigwedge_G^s\Hom_G(M,\ZZ[G])$ as an element of $\Hom_G(\bigwedge_G^rM, \bigwedge_G^{r-s}M)$.

For a $G$-algebra $Q$ and a homomorphism $f$ in $\Hom_G(M,Q)$, there is a $G$-homomorphism
$$\bigwedge_G^rM \longrightarrow (\bigwedge_G^{r-1}M) \otimes_{G}Q$$
defined by
$$m_1\wedge\cdots\wedge m_r \mapsto \sum_{i=1}^r (-1)^{i-1}m_1\wedge\cdots\wedge m_{i-1}\wedge m_{i+1}\wedge\cdots \wedge m_r \otimes f(m_i).$$
By the same method as the construction of (\ref{extmap}), we have a homomorphism
\begin{eqnarray}
\bigwedge_G^s\Hom_G(M,Q) \longrightarrow \Hom_G(\bigwedge_G^rM, (\bigwedge_G^{r-s}M)\otimes_G Q). \label{extmap2}
\end{eqnarray}

In the sequel we will find an explicit description of this homomorphism to be useful. This description is well-known and given by the following proposition, the proof of which we omit.

\begin{proposition} \label{propformula}
Let $m_1,\ldots,m_r \in M$ and $f_1,\ldots,f_s\in \Hom_G(M,Q)$. Then we have
$$(f_1\wedge\cdots\wedge f_s)(m_1\wedge\cdots\wedge m_r) =\sum_{\sigma \in {\mathfrak{S}_{r,s}} }{\rm{sgn}}(\sigma) m_{\sigma(s+1)}\wedge\cdots\wedge m_{\sigma(r)}\otimes \det(f_i(m_{\sigma(j)}))_{1\leq i,j\leq s},$$
where
$$\mathfrak{S}_{r,s}:=\{ \sigma \in \mathfrak{S}_r  : \sigma(1) < \cdots < \sigma(s) \text{ and } \sigma(s+1) <\cdots<\sigma(r) \}.$$
In particular, if $r=s$, then we have
$$(f_1 \wedge \cdots \wedge f_r)(m_1 \wedge \cdots \wedge m_r)=\det(f_i(m_j))_{1\leq i,j \leq r}.$$
\end{proposition}

We will also find the technical observations that are contained in the next two results to be very useful.

\begin{lemma} \label{leme}
Let $E$ be a field and $A$ an $n$-dimensional $E$-vector space. If we have an $E$-linear map
$$\Psi:A\longrightarrow E^{\oplus m},$$
where $\Psi=\bigoplus_{i=1}^m \psi_i$ with $\psi_1,\ldots,\psi_m \in \Hom_E(A,E)$ ($m\leq n$), then we have
$$\im(\bigwedge_{1\leq i\leq m}\psi_i : \bigwedge_E^n A \longrightarrow \bigwedge_E^{n-m}A)=\begin{cases}
\bigwedge_E^{n-m}\ker (\Psi), & \text{if $\Psi$ is surjective},  \\
0, &\text{if $\Psi$ is not surjective.}
\end{cases}
$$
\end{lemma}

\begin{proof}
Suppose first that $\Psi$ is surjective. Then there exists a subspace $B \subset A$ such that $A=\ker (\Psi) \oplus B$ and $\Psi$ maps $B$ isomorphically onto $E^{\oplus m}$.
We see that $\bigwedge_{1\leq i \leq m}\psi_i$ induces an isomorphism
$$\bigwedge_E^m B \stackrel{\sim}{\longrightarrow} E.$$
Hence we have an isomorphism
$$\bigwedge_{1\leq i\leq m}\psi_i:\bigwedge_E^nA=\bigwedge_E^{n-m}\ker (\Psi) \otimes_E \bigwedge_E^m B \stackrel{\sim}{\longrightarrow} \bigwedge_E^{n-m}\ker (\Psi).$$
In particular, we have
$$\im(\bigwedge_{1\leq i\leq m}\psi_i : \bigwedge_E^n A \longrightarrow \bigwedge_E^{n-m}A)=\bigwedge_E^{n-m}\ker (\Psi).$$
Next, suppose that $\Psi$ is not surjective. Then $\psi_1,\ldots,\psi_m \in \Hom_E(A,E)$ are linearly dependent. In fact,
since each $\psi_i$ is contained in $\Hom_E(A/ \ker (\Psi),E)$, we have
$$\dim_E(\langle\psi_1,\ldots,\psi_m\rangle) \leq \dim_E(A/\ker (\Psi)) =\dim_E(\im (\Psi)) ,$$
so $\dim_E(\langle\psi_1,\ldots,\psi_m\rangle) <m$ if $\dim_E(\im (\Psi))<m$.
This shows that the element $\bigwedge_{1\leq i \leq m}\psi_i$ vanishes, as required.
\end{proof}

Using the same notation as in Lemma \ref{leme}, we now consider an endomorphism $\psi \in \End_E(A)$. We write $r_\psi$ for the dimension over $E$ of $\ker(\psi)$ and consider the composite isomorphism
\begin{eqnarray}
F_\psi: \bigwedge_E^n A \otimes_E \bigwedge_E^n \Hom_E(A,E)
&\simeq& {\det}_E(A) \otimes_E {\det}_E^{-1}(A) \nonumber \\
&\stackrel{\sim}{\longrightarrow}& {\det}_E( \ker (\psi)) \otimes_E {\det}_E^{-1} (\coker (\psi)) \nonumber \\
&\simeq& \bigwedge_E^{r_\psi} \ker(\psi) \otimes_E \bigwedge_E^{r_\psi}\Hom_E(\coker(\psi),E), \nonumber
\end{eqnarray}
where the second isomorphism is induced by the tautological exact sequence
$$0 \longrightarrow \ker(\psi) \longrightarrow A \stackrel{\psi}{\longrightarrow} A \longrightarrow \coker(\psi) \longrightarrow 0.$$

Then the proof of Lemma \ref{leme} leads directly to the following useful description of this isomorphism $F_\psi$.

\begin{lemma}\label{remarkpsi} With $E, A$ and $\psi$ as above, we fix an $E$-basis $\{b_1,\ldots, b_n\}$ of $A$ so that $\im (\psi)=\langle b_{r_\psi+1},\ldots,b_n \rangle$ and write $\{b_1^\ast,\ldots, b_n^\ast\}$ for the corresponding dual basis of
 $\Hom_E(A,E)$. For each index $i$ we also set $\psi_i:=b_i^\ast \circ \psi$. 

Then for every $a$ in $\bigwedge_E^n A$ the element $(\bigwedge_{r_\psi < i \leq n} \psi_i)(a)$ belongs to $\bigwedge_E^{r_\psi} \ker(\psi) $ and one has
\[ F_\psi(a \otimes (b_1^\ast \wedge\cdots \wedge b_n^\ast)) = (-1)^{r_\psi (n-r_\psi)}(\bigwedge_{r_\psi < i \leq n} \psi_i)(a) \otimes  (b_1^\ast \wedge \cdots \wedge b_{r_\psi}^\ast).\]
Here, on the right hand side of the equation, we use the equality $\im (\psi)=\langle b_{r_\psi+1},\ldots,b_n \rangle$ to regard $b_i^\ast$ for each $i$ with $1\le i \le r_\psi$ as an element of $\Hom_E(\coker(\psi),E)$. \end{lemma}

\subsection{The Rubin lattices}

The following definition is due to Rubin \cite[\S 1.2]{R}.
We adopt the notation in \cite{sano} for the lattice.
Note that the notation $\bigcap$ does not mean the intersection.

\begin{definition} \label{deflat}
\begin{rm} For a finitely generated $G$-module $M$ and a non-negative integer $r$ we define the `$r$-th Rubin lattice' by setting
$$\bigcap_G^r M =\{ m\in\QQ\bigwedge_G^rM  :  \Phi(m)\in\ZZ[G] \mbox{ for all }\Phi\in\bigwedge_G^r\Hom_G(M, \ZZ[G]) \}.$$
In particular, one has $\bigcap_G^0 M=\ZZ[G].$
\end{rm}
\end{definition}

\begin{remark} \label{rlat}
\begin{rm}
We define the homomorphism
$\iota : \bigwedge_G^r\Hom_G(M,\bbZ[G]) \rightarrow \Hom_G(\bigwedge_G^rM,\bbZ[G])$
by sending each element $\varphi_1 \wedge \cdots \wedge \varphi_r $ to $\varphi_r \circ \cdots \circ \varphi_1$ (see (\ref{extmap})). Then it is not difficult to see that the map
$$\bigcap_G^r M \stackrel{\sim}{\longrightarrow} \Hom_G(\im (\iota), \bbZ[G]) \quad ; \quad m \mapsto (\Phi \mapsto \Phi(m))$$
is an isomorphism (see \cite[\S 1.2]{R}).
\end{rm}
\end{remark}

By this remark, one obtains the following result.

\begin{proposition} \label{propproj}
Let $P$ be a finitely generated projective $G$-module. Then we have
$$\bigcap_G^r P=\bigwedge_G^r P$$
for all non-negative integers $r$.
\end{proposition}

\begin{lemma} \label{lemr}
Let $M$ be a $G$-module. Suppose that there is a finitely generated projective $G$-module $P$ and an injection $j : M \hookrightarrow P$ whose cokernel is $\ZZ$-torsion-free.
\begin{itemize}
\item[(i)]{The map
$$\Hom_G(P,\ZZ[G]) \longrightarrow \Hom_G(M,\ZZ[G])$$
induced by $j$ is surjective.}
\item[(ii)]
{If we regard $M$ as a submodule of $P$ via $j$, then we have
$$\bigcap_G^rM = (\bbQ\bigwedge_G^rM) \cap \bigwedge_G^rP.$$}
\end{itemize}
\end{lemma}
\begin{proof}
The assertion (i) follows from \cite[Prop. 1.1 (ii)]{R}. Note that
$$\bigwedge_G^r\Hom_G(P,\ZZ[G]) \longrightarrow \bigwedge_G^r\Hom_G(M,\ZZ[G])$$
is also surjective. This induces a surjection
$$\im (\iota_P) \longrightarrow \im (\iota_M),$$
where $\iota_P$ and $\iota_M$ denote the maps defined in Remark \ref{rlat} for $P$ and $M$, respectively. Hence, taking the dual, we have an injection
$$\bigcap_G^rM\simeq \Hom_G(\im (\iota_M),\ZZ[G]) \longrightarrow \Hom_G(\im (\iota_P),\ZZ[G])\simeq \bigcap_G^rP.$$
Since $P$ is projective, we have $\bigcap_G^rP=\bigwedge_G^rP$ by Proposition \ref{propproj}. Hence we have
$$\bigcap_G^rM\subset \bigwedge_G^rP.$$
Next, we show the reverse inclusion `$\supset$'. To do this we fix $a$ in $(\bbQ\bigwedge_G^rM) \cap \bigwedge_G^rP$ and $\Phi$ in $\bigwedge_G^r\Hom_G(M,\ZZ[G])$. By (i),
we can take a lift $\widetilde \Phi \in \bigwedge_G^r\Hom_G(P,\ZZ[G])$ of $\Phi$. Since $a\in \bigwedge_G^rP$, we have
$$\Phi(a)=\widetilde \Phi(a)\in \bbZ[G].$$
This shows that $a$ belongs to $\bigcap_G^rM$, as required.
\end{proof}

\begin{remark} \label{remr}
\begin{rm}
The proof of Lemma \ref{lemr} shows that the cokernel of the injection
$$\bigcap_G^rM \longrightarrow \bigwedge_G^r P$$
is $\ZZ$-torsion-free.
This implies that for any abelian group $A$, the map
$$(\bigcap_G^rM) \otimes_\ZZ A \longrightarrow( \bigwedge_G^r P)\otimes_\ZZ A$$
is injective.
\end{rm}
\end{remark}

\subsection{}

In the sequel we fix a subgroup $H$ of $G$ and an ideal $J$ of $\ZZ[H]$.
Recall that we denote the augmentation ideal of $\ZZ[H]$ by $I(H)$. Put $J_H:=J/I(H)J$. We also put $\mathcal{J}:=J\ZZ[G],$
and $\mathcal{J}_H := \mathcal{J}/I(H)\mathcal{J}.$

\begin{proposition} \label{propisom}
We have a natural isomorphism of $G/H$-modules
$$\mathcal{J}_H \simeq  \ZZ[G/H] \otimes_\ZZ J_H.$$
\end{proposition}

\begin{proof}
Define a homomorphism
$$\ZZ[G/H]\otimes_\ZZ J_H \longrightarrow \mathcal{J}_H$$
by $\tau \otimes \overline a \mapsto \overline {\widetilde \tau a}$, where $\tau \in G/H$, $a \in J$, and $\widetilde \tau \in G$ is a lift of $\tau$. One can easily check that this homomorphism is well-defined, and bijective.
\end{proof}

\begin{definition}
\begin{rm}
Let $M$ be a $G$-module. For $\varphi \in \Hom_G(M,\ZZ[G])$, we define $\varphi^H \in \Hom_{G/H}(M^H,\ZZ[G/H])$ by
$$M^H \stackrel{\varphi}{\longrightarrow} \ZZ[G]^H \simeq \ZZ[G/H],$$
where the last isomorphism is given by
$\N_H=\sum_{\sigma \in H} \sigma \mapsto 1$. Let $r$ be a non-negative integer. For $\Phi \in \bigwedge_G^r \Hom_G(M,\ZZ[G])$, we define $\Phi^H \in \bigwedge_{G/H}^r \Hom_{G/H}(M^H,\ZZ[G/H])$ to be the image of $\Phi$ under the map
$$\varphi_1\wedge\cdots\wedge \varphi_r \mapsto \varphi_1^H \wedge\cdots \wedge \varphi_r^H.$$
For convention, if $r=0$, then we define $\Phi^H \in \ZZ[G/H]$ to be the image of $\Phi \in \ZZ[G]$ under the natural map $: \ZZ[G] \longrightarrow \ZZ[G/H]$.
\end{rm}
\end{definition}

\begin{proposition} \label{propphi}
Let $M$ be a $G$-module and $r \in \ZZ_{\geq 0}$. For any $m \in \QQ \bigwedge_G^r M$ and $\Phi \in \bigwedge_G^r \Hom_G(M,\ZZ[G])$, we have
$$\Phi(m)=\Phi^H(\N_H^r m) \ in \ \QQ[G/H],$$
where $\N_H^r$ denote the map $\QQ \bigwedge_G^r M \to \QQ \bigwedge_{G/H}^rM^H$ induced by $\N_H: M \to M^H$.
\end{proposition}

\begin{proof}
This follows directly from the definition of $\Phi^H$.
\end{proof}

We consider the canonical map
$$\nu : \bigcap_{G/H}^r M^H \longrightarrow \bigcap_{G}^r M$$
which is defined as follows.
Let
$$\iota : \bigwedge_G^r \Hom_G(M,\ZZ[G]) \longrightarrow \Hom_G(\bigwedge_G^r M,\ZZ[G])$$
and
$$\iota_H : \bigwedge_{G/H}^r \Hom_{G/H}(M^H,\ZZ[G/H] ) \longrightarrow \Hom_{G/H}(\bigwedge_{G/H}^r M^H,\ZZ[G/H])$$
be the homomorphisms defined in Remark \ref{rlat}. The map
$$\im (\iota) \longrightarrow \im (\iota_H) \ ; \ \iota(\Phi) \mapsto \iota_H(\Phi^H)$$
induces a map
$$\alpha: \Hom_G(\im (\iota_H),\ZZ[G]) \longrightarrow \Hom_G(\im (\iota),\ZZ[G]) \simeq \bigcap_G^r M.$$
Note that we have a canonical isomorphism
$$\beta: \Hom_G(\im (\iota_H),\ZZ[G]) \stackrel{\sim}{\longrightarrow} \Hom_{G/H}(\im (\iota_H),\ZZ[G/H])\simeq \bigcap_{G/H}^rM^H \ ; \ \varphi \mapsto \varphi^H.$$
We define a map $\nu$ by
$$\nu:=\alpha\circ \beta^{-1} : \bigcap_{G/H}^rM^H \longrightarrow \bigcap_{G}^r M.$$

\begin{proposition} \label{propinj}
Let $M$ be a finitely generated $G$-module which is $\ZZ$-torsion-free. For any $r \in \ZZ_{\geq 0}$, the map
$\nu : \bigcap_{G/H}^r M^H \rightarrow \bigcap_{G}^r M$ is injective.
Furthermore, the maps
$$(\bigcap_{G/H}^r M^H) \otimes_\ZZ J_H \longrightarrow (\bigcap_{G}^r M) \otimes_\ZZ J_H \longrightarrow( \bigcap_G^r M) \otimes_\ZZ \ZZ[H]/I(H)J$$
are both injective, where the first map is induced by $\nu$, and the second by inclusion $J_H \hookrightarrow \ZZ[H]/I(H)J$.
\end{proposition}

\begin{proof}
The proof is the same as \cite[Lem. 2.11]{sano}, so we omit it.
\end{proof}

\begin{remark} \label{remnu}
\begin{rm}
The inclusion $M^H \subset M$ induces a map
$$\xi: \bigcap_{G/H}^r M^H \longrightarrow \bigcap_{G}^rM.$$
We note that this map {\it does not coincide} with
the above map $\nu$ if $r >1$. Indeed, one can check that
$\im (\xi) \subset |H|^{{\rm max} \{0,r-1 \}} \bigcap_G^r M$
(see \cite[Lem. 4.8]{MR2}), and
$$\nu =  |H|^{-{\rm max} \{0,r-1 \}} \xi. $$
\end{rm}
\end{remark}

\begin{remark} \label{remnorm}
\begin{rm}
Let $P$ be a finitely generated projective $G$-module. Then, any element of $P^H$ is written as $\N_H a$ with some $a \in P$, since $P$ is cohomologically trivial. One can check that, if $r>0$ (resp. $r=0$), then the map $\nu :\bigwedge_{G/H}^r P^H \rightarrow \bigwedge_G^r P$ constructed above coincides with the map
$$\N_H a_1\wedge\cdots\wedge \N_H a_r \mapsto \N_H a_1 \wedge \cdots \wedge a_r$$
$$(\mbox{resp. } \ZZ[G/H] \simeq \ZZ[G]^H \hookrightarrow \ZZ[G]).$$
In particular, we know that $\im (\nu) = \N_H \bigwedge_G^r P$.
\end{rm}
\end{remark}

\begin{proposition} \label{thminj}
Let $M$ be a finitely generated $G$-module which is $\ZZ$-torsion-free, and $r \in \bbZ_{\geq0}$. Then the map
$$(\bigcap_{G/H}^rM^H)\otimes_\bbZ J_H \longrightarrow \Hom_G(\bigwedge_G^r\Hom_G(M,\bbZ[G]),\mathcal{J}_H) \quad ; \quad \alpha \mapsto (\Phi \mapsto \Phi^H(\alpha))$$
is injective. (We regard $\Phi^H(\alpha) \in \ZZ[G/H]\otimes_\ZZ J_H$ as an element of $\mathcal{J}_H$ via the isomorphism $\mathcal{J}_H \simeq \ZZ[G/H] \otimes_\ZZ J_H$ in Proposition \ref{propisom}.)
\end{proposition}

\begin{proof}
The proof is the same as \cite[Th. 2.12]{sano}.
\end{proof}

\subsection{}

The following definition is originally due to Darmon \cite{D}, and used in \cite[Def. 2.13]{sano} and \cite[Def. 5.1]{MR2}.

\begin{definition} \label{defnorm}
\begin{rm}
Let $M$ be a $G$-module. For $m \in M$, define
$$\mathcal{N}_H(m)=\sum_{\sigma \in H} \sigma m \otimes \sigma^{-1} \in M \otimes_\ZZ \ZZ[H]/I(H)J.$$
\end{rm}
\end{definition}

The following proposition is an improvement of the result of the third author in \cite[Prop. 2.15]{sano}.

\begin{proposition} \label{propnorm}
Let $P$ be a finitely generated projective $G$-module, $r \in \ZZ_{\geq 0}$, and
$$\nu : (\bigwedge_{G/H}^r P^H) \otimes_{\ZZ} J_H \longrightarrow (\bigwedge_{G}^r P) \otimes_\ZZ \ZZ[H]/I(H)J$$
be the injection in Proposition \ref{propinj}.
For an element $a \in \bigwedge_G^r P$, the following are equivalent.
\begin{itemize}
\item[(i)]{$a \in \mathcal{J}\bigwedge_G^r P$,}
\item[(ii)]{$\mathcal{N}_H(a) \in \im (\nu)$,}
\item[(iii)]{$\Phi(a)\in \mathcal{J}$ for every $\Phi \in \bigwedge_G^r \Hom_G(P,\ZZ[G])$.}
\end{itemize}
Furthermore, if the above equivalent conditions are satisfied, then for every $\Phi \in \bigwedge_G^r \Hom_G(P,\ZZ[G])$ we have
$$\Phi(a)=\Phi^H (\nu^{-1}(\mathcal{N}_H(a))) \ in \ \mathcal{J}_H,$$
where we regard $\Phi^H(\nu^{-1}(\mathcal{N}_H(a))) \in \ZZ[G/H]\otimes_\ZZ J_H$ as an element of $\mathcal{J}_H$ via the isomorphism $\mathcal{J}_H \simeq \ZZ[G/H] \otimes_\ZZ J_H$ in Proposition \ref{propisom}.
\end{proposition}

\begin{proof}
By Swan's Theorem (see \cite[(32.1)]{curtisr}), for every prime $p$, $P_p$ is a free $\bbZ_p[G]$-module of rank, $d$ say, independent of $p$.
Considering locally, we may assume that $P$ is a
free $G$-module of rank $d$. We may assume $r\leq d$.
Clearly, (i) implies (iii). We shall show that (iii) implies (ii).
Suppose $\Phi(a) \in \mathcal{J}$ for all
$\Phi\in \bigwedge_G^r\Hom_G(P,\bbZ[G])$. Fix a basis $\{ b_1,\ldots,b_d \}$ of $P$. Write
$$a=\sum_{\mu\in \mathfrak{S}_{d,r}}x_\mu b_{\mu(1)}\wedge\cdots\wedge b_{\mu(r)},$$
with some $x_\mu \in \bbZ[G]$. For each $\mu$, by Proposition \ref{propformula}, we have
$$x_\mu =(b_{\mu(1)}^\ast \wedge\cdots\wedge b_{\mu(r)}^\ast)(a) \in \mathcal{J},$$
where $b_i^\ast \in \Hom_G(P,\bbZ[G])$ is the dual basis of $b_i$. For each $\tau \in G/H$, fix a lift $\widetilde \tau \in G$. Note that we have a direct sum decomposition
$$\mathcal{J} = \bigoplus_{\tau\in G/H}J \widetilde \tau.$$
Therefore, we can write each $x_\mu$ as follows:
$$x_\mu=\sum_{\tau \in G/H}y_{\tau\mu}\widetilde \tau,$$
where $y_{\tau\mu} \in J$. Hence we have
\begin{eqnarray}
\mathcal{N}_H(a) &=& \sum_{\sigma \in H} \sum_{\mu \in \mathfrak{S}_{d,r}} \sum_{\tau\in G/H} \sigma y_{\tau\mu} \widetilde \tau b_{\mu(1)}\wedge\cdots \wedge b_{\mu(r)}\otimes \sigma^{-1} \nonumber \\
&=& \sum_{\sigma \in H} \sum_{\mu \in \mathfrak{S}_{d,r}} \sum_{\tau\in G/H} \sigma \widetilde \tau b_{\mu(1)}\wedge\cdots \wedge b_{\mu(r)}\otimes \sigma^{-1}y_{\tau\mu} \nonumber \\
&=& \sum_{\mu\in\mathfrak{S}_{d,r}} \sum_{\tau\in G/H} \N_H  \widetilde \tau b_{\mu(1)}\wedge\cdots \wedge b_{\mu(r)}\otimes y_{\tau\mu} \nonumber \\
&\in & \N_H \bigwedge_G^r P \otimes_\bbZ J_H=\im (\nu) \nonumber
\end{eqnarray}
(see Remark \ref{remnorm}). This shows (ii). We also see by Remark \ref{remnorm} that
$$\nu^{-1}(\mathcal{N}_H(a))= \sum_{\mu \in \mathfrak{S}_{d,r}} \sum_{\tau\in G/H}\tau \N_H b_{\mu(1)}\wedge\cdots \wedge \N_H b_{\mu(r)}\otimes y_{\tau\mu} \in   (\bigwedge_{G/H}^rP^H)\otimes_\bbZ J_H.$$
Hence, by Proposition \ref{propphi}, we have
$$\Phi(a)=\Phi^H(\nu^{-1}(\mathcal{N}_H(a))) \quad in \quad \mathcal{J}_H$$
for all $\Phi\in\bigwedge_G^r\Hom_G(P,\bbZ[G])$.

Finally, we show that (ii) implies (i). Suppose $\mathcal{N}_H(a) \in \im (\nu)=(\N_H\bigwedge_G^r P) \otimes_\bbZ J_H$. As before, we write
$$a=\sum_{\mu\in\mathfrak{S}_{d,r}}\sum_{\tau \in G/H}y_{\tau\mu}\widetilde \tau b_{\mu(1)}\wedge \cdots \wedge b_{\mu(r)}$$
with $y_{\tau\mu} \in \bbZ[H]$. We have
$$\mathcal{N}_H(a)=\sum_{\sigma \in H} \sum_{\mu\in\mathfrak{S}_{d,r}} \sum_{\tau\in G/H} \sigma \widetilde \tau b_{\mu(1)}\wedge\cdots \wedge b_{\mu(r)}\otimes \sigma^{-1}y_{\tau\mu} \in (\N_H\bigwedge_G^r P) \otimes_\bbZ J_H.$$
Since $(\bigwedge_G^r P) \otimes_\bbZ \bbZ[H]/I(H)J \simeq \bigoplus_{\sigma,\mu,\tau} \bbZ[H]/I(H)J$ as abelian groups, we must have $y_{\tau\mu} \in J$.
This shows that $a \in \mathcal{J}\bigwedge_G^r P$.
\end{proof}

\section{Congruences for Rubin-Stark elements}\label{sano conj}

For a finite abelian extension $K/k$, and an intermediate field
$L$, a conjecture which describes a congruence relation between two
Rubin-Stark elements for $K/k$ and $L/k$ was
formulated by the third author in \cite[Conj. 3]{sano}.
Mazur and Rubin also formulated in \cite[Conj. 5.2]{MR2}
essentially the same conjecture.
In this section, we formulate a refined version (see Conjecture \ref{mrsconj}) of
these conjectures. We also recall a conjecture formulated by the first author, which was studied in \cite{Hay}, \cite{burns}, \cite{GK0}, \cite{GK}, \cite{tan}, and \cite{sano} (see Conjecture \ref{burnsconj}).
In \cite[Th. 3.15]{sano}, the third author proved a link between Conjecture \ref{mrsconj} and Conjecture \ref{burnsconj}.
We now improve the argument given there to show that Conjecture \ref{mrsconj} and Conjecture \ref{burnsconj} are in fact equivalent (see Theorem \ref{thmequiv}). Finally we prove that the natural equivariant leading term conjecture (Conjecture \ref{ltc}) implies both
Conjecture \ref{mrsconj} and
Conjecture \ref{burnsconj} (see Theorem \ref{ltcmrs}).

\subsection{The Rubin-Stark conjecture} \label{secrs}

In this subsection, we recall the formulation of the Rubin-Stark conjecture \cite[Conj. B$'$]{R}.

Let $K/k,G,S,T$ be as in \S \ref{secltc}, namely, $K/k$ is a finite abelian extension of global fields, $G$ is its Galois group, $S$ is a non-empty finite set of places of $k$ such that $S_{\infty}(k)\cup S_{\rm ram}(K/k) \subset S$, and $T$ is a finite set of places of $k$ which is disjoint from $S$. In this section, we assume that $\mathcal{O}_{K,S,T}^\times$ is $\ZZ$-torsion-free.

Following Rubin \cite[Hyp. 2.1]{R} we assume that $S$ satisfies
the following hypothesis with respect to some chosen integer $r$
with $0 \le r < |S|$:
there exists a subset $V \subset S$ of order $r$ such that each place in $V$ splits completely in $K/k$.

Recall that for any $\chi \in\widehat G$ we denote by $r_{\chi,S}$ the order of vanishing of $L_{k,S,T}(\chi,s)$ at $s=0$. We know by \cite[Chap. I, Prop. 3.4]{tate} that
\begin{eqnarray}
r_{\chi,S}=\dim_\CC (e_\chi \CC X_{K,S})=
\begin{cases}
|\{v\in S : \chi(G_v)=1\}| &\text{ if }\chi\neq 1,\\
|S|-1 &\text{ if }\chi=1.
\end{cases} \label{chi rank}
\end{eqnarray}
Therefore, the existence of $V$ ensures that $r \leq r_{\chi,S}$ for every $\chi$ and hence  the function
$s^{-r}L_{k,S,T}(\chi,s)$
is holomorphic at $s=0$. We define the `$r$-th order Stickelberger element' by
$$\theta_{K/k,S,T}^{(r)}:=\lim_{s\to 0}\sum_{\chi \in \widehat G}s^{-r}L_{k,S,T}(\chi^{-1},s)e_\chi \in \RR[G].$$
Note that the $0$-th order Stickelberger element $\theta_{K/k,S,T}^{(0)}(=\theta_{K/k,S,T}(0)$) is the usual Stickelberger element.

Recall that we have the regulator isomorphism
\begin{equation}\label{lambdadef} \lambda_{K,S} : \br \mathcal{O}^\times_{K,S,T}
\stackrel{\sim}{\longrightarrow} \br X_{K,S} \nonumber \end{equation}
defined by
\[\lambda_{K,S}(a) = -\sum_{w\in S_K}\mathrm{log}|a|_w
w.\]
This map $\lambda_{K,S}$ induces the isomorphism
$$\bigwedge_{\RR[G]}^r \RR \mathcal{O}_{K,S,T}^\times \stackrel{\sim}{\longrightarrow}
\bigwedge_{\RR[G]}^r \RR X_{K,S},$$
which we also denote by $\lambda_{K,S}$.
For each place $v\in S$, fix a place $w$ of $K$ lying above $v$.
Take any $v_0 \in S \setminus V$, and define the `($r$-th order) Rubin-Stark element'
\begin{equation}\label{nars}
\epsilon^V_{K/k,S,T} \in \bigwedge^r_{\RR [G]}\RR\mathcal{O}_{K,S,T}^\times
=
\RR\bigwedge^r_{G}\mathcal{O}_{K,S,T}^\times.\nonumber \end{equation}
by
\begin{equation}\label{RSE}
\lambda_{K,S}(\epsilon_{K/k,S,T}^V)=\theta_{K/k,S,T}^{(r)} \bigwedge_{v\in V}(w-w_0),
\end{equation}
where $\bigwedge_{v\in V}(w-w_0)$ is arranged by some chosen order of the elements in $V$.
One can show that the Rubin-Stark element $\epsilon_{K/k,S,T}^V$ does not depend on the choice of $v_0 \in S\setminus V$.

We consider the Rubin lattice
$$\bigcap_G^r \mathcal{O}_{K,S,T}^\times \subset \QQ \bigwedge_G^r \mathcal{O}_{K,S,T}^\times$$
(see Definition \ref{deflat}).
The Rubin-Stark conjecture claims

\begin{conjecture}[The Rubin-Stark conjecture for $(K/k,S,T,V)$] \label{rsconj}
One has
$$\epsilon_{K/k,S,T}^V \in \bigcap_G^r \mathcal{O}_{K,S,T}^\times.$$
\end{conjecture}


\begin{remark}
\begin{rm}
One can check that the above Rubin-Stark conjecture is equivalent to \cite[Conj. B$'$]{R} for the data $(K/k,S,T,V)$, and that our Rubin-Stark element $\epsilon_{K/k,S,T}^V$ coincides with the unique element predicted by \cite[Conj. B$'$]{R}. This shows, in particular, that the validity of the  conjecture does not depend on the choice of the places lying above $v\in S$ or on the ordering of the elements in $V$.
\end{rm}
\end{remark}

\begin{remark} \label{knownrs}
\begin{rm}
The Rubin-Stark conjecture for $(K/k,S,T,V)$ is known to be true in the following cases:
\begin{itemize}
\item[(i)]{$r=0$. In this case $\epsilon_{K/k,S,T}^\emptyset
=\theta_{K/k,S,T}^{(0)}=\theta_{K/k,S,T}(0) \in \RR[G]$ so the Rubin-Stark conjecture claims only that $\theta_{K/k,S,T}(0) \in \ZZ[G]$ which is a  celebrated result of Deligne-Ribet, Cassou-Nogu\`es, and Barsky.}
\item[(ii)]{$[K:k ]\leq 2$. This is due to Rubin \cite[Cor. 3.2 and Th. 3.5]{R}.}
\item[(iii)]{$K$ is an abelian extension over $\QQ$. This is due to the first author \cite[Th. A]{burns}.}
\item[(iv)]{$k$ is a global function field. This is due to the first author \cite[Th. A]{burns}.}
\end{itemize}
\end{rm}
\end{remark}

\subsection{Conventions for Rubin-Stark elements} \label{convention}
The notation $\epsilon_{K/k,S,T}^V$ has some ambiguities, since $\epsilon_{K/k,S,T}^V$ depends on the choice of the places lying above $v\in S$, and on the choice of the order of the elements in $V$. To avoid this ambiguity, we use the following convention:
when we consider the Rubin-Stark element $\epsilon_{K/k,S,T}^V$, we always fix a place $w$ of $K$ lying above each $v\in S$, and label the elements of $S$ as
$$S=\{ v_0, v_1,\ldots,v_n \}$$
so that $V=\{ v_1,\ldots , v_r\}$, and thus we fix the order of the elements in $V$. So, under this convention, the Rubin-Stark element $\epsilon_{K/k,S,T}^V$ is the element characterized by
$$\lambda_{K,S}(\epsilon_{K/k,S,T}^V)=\theta_{K/k,S,T}^{(r)}\bigwedge_{1\leq i \leq r}(w_i-w_0).$$

\subsection{Conjectures on Rubin-Stark elements} \label{secmrs}
In this subsection, we give a refinement of the conjecture
formulated by the third author \cite[Conj. 3]{sano}, and Mazur and Rubin \cite[Conj. 5.2]{MR2}.
Let $K/k,G,S,T$ be as before, and
we assume that, for a non-negative integer $r$, there exists a subset $V\subset S$ of order $r$ such that each place in $V$ splits completely in $K$. We fix a subgroup $H$ of $G$ for which, for some integer $r'$ with $r'\ge r$, there exists a subset $V' \subset S$ of order $r'$, which contains $V$, and satisfies that each place in $V'$ splits completely in the field $L := K^H$.

Following the convention in \S \ref{convention}, we fix, for each place $v\in S$, a place $w$ of $K$ lying above $v$, and label the elements of $S$ as $S=\{ v_0, \ldots,v_n\}$ so that $V=\{ v_1,\ldots,v_r\}$ and $V'=\{ v_1,\ldots,v_{r'} \}$. We consider the Rubin-Stark elements $\epsilon_{K/k,S,T}^V$ and $\epsilon_{L/k,S,T}^{V'}$ characterized by
$$\lambda_{K,S}(\epsilon_{K/k,S,T}^V)=\theta_{K/k,S,T}^{(r)} \bigwedge_{1\leq i\leq r}(w_i-w_0)$$
and
$$\lambda_{L,S}(\epsilon_{L/k,S,T}^{V'})=\theta_{L/k,S,T}^{(r')} \bigwedge_{1\leq i \leq r'}(w_i-w_0)$$
respectively, where we denote the place of $L$ lying under $w$ also by $w$.

For each integer $i$ with $1 \leq i \le n$ we write $G_i$ for the decomposition group of $v_i$ in $G$.
For any subgroup $U \subset G$, recall that the augmentation ideal of $\ZZ[U]$ is denoted by $I(U)$.
Put $\mathcal{I}_i:=I(G_i)\ZZ[G]$ and $I_{H}:=I(H)\ZZ[G]$.
We define
$$\Rec_i: \mathcal{O}_{L,S,T}^\times \longrightarrow
(\mathcal{I}_i)_{H}=\mathcal{I}_i/I_{H}\mathcal{I}_i$$
by
$$\Rec_i(a)=\sum_{\tau \in G/H}\tau^{-1}({\rm{rec}}_{w_i}(\tau a)-1).$$
Here, ${\rm{rec}}_{w_i}$ is the reciprocity map
$L_{w_{i}}^{\times} \rightarrow G_{i}$ at $w_i$.
Note that $\tau^{-1}({\rm{rec}}_{w_i}(\tau a)-1)$ is well-defined
for $\tau \in G/H$ in $(\mathcal{I}_i)_{H}$.

We put $W:=V'\setminus V=\{v_{r+1},\ldots,v_{r'}\}$.
We define an ideal $J_{W}$ of $\ZZ[H]$ by
\[J_{W} := \begin{cases} (\prod_{r<i \leq r'} I(G_{i})) \ZZ[H],
&\text{ if $W \neq \emptyset$},\\
 \ZZ[H], &\text{ if $W = \emptyset$,}\end{cases}\]
and put $(J_{W})_{H}:=J_{W}/I(H)J_{W}$.
We also define an ideal $\mathcal{J}_{W}$ of $\ZZ[G]$ by
\[\mathcal{J}_{W} := \begin{cases} \prod_{r<i\leq r'}\mathcal{I}_i,
&\text{ if $W \neq \emptyset$},\\
 \ZZ[G], &\text{ if $W = \emptyset$,}\end{cases}\]
and put $(\mathcal{J}_{W})_{H} :=\mathcal{J}_{W}/I_H\mathcal{J}_{W}$. Note that $\mathcal{J}_W=J_W\ZZ[G]$.
By Proposition \ref{propisom}, we have a natural isomorphism of $G/H$-modules
$\ZZ[G/H] \otimes_\ZZ (J_{W})_{H} \simeq (\mathcal{J}_{W})_{H}$.
We consider the graded $G/H$-algebra
$$\mathcal{Q}_W:=\bigoplus_{a_1, \ldots, a_{r'-r} \in \ZZ_{\geq 0}} (\mathcal{I}_{r+1}^{a_1}\cdots\mathcal{I}_{r'}^{a_{r'-r}})_H,$$
where
\[ (\mathcal{I}_{r+1}^{a_1}\cdots\mathcal{I}_{r'}^{a_{r'-r}})_H :=\mathcal{I}_{r+1}^{a_1}\cdots\mathcal{I}_{r'}^{a_{r'-r}}/I_H\mathcal{I}_{r+1}^{a_1}\cdots\mathcal{I}_{r'}^{a_{r'-r}},\]
and we define the $0$-th power of any ideal of $\ZZ[G]$ to be $\ZZ[G]$.

For any integer $i$ with $r<i\leq r'$ we regard $\Rec_i$ as an element of $\Hom_{G/H}(\mathcal{O}_{L,S,T}^\times,\mathcal{Q}_W)$ via the natural embedding $(\mathcal{I}_i)_H \hookrightarrow \mathcal{Q}_W$.
Then by the same method as in \cite[Prop. 2.7]{sano} (or \cite[Cor. 2.1]{MR2}),
one shows that $\bigwedge_{r<i\leq r'} \Rec_i \in \bigwedge_{G/H}^{r'-r} \Hom_{G/H}(\mathcal{O}_{L,S,T}^\times, \mathcal{Q}_W)$ induces
the map
\begin{equation} \label{Map1}
\bigcap_{G/H}^{r'} \mathcal{O}_{L,S,T}^\times \longrightarrow (\bigcap_{G/H}^{r} \mathcal{O}_{L,S,T}^\times) \otimes_{G/H} (\mathcal{J}_{W})_{H} \simeq (\bigcap_{G/H}^{r} \mathcal{O}_{L,S,T}^\times) \otimes_{\ZZ} (J_{W})_{H},
\end{equation}
which we denote by $\Rec_W$.

Following Definition \ref{defnorm}, we define
$$\mathcal{N}_H : \bigcap_{G}^r \mathcal{O}_{K,S,T}^\times \longrightarrow (\bigcap_{G}^r \mathcal{O}_{K,S,T}^\times) \otimes_{\ZZ} \ZZ[H]/I(H)J_{W}$$
by
$\mathcal{N}_H(a)=\sum_{\sigma \in H} \sigma a \otimes \sigma^{-1}$.

Note that since
$(\mathcal{O}_{K,S,T}^\times)^{H}=\mathcal{O}_{L,S,T}^\times$,
there is a natural injective homomorphism
$$
\nu : (\bigcap_{G/H}^{r} \mathcal{O}_{L,S,T}^\times) \otimes_{\ZZ} (J_{W})_{H}
\longrightarrow
(\bigcap_{G}^r \mathcal{O}_{K,S,T}^\times) \otimes_{\ZZ} \ZZ[H]/I(H)J_{W}
$$
by Proposition \ref{propinj}.

\vspace{5mm}

To state the following conjecture we assume the validity of the Rubin-Stark conjecture (Conjecture \ref{rsconj}) for both $(K/k,S,T,V)$ and $(L/k,S,T,V')$.

\begin{conjecture} [${\rm MRS}(K/L/k,S,T,V,V')$] \label{mrsconj}
The element $\mathcal{N}_H(\epsilon_{K/k,S,T}^V)$ belongs to
$\im (\nu)$,
and satisfies
$$\mathcal{N}_H(\epsilon_{K/k,S,T}^V)=(-1)^{r(r'-r)}\cdot
\nu(\Rec_W(\epsilon_{L/k,S,T}^{V'})).$$
\end{conjecture}

\begin{remark}\begin{rm} In this article we write `${\rm MRS}(K/L/k,S,T)$ is valid' to mean that the statement of Conjecture \ref{mrsconj} is valid for
all possible choices of  $V$ and $V'$.\end{rm}\end{remark}

\begin{remark}
\begin{rm} In \S\S \ref{dargro} and \ref{grost} we show that Conjecture \ref{mrsconj} is a natural refinement and generalization of a conjecture of Darmon from \cite{D} and of conjectures of Gross from \cite{Gp} and \cite{G}. In a subsequent article we will show that the known validity of Conjecture \ref{mrsconj} in the case $k = \QQ$ also implies a refinement of the main result of Solomon in \cite{solomon} concerning the `wild Euler system' that he constructs in loc. cit. \end{rm}
\end{remark}

\begin{remark}
\begin{rm}
\label{refined sano}
One has $I(G_{i})\ZZ[H] \subset I(H)$, so
$J_{W} \subseteq I(H)^e$ where $e := r'-r\ge 0$.
Thus there is a natural homomorphism
%
%
\[ (\bigcap^r_{G}\mathcal{O}^\times_{K,S,T})\otimes _\ZZ (J_{W})_{H}
\longrightarrow
(\bigcap^r_{G}\mathcal{O}^\times_{K,S,T})\otimes _\ZZ I(H)^{e}/I(H)^{e+1}
.  \]
Conjecture \ref{mrsconj} is therefore a strengthening of
the central conjecture of the third author in \cite[Conj. 3]{sano}
and of the conjecture formulated by Mazur and Rubin in \cite[Conj. 5.2]{MR2}, both of which claim only that the given equality is valid after projection to the group
$({\bigcap}^r_{G}\mathcal{O}^\times_{K,S,T})\otimes _\ZZ I(H)^{e}/I(H)^{e+1}$. This refinement is in the same spirit as  Tate's strengthening in \cite{tatebalt} of the `refined class number formula' formulated by Gross in  \cite{G}.
\end{rm}
\end{remark}

\begin{remark}
\begin{rm}
Note that, when $r=0$, following \cite[Def. 2.13]{sano} $\mathcal{N}_H$ would be defined to be the natural map $\ZZ[G]\rightarrow \ZZ[G]/I_H\mathcal{J}_{W}$, but this does not make any change because of the observation of Mazur and Rubin
in \cite[Lem. 5.6 (iv)]{MR2}.
Note also that, by Remark \ref{remnu}, the map ${\bf j}_{L/K}$ in \cite[Lem. 4.9]{MR2} (where our $K/L$ is denoted by $L/K$) is essentially the same as our homomorphism $\nu$.
We note that Mazur and Rubin does not use the fact
that ${\bf j}_{L/K}$ is injective, so the formulation of \cite[Conj. 3]{sano} is slightly stronger than the conjecture \cite[Conj. 5.2]{MR2}.
\end{rm}
\end{remark}

\vspace{5mm}

We next state a refinement of a conjecture that was formulated
by the first author in \cite{burns} (the original version of which has been studied in many subsequent articles of different authors including \cite{Hay}, \cite{GK0}, \cite{GK}, \cite{tan}, and \cite{sano}).

\begin{conjecture} [${\rm B}(K/L/k,S,T,V,V')$] \label{burnsconj}
For every $\Phi \in \bigwedge_G^r \Hom_G(\mathcal{O}_{K,S,T}^\times,\ZZ[G])$, we have
$$\Phi(\epsilon_{K/k,S,T}^V)\in \mathcal{J}_{W}$$
and an equality
$$\Phi(\epsilon_{K/k,S,T}^V)=(-1)^{r(r'-r)}\Phi^H(\Rec_W(\epsilon_{L/k,S,T}^{V'})) \ in \ (\mathcal{J}_{W})_H.$$
\end{conjecture}


In this article we improve an argument
of the third author in \cite{sano} to prove the following result.

\begin{theorem} \label{thmequiv} The conjectures {\rm MRS}$(K/L/k,S,T,V,V')$ and {\rm B}$(K/L/k,S,T,V,V')$ are equivalent.
\end{theorem}

The proof of this result will be given in \S \ref{secequiv}.

\subsection{}
As a preliminary step, we choose a useful representative of the complex
\[ D_{K,S,T}^\bullet:=R\Gamma_T((\mathcal{O}_{K,S})_\mathcal{W},\GG_m) \in D^{\rm p}(\ZZ[G]).\]
To do this we follow the method used in \cite[\S 7]{burns}.

Let $d$ be a sufficiently large integer, and $F$ be a free $G$-module of rank $d$ with basis $b=\{b_i\}_{1\leq i \leq d}$. We define a surjection
$$\pi: F\longrightarrow \mathcal{S}^{{\rm tr}}_{S,T}(\GG_{m/K})(=H^1(D_{K,S,T}^\bullet))$$
as follows. Recall that $S=\{v_0,\ldots,v_n \}$. Let $F_{\leq n}$ be a free $\ZZ[G]$-module
generated by $\{ b_i \}_{1\leq i \leq n}$.
First, choose a homomorphism
$$\pi_1 : F_{\leq n} \longrightarrow \mathcal{S}^{{\rm tr}}_{S,T}(\GG_{m/K})$$
such that the composition map
$$F_{\leq n} \stackrel{\pi_1}{\longrightarrow} \mathcal{S}^{{\rm tr}}_{S,T}(\GG_{m/K})\longrightarrow X_{K,S}$$
sends $b_i$ to $w_i-w_0$. (Such a homomorphism exists since $F_{\leq n}$ is free.)  Next, let $A$ denote the kernel of the composition map
$$\mathcal{S}^{{\rm tr}}_{S,T}(\GG_{m/K}) \longrightarrow X_{K,S} \longrightarrow Y_{K,S\setminus \{v_0 \}},$$
where the last map sends the places above $v_0$ to $0$. Since $d$ is sufficiently large, we can choose a surjection
$$\pi_2 : F_{>n} \longrightarrow A, $$
where $F_{>n}$ is the free $\ZZ[G]$-module generated by $\{ b_i\}_{n<i\leq d}$.
Define
$$\pi:=\pi_1\oplus\pi_2 : F=F_{\leq n}\oplus F_{>n} \longrightarrow
\mathcal{S}^{{\rm tr}}_{S,T}(\GG_{m/K}).$$
One can easily show that $\pi $ is surjective.

$D_{K,S,T}^\bullet$ defines a Yoneda extension class in
$\Ext_G^2(\mathcal{S}^{{\rm tr}}_{S,T}(\GG_{m/K}),\mathcal{O}_{K,S,T}^\times)$. Since $D_{K,S,T}^\bullet$ is perfect, this class is represented by an exact sequence of the following form:
\begin{eqnarray}
0\longrightarrow \mathcal{O}_{K,S,T}^\times \longrightarrow P \stackrel{\psi}{\longrightarrow} F \stackrel{\pi}{\longrightarrow} \mathcal{S}^{{\rm tr}}_{S,T}(\GG_{m/K}) \longrightarrow 0, \label{tateseq}
\end{eqnarray}
where $\pi$ is the above map and $P$ is a  cohomologically-trivial $G$-module. Since $\mathcal{O}_{K,S,T}^\times$ is $\ZZ$-torsion-free, it follows that $P$ is also $\ZZ$-torsion-free. Hence, $P$ is projective. Note that the complex
$$P \stackrel{\psi}{\longrightarrow} F,$$
where $P$ is placed in degree $0$, is quasi-isomorphic to $D_{K,S,T}^\bullet$. Hence we have an isomorphism
\begin{eqnarray}
{\det}_G(D_{K,S,T}^\bullet) \simeq {\det}_G(P) \otimes_G {\det}^{-1}_G(F). \label{detisom}
\end{eqnarray}

For each $1\leq i \leq d$, we define
$$\psi_i:=b_i^\ast \circ \psi \in \Hom_G(P,\ZZ[G]),$$
where $b_i^\ast \in \Hom_G(F,\ZZ[G])$ is the dual basis of $b_i\in F$.

\subsection{The equivalence of Conjectures \ref{mrsconj} and \ref{burnsconj}} \label{secequiv}
In this subsection, we prove Theorem \ref{thmequiv}. The existence of the exact sequence (\ref{tateseq}) is essential in the proof.

\begin{proof}[Proof of Theorem \ref{thmequiv}]
We regard $\mathcal{O}_{K,S,T}^\times \subset P$ by the exact sequence (\ref{tateseq}). Note that, by Lemma \ref{lemr} (i),  the map
$$\bigwedge_G^r \Hom_G(P,\ZZ[G]) \longrightarrow \bigwedge_G^r \Hom_G(\mathcal{O}_{K,S,T}^\times, \ZZ[G])$$
is surjective, since $P/\mathcal{O}_{K,S,T}^\times \simeq \im (\psi) \subset F$ is $\ZZ$-torsion-free.
By Proposition \ref{propnorm}, Conjecture \ref{mrsconj} implies Conjecture \ref{burnsconj}. The converse follows from Proposition \ref{propnorm}, Proposition \ref{thminj}, and Remark \ref{remr}.
\end{proof}

\subsection{The leading term conjecture implies the Rubin-Stark conjecture} \label{secltcrs}

The following result was first proved by the first author in \cite[Cor. 4.1]{burns} but the proof given here is very much simpler than that given in loc. cit.

\begin{theorem}\label{ltcrs}
{\rm ${\rm LTC}(K/k)$} implies the Rubin-Stark conjecture for both
$(K/k,S,T,V)$ and $(L/k,S,T,V')$.
\end{theorem}

\begin{proof} 

Assume that ${\rm LTC}(K/k)$ is valid so the zeta element $z_{K/k,S,T}$ is a $\ZZ[G]$-basis of ${\det}_G(D_{K,S,T}^\bullet)$. In this case one also knows that $P$ must be free of rank $d$ and we define $z_b \in \bigwedge_G^d P$ to be the element corresponding to the zeta element $z_{K/k,S,T} \in {\det}_G(D_{K,S,T}^\bullet)$ via the isomorphism
$$\bigwedge_G^d P \stackrel{\sim}{\longrightarrow} \bigwedge_G^d P \otimes \bigwedge_G^d\Hom_G(F,\ZZ[G]) \simeq {\det}_G(D_{K,S,T}^\bullet),$$
where the first isomorphism  is defined by
$$a \mapsto a \otimes \bigwedge_{1\leq i\leq d } b_i^\ast,$$
and the second isomorphism is given by (\ref{detisom}).

Then Theorem \ref{ltcrs} follows immediately from the next theorem (see also
Corollary \ref{remzetars} below for $(L/k,S,T,V')$).
\end{proof}

\begin{remark}\label{alb} In \cite{vallieres} Valli\`eres closely follows the proof of \cite[Cor. 4.1]{burns} to show that Conjecture \ref{ltc0} (and hence also LTC$(K/k$) by virtue of Proposition \ref{ltc prop2}) implies the extension of the Rubin-Stark Conjecture formulated by Emmons and Popescu in \cite{EP}. The arguments used here can be used to show that LTC$(K/k$) implies a refinement of the main result of Valli\`eres, and hence also of the conjecture of Emmons and Popescu, that is in the spirit of Theorem \ref{MT2}. This result is to be explained in forthcoming work of Livingstone-Boomla. \end{remark}

The following theorem was essentially obtained in \cite{burns} by the first author.
This theorem describes the Rubin-Stark element in terms of
the zeta elements.
It is a key to prove Theorem \ref{ltcrs}, and also
plays important roles in the proofs of Theorem \ref{ltcmrs} and Theorem \ref{ltcfit} given below.

\begin{theorem} \label{zetars}
Assume that ${\rm LTC}(K/k)$ holds. Then, regarding $\mathcal{O}_{K,S,T}^\times$ as a submodule of $P$, one has
$$(\bigwedge_{r <i\leq d}\psi_i)(z_b) \in \bigcap_G^r \mathcal{O}_{K,S,T}^\times (\subset \bigwedge_G^r P)$$
(see Lemma \ref{lemr} (ii)) and also
$$(-1)^{r(d-r)}(\bigwedge_{r <i\leq d}\psi_i)(z_b)=\epsilon_{K/k,S,T}^V.$$
\end{theorem}

\begin{proof}
Take any $\chi \in \widehat G$. Recall from (\ref{chi rank}) that
\[ r_{\chi,S}=\dim_\CC(e_\chi\CC X_{K,S})=\dim_\CC(e_\chi\CC \mathcal{O}_{K,S,T}^\times)\]
(the last equality follows from $\CC \mathcal{O}_{K,S,T}^\times \simeq \CC X_{K,S}$). Consider the map
$$\Psi:=\bigoplus_{r<i \leq d}\psi_i : e_\chi \bbC P \longrightarrow e_\chi \bbC[G]^{\oplus (d-r)}.$$
This map is surjective if and only if $r_{\chi,S}=r$. Indeed, if $r_{\chi,S}=r$, then $\{ e_\chi(w_i-w_0) \}_{1\leq i \leq r}$ is a $\CC$-basis of $e_\chi \bbC X_{K,S}$,
so $e_\chi\bbC\im (\psi)=e_\chi\bbC \ker (\pi)=\bigoplus_{r<i\leq d}e_\chi\bbC[G]b_i$. In this case, $\Psi$ is surjective.
If $r_{\chi,S}>r$, then $\dim_\bbC(e_\chi\bbC\im(\psi))=d-r_{\chi,S}<d-r$, so $\Psi$ is not surjective. Applying Lemma \ref{leme}, we have
$$e_\chi (\bigwedge_{r<i\leq d}\psi_i)(z_b)
\begin{cases}
\in e_\chi\bbC\bigwedge_G^r\mathcal{O}_{K,S,T}^\times, &\text{if $r_{\chi,S}=r$,} \\
=0, &\text{if $r_{\chi,S}>r$.}
\end{cases}
$$
From this and Lemma \ref{lemr} (ii), we have
$$(\bigwedge_{r< i \leq d}\psi_i)(z_b)\in (\bbQ \bigwedge_G^r\mathcal{O}_{K,S,T}^\times)\cap \bigwedge_G^rP=\bigcap_G^r\mathcal{O}_{K,S,T}^\times.$$
By Lemma \ref{remarkpsi} and the definition of $z_b$, we have
$$\lambda_{K,S}((-1)^{r(d-r)}(\bigwedge_{r<i\leq d}\psi_i)(z_b))=\theta_{K/k,S,T}^{(r)}\bigwedge_{1\leq i \leq r}(w_i-w_0).$$
By the characterization of the
Rubin-Stark element, we have
$$(-1)^{r(d-r)}(\bigwedge_{r <i\leq d}\psi_i)(z_b)=\epsilon_{K/k,S,T}^V.$$
This completes the proof.
\end{proof}

By the same argument as above, one obtains the following result.

\begin{corollary} \label{remzetars}
Assume that ${\rm LTC}(K/k)$ holds. Then we have an equality
$$(-1)^{r'(d-r')}(\bigwedge_{r' <i\leq d}\psi_i^H)(\N_H^d z_b)=
\epsilon_{L/k,S,T}^{V'}$$
in $\bigcap_{G/H}^{r'}\mathcal{O}_{L,S,T}^\times$.
\end{corollary}

\subsection{The leading term conjecture implies Conjecture \ref{mrsconj}}

In this subsection we prove the following result.

\begin{theorem} \label{ltcmrs}
{\rm ${\rm LTC}(K/k)$} implies {\rm ${\rm MRS}(K/L/k,S,T,V,V')$}.
\end{theorem}

By Remark \ref{knownltc}, this directly implies the following result.

\begin{corollary} \label{cormrs}
{\rm ${\rm MRS}(K/L/k,S,T,V,V')$} is valid if $K$ is an abelian extension over $\QQ$ or if $k$ is a function field.
\end{corollary}

\begin{remark}
\begin{rm}
Theorem \ref{ltcmrs} is an improvement of the main result in \cite[Th. 3.22]{sano} by the third author, which asserts that under some hypotheses ${\rm LTC}(K/k)$ implies most of Conjecture \ref{mrsconj}.
In \cite[Th. 3.1]{burns}, the first author proved that ${\rm LTC}(K/k)$ implies most of Conjecture \ref{burnsconj}. Since we know by Theorem \ref{thmequiv} that Conjecture \ref{mrsconj} and Conjecture \ref{burnsconj} are equivalent, Theorem \ref{ltcmrs} is also an improvement of \cite[Th. 3.1]{burns}.
\end{rm}
\end{remark}

\begin{remark}
\begin{rm}
In \cite[\S 4]{sano}, by using a weak version of Corollary \ref{cormrs}, the third author gave another proof of the `except $2$-part' of Darmon's conjecture on cyclotomic units \cite{D}, which was first proved by Mazur and Rubin in \cite{MR} via Kolyvagin systems.
In \S \ref{dargro}, we shall use Corollary \ref{cormrs} to give a full proof of a refined version of Darmon's conjecture, and also give a new evidence for Gross's conjecture on tori \cite{G}, which was studied by Hayward \cite{Hay}, Greither and Ku\v cera \cite{GK0}, \cite{GK}.
\end{rm}
\end{remark}

We prove Theorem \ref{ltcmrs} after proving some lemmas. The following lemma is a restatement of \cite[Lem. 7.4]{burns}.

\begin{lemma} \label{lempsi}
If $1\leq i \leq n$, then we have an inclusion
$$\im (\psi_i) \subset \mathcal{I}_i.$$
In particular, $\psi_i=0$ for $1\leq i \leq r$.
\end{lemma}

\begin{proof}
Take any $a\in P$. Write
$$\psi(a)=\sum_{j=1}^d x_j b_j$$
with some $x_j \in \ZZ[G]$. For each $i$ with $1\leq i \leq n$, we show that $x_i \in \mathcal{I}_i$, or equivalently, ${\N_{G_i}} x_i=0$. Noting that $F^{G_i}$ is a free $G/G_i$-module with basis $\{ {\N_{G_i}}b_j\}_{1\leq j \leq d}$, it is sufficient to show that
$$\sum_{j=1}^d {\N_{G_i}}x_j b_j \in \langle{\N_{G_i}}b_j  :  1\leq j \leq d, j\neq i\rangle_{G/G_i}.$$
The left hand side is equal to $\psi({\N_{G_i}}a)$. By the exact sequence (\ref{tateseq}), this is contained in $\ker (\pi |_{F^{G_i}})$. Note that we have a natural isomorphism
$${\N_{G_i}}X_{K,S} \simeq X_{K^{G_i},S}.$$
Since $v_i$ splits completely in $K^{G_i}$, the $G/G_i$-submodule of ${\N_{G_i}}X_{K,S}$ generated by ${\N_{G_i}}(w_i-w_0)$ is isomorphic to $\ZZ [G/G_i]$. This shows that
$$\ker(\pi |_{F^{G_i}}) \subset \langle{\N_{G_i}}b_j  :  1\leq j \leq d, j\neq i\rangle_{G/G_i}.$$
\end{proof}

For each integer $i$ with $r<i\leq r'$,  we define a map
$$\widetilde \Rec_i : P^H \longrightarrow (\mathcal{I}_i)_H$$
as follows. For $a\in P^H$, take $\widetilde a \in P$ such that $\N_H\widetilde a =a$ (this is possible since $P$ is cohomologically-trivial). Define
$$\widetilde \Rec_i(a):={ \psi_i(\widetilde a)} \ {\rm mod} \
I_H\mathcal{I}_i \
\in (\mathcal{I}_i)_H.$$
(Note that $\im (\psi_i) \subset \mathcal{I}_i$ by Lemma \ref{lempsi}.)
One can easily check that this is well-defined.

\begin{lemma} \label{lem1}
On $\mathcal{O}_{L,S,T}^\times$, which we regard as a submodule of $P^H$,
$\widetilde\Rec_i$ coincides with the map $\Rec_i$. In particular, by the construction of (\ref{extmap2}), we can extend the map
$$\Rec_W : \bigcap_{G/H}^{r'}\mathcal{O}_{L,S,T}^\times \longrightarrow (\bigcap_{G/H}^r \mathcal{O}_{L,S,T}^\times) \otimes_\ZZ (J_{W})_H$$
to
$$\widetilde \Rec_W :=\bigwedge_{r <i\leq r'}\widetilde \Rec_i : \bigwedge_{G/H}^{r'} P^H \longrightarrow (\bigwedge_{G/H}^r P^H) \otimes_{G/H} (\mathcal{J}_{W})_H \simeq(\bigwedge_{G/H}^r P^H) \otimes_\ZZ (J_{W})_H \ .$$
\end{lemma}

\begin{proof}
The proof is essentially the same as \cite[Prop. 10.1]{burns}. For $a \in \mathcal{O}_{L,S,T}^\times$, take $\widetilde a \in P$ such that $\N_H \widetilde a =a$ in $P^H$. For each $\tau \in G/H$, fix a lift $\widetilde \tau \in G$.
Regard $F$ as the free $H$-module with basis $\{ \widetilde \tau b_i \}_{i,\tau}$. By the argument in the proof of \cite[Lem. 10.3]{burns}, we know for every $r<i\leq r'$ that
$$(\widetilde \tau b_i)^\ast \circ \psi (\widetilde a) = {\rm rec}_{\widetilde \tau w_i}(a)-1={\rm rec}_{w_i}(\tau^{-1}a) -1,$$
where $(\widetilde \tau b_i)^\ast \in \Hom_{H}(F,\ZZ[H])$ is the dual basis of $F$ as a free $H$-module. We easily see that
$$\widetilde \Rec_i(a)=\psi_i(\widetilde a)=\sum_{\tau \in G/H} \widetilde \tau ( (\widetilde \tau b_i)^\ast \circ \psi (\widetilde a)).$$
Hence we have
$$\widetilde \Rec_i(a)= \sum_{\tau \in G/H} \widetilde \tau ({\rm rec}_{w_i}(\tau^{-1}a)-1)=\Rec_i(a).$$
\end{proof}

Note that, by Lemma \ref{lempsi}, $\bigwedge_{r < i\leq d}\psi_i$ defines a map
$$\bigwedge_G^dP \longrightarrow \mathcal{J}_{W}\bigwedge_G^r P.$$
Let $\nu$ be the injection
$$\nu : (\bigwedge_{G/H}^r P^H) \otimes_\ZZ (J_{W})_H \longrightarrow (\bigwedge_G^r P) \otimes_\ZZ \ZZ[H]/I(H)J_{W} $$
in Proposition \ref{propinj}.
By Proposition \ref{propnorm}, we have
$$\mathcal{N}_H (\mathcal{J}_{W}\bigwedge_G^r P) \subset \im (\nu),$$
so we can define a map
$$\nu^{-1}\circ \mathcal{N}_H : \mathcal{J}_{W}\bigwedge_G^r P \longrightarrow (\bigwedge_{G/H}^rP^H )\otimes_\ZZ (J_{W})_H. $$

\begin{lemma} \label{lemcomm}
We have the following commutative diagram:
\[\xymatrix{
{\bigwedge_G^d P} \ar[d]_{\N_H^d} \ar[r]^{} &  {\mathcal{J}_{W}\bigwedge_G^rP} \ar[d]^{\nu^{-1}\circ \mathcal{N}_H} \\
{\bigwedge_{G/H}^dP^H} \ar[r]^{} & {(\bigwedge_{G/H}^rP^H) \otimes_\bbZ (J_{W})_H,}  \\
}\]
where the top arrow is $(-1)^{r(d-r)}\bigwedge_{r <i \leq d}\psi_i$, and
the bottom arrow is the composition of $(-1)^{r(r'-r)}\widetilde\Rec_W$ and
$(-1)^{r'(d-r')}\bigwedge_{r'<i \leq d}\psi_i^H$.
\end{lemma}

\begin{proof}
We can prove this lemma by explicit computations,
using Proposition \ref{propformula}, Proposition \ref{propphi}, and Remark \ref{remnorm}.
\end{proof}

\begin{proof}[Proof of Theorem \ref{ltcmrs}]
By Remark \ref{remr} we may compute in
$(\bigwedge_{G/H}^{r}P^H) \otimes_\ZZ (J_W)_H$.
Using Corollary \ref{remzetars},  Lemma \ref{lem1},
Lemma \ref{lemcomm}, and Theorem \ref{zetars} in this order, we compute
\begin{eqnarray*}
(-1)^{r(r'-r)}
\Rec_W(\epsilon_{L/k,S,T}^{V'}) &=&
(-1)^{r(r'-r)}
\widetilde\Rec_W((-1)^{r'(d-r')}(\bigwedge_{r' <i\leq d}\psi_i^H)(\N_H^d z_b)) \\
&=& (-1)^{r(d-r)} \nu^{-1}(\mathcal{N}_H((\bigwedge_{r <i\leq d}\psi_i)(z_b))) \\
&=&   \nu^{-1}(\mathcal{N}_H(\epsilon_{K/k,S,T}^{V})).
\end{eqnarray*}
This completes the proof of Theorem \ref{ltcmrs}.
\end{proof}

\section{Conjectures of Darmon and of Gross} \label{dargro}

In this section we use Corollary \ref{cormrs} to prove a refined version of the conjecture formulated by Darmon in \cite{D} and to obtain important new evidence for a refined version of the `conjecture for tori' formulated by Gross in \cite{G}. 


\subsection{Darmon's Conjecture}\label{darmon sect}




We formulate a slightly modified and refined version of Damon's conjecture (\cite{D},\cite{MR}).

Let $L$ be a real quadratic field. Let $f$ be the conductor of $L$. Let $\chi$ be the Dirichlet character defined by
$$\chi : (\ZZ/f\ZZ)^\times=\Gal(\QQ(\mu_f)/\QQ) \longrightarrow \Gal(L/\QQ) \simeq \{ \pm 1\},$$
where the first map is the restriction map.
Fix a square-free positive integer $n$ which is coprime to $f$, and  let $K$ be the maximal real subfield of $L(\mu_n)$. Set $G:=\Gal(K/\Q)$ and $H:=\Gal(K/L)$. Put
$n_{\pm}:=\prod_{\ell | n, \chi(\ell)=\pm1}\ell$, and $\nu_{\pm}:=|\{ \ell | n_{\pm} \}|$ (in this section, $\ell $ always denotes a prime number).
We fix an embedding $\overline \Q \hookrightarrow \CC$. Define a cyclotomic unit by
$$\beta_n:=\N_{L(\mu_n)/K}(\prod_{\sigma \in \Gal(\Q(\mu_{nf})/\Q(\mu_n))}\sigma(1-\zeta_{nf})^{\chi(\sigma)}) \in K^\times,$$
where $\zeta_{nf}=e^{\frac{2\pi i}{nf}}$. Let $\tau$ be the generator of $G/H=\Gal(L/\Q)$.
Write $n_+=\ell_1 \cdots \ell_{\nu_+}$.
Note that $(1-\tau)\mathcal{O}_{L}[1/n]^\times$ is a free abelian group of rank $\nu_++1$ (see \cite[Lem. 3.2 (ii)]{MR}).
Take $u_0,\ldots,u_{\nu_+} \in \mathcal{O}_L[1/n]^\times$ so that $\{u_0^{1-\tau}, \ldots, u_{\nu_+}^{1-\tau}\}$ is a basis of $(1-\tau)\mathcal{O}_L[1/n]^\times$ and that $\det(\log|u_i^{1-\tau}|_{\lambda_j})_{0\leq i,j\leq \nu_+}>0$,
where each $\lambda_j$ ($1\leq j \leq \nu_+$) is a (fixed) place of $L$ lying above $\ell_j$, and $\lambda_0$ is the infinite place of $L$ determined by the embedding
$\overline \Q \hookrightarrow \CC$ fixed above. Define
$$R_n:=(\bigwedge_{1\leq i \leq \nu_+}({\rm rec}_{\lambda_i}(\cdot)-1))(u_0^{1-\tau}\wedge\cdots\wedge u_{\nu_+}^{1-\tau}) \in L^\times \otimes_\ZZ (J_{n_+})_H,$$
where
\[J_{n_+} := \begin{cases} (\prod_{i= 1}^{\nu_+}I(G_{\ell_i}),
&\text{ if $\nu_+ \neq 0$},\\
 \ZZ[H], &\text{ if $\nu_+=0$,}\end{cases}\]
where $G_{\ell_i}$ is the decomposition group of $\ell_i$ in $G$ (note that since $\ell_i$ splits in $L$, we have $G_{\ell_i} \subset H$), and $(J_{n_+})_H:=J_{n_+}/I(H)J_{n_+}$.
We set $h_n:=|{\rm{Pic}}(\mathcal{O}_L[\frac1n])|.$
For any element $a \in K^\times$, following Definition \ref{defnorm} we define
$$\mathcal{N}_H(a) := \sum_{\sigma\in H} \sigma a \otimes \sigma^{-1} \in K^\times \otimes_\ZZ \ZZ[H]/I(H)J_{n_+}.$$

Note that, since $K^\times/L^\times$ is $\ZZ$-torsion-free,
the natural map
$$(L^\times/\{\pm 1\})\otimes_\ZZ (J_{n_+})_H \longrightarrow (K^\times/\{\pm 1\})\otimes_\ZZ \ZZ[H]I(H)J_{n_+}$$
is injective.

Our refined Darmon's conjecture is formulated as follows.

\begin{theorem} \label{darconj} One has
$$\mathcal{N}_H(\beta_n)=-2^{\nu_-}h_nR_n \quad in \quad (L^\times/\{\pm 1\})\otimes_\ZZ (J_{n_+})_H.$$
\end{theorem}

\begin{remark}
Let $I_n$ be the augmentation ideal of $\ZZ[\Gal(L(\mu_n)/L)]$. Note that there is a natural isomorphism
$$I_n^{\nu_+}/I_n^{\nu_++1}\otimes_\ZZ \ZZ[\frac12] \stackrel{\sim}{\longrightarrow} I(H)^{\nu_+}/I(H)^{\nu_++1} \otimes_\ZZ \ZZ[\frac12].$$
It is not difficult to see that the following statement is equivalent to \cite[Th. 3.9]{MR}:
$$\mathcal{N}_H(\beta_n)=-2^{\nu_-}h_nR_n \quad in \quad (L^\times/\{\pm 1\})\otimes_\ZZ I(H)^{\nu_+}/I(H)^{\nu_++1} \otimes_\ZZ \ZZ[\frac12]$$
(see \cite[Lem. 4.7]{sano}). Since there is a natural map $(J_{n_+})_H \longrightarrow I(H)^{\nu_+}/I(H)^{\nu_++1}$, Theorem \ref{darconj} refines \cite[Th. 3.9]{MR}.
Note also that, in the original Darmon's conjecture, the cyclotomic unit is defined by
$$\alpha_n:=\prod_{\sigma \in \Gal(\Q(\mu_{nf})/\Q(\mu_n))}\sigma(1-\zeta_{nf})^{\chi(\sigma)},$$
whereas our cyclotomic unit is $\beta_n=\N_{L(\mu_n)/K}(\alpha_n)$. Since cyclotomic units, as Stark elements, lie in real fields, so it is natural to
consider $\beta_n$. Thus, modifying the original Darmon's conjecture in the `$2$-part', we obtained Theorem \ref{darconj}, which does not
exclude the `$2$-part'.

\end{remark}

\begin{proof}[Proof of Theorem \ref{darconj}]
We show that Darmon's conjecture is a consequence of Conjecture \ref{mrsconj}, and use Corollary \ref{cormrs} to prove it. We fit notation in this section into that in \S \ref{sano conj}.  Set $S:=\{ \infty \}\cup\{ \ell |nf \}$.
Take a prime $v_0$ of $\Q$, which divides $f$. We denote by $w_1$ the infinite place of $K$ (and also $L$) which corresponds to the fixed embedding $\QQ \hookrightarrow \CC$. For $2\leq i \leq \nu_++1$, set $w_i:=\lambda_{i-1} $.
Let $T$ be a finite set of primes that is disjoint from $S$ and satisfying that $\mathcal{O}_{K,S,T}^\times$ is $\ZZ$-torsion-free. (In the sequel, we refer such a set of primes as `$T$'.) Since $K$ and $L$ are abelian over $\Q$, the Rubin-Stark conjecture
for $K/\Q$ and $L/\Q$ holds (see Remark \ref{knownrs} (iii)). Set $V:=\{\infty\}$ and $V':=\{ \infty, \ell_1,\ldots,\ell_{\nu_+}\}$. We denote
$\epsilon_{K,T}=\epsilon_{K/\QQ,S,T}^V \in \mathcal{O}_{K,S,T}^\times$ and
$\epsilon_{L,T}=\epsilon_{L/\QQ,S,T}^{V'} \in \bigcap_{G/H}^{\nu_++1}\mathcal{O}_{L,S,T}^\times$ for the Rubin-Stark elements, characterized by
$$\lambda_{K,S}(\epsilon_{K,T})=\theta_{K/\Q,S,T}^{(1)}(w_1-w_0),$$
$$\lambda_{L,S}(\epsilon_{L,T})=
\theta_{L/\Q,S,T}^{(\nu_++1)}\bigwedge_{1\leq i \leq \nu_++1}(w_i-w_0).$$
We take $\mathscr{T}$, a finite family of `$T$', such that
$$\sum_{T\in \mathscr{T}}a_T\delta_T=2$$
for some $a_T \in \ZZ[G]$, where $\delta_T:=\prod_{\ell \in T}(1-\ell {\rm{Fr}}_\ell^{-1})$ (see \cite[Chap. IV, Lem. 1.1]{tate}). By \cite[Lem. 4.6]{sano}, we have
$$(1-\tau)\sum_{T\in \mathscr{T}}a_T \epsilon_{K,T}= \beta_n \quad in \quad K^\times/\{\pm 1\},$$
(where $\tau \in \Gal(L/\QQ)$ is regarded as an element of $\Gal(K/\QQ(\mu_n)^+)$) and
$$(1-\tau)\sum_{T\in \mathscr{T}}a_T \epsilon_{L,T}=(-1)^{\nu_++1}2^{\nu_-}h_n(1-\tau) u_0\wedge\cdots\wedge u_{\nu_+} \quad in \quad \Q \bigwedge_{G/H}^{\nu_++1}\mathcal{O}_{L,S}^\times.$$
As in \S \ref{secmrs}, for $1<i \leq \nu_++1$ we denote by $\Rec_i$ the homomorphism
$$\Rec_i : \mathcal{O}_{L,S,T}^\times \longrightarrow (\mathcal{J}_{n_+})_H =J_{n_+} \ZZ[G]/I_HJ_{n_+} \ZZ[G]$$
defined by
$$\Rec_i(a)={\rm rec}_{\lambda_{i-1}}(a)-1 + \tau ({\rm rec}_{\lambda_{i-1}}(\tau a)-1).$$
$\bigwedge_{1<i \leq \nu_++1}\Rec_i$ induces a homomorphism
$$\bigcap_{G/H}^{\nu_++1}\mathcal{O}_{L,S,T}^\times \longrightarrow  (\bigcap_{G/H}^1 \mathcal{O}_{L,S,T}^\times) \otimes_\ZZ (J_{n_+})_H= \mathcal{O}_{L,S,T}^\times \otimes_\ZZ (J_{n_+})_H,$$
which we denote by $\Rec_{n_+}$.
We compute
\begin{eqnarray}
(1-\tau)\sum_{T\in \mathscr{T}}a_T\Rec_{n_+}(\epsilon_{L,T})
&=& \sum_{T\in \mathscr{T}}\Rec_{n_+}(a_T(1-\tau)\epsilon_{L,T}) \nonumber \\
&=& \sum_{T\in \mathscr{T}}(\bigwedge_{1\leq i \leq \nu_+}({\rm rec}_{\lambda_i}(\cdot)-1)))((1-\tau)^{\nu_++1}a_T\epsilon_{L,T}) \nonumber \\
&=& (\bigwedge_{1\leq i \leq \nu_+}({\rm rec}_{\lambda_i}(\cdot)-1))((-1)^{\nu_++1}2^{\nu_-}h_n u_0^{1-\tau}\wedge\cdots\wedge u_{\nu_+}^{1-\tau}) \nonumber \\
&=& (-1)^{\nu_++1}2^{\nu_-}h_n R_n. \nonumber
\end{eqnarray}
By Corollary \ref{cormrs}, we have
$$\mathcal{N}_H(\epsilon_{K,T})=(-1)^{\nu_+} \Rec_{n_+}(\epsilon_{L,T})$$
(note that the map $\nu$ in Conjecture \ref{mrsconj} is the natural inclusion map in this case.)
Hence, we have
\begin{eqnarray}
\mathcal{N}_H(\beta_n)&=&(1-\tau)\sum_{T\in\mathscr{T}}a_T\mathcal{N}_H(\epsilon_{K,T}) \nonumber \\
&=&(-1)^{\nu_+}(1-\tau)\sum_{T\in\mathscr{T}}a_T\Rec_{n_+}(\epsilon_{L,T}) \nonumber \\
&=&-2^{\nu_-}h_n R_n, \nonumber
\end{eqnarray}
as required.
\end{proof}

\subsection{Gross's conjecture for tori} \label{gro}
In this section we use Corollary \ref{cormrs} to obtain some new evidence in support of the `conjecture for tori' formulated by Gross in \cite{G}.

We review the formulation of Gross's conjecture for tori. We follow \cite[Conj. 7.4]{Hay}. Let $k$ be a global field, and $L/k$ be a quadratic extension. Let $\widetilde L/k$ be a finite abelian extension,
which is disjoint to $L$, and set $K:=L  \widetilde L$. Set $G:=\Gal(K/k)$, and $H:=\Gal(K/L)=\Gal(\widetilde L/k)$. Let $\tau$ be the generator of $G/H=\Gal(L/k)$.
Let $S$ be a non-empty finite set of places of $k$ such that $S_\infty(k) \cup S_{\rm ram}(K/k)\subset S$. Let $T$ be a finite set of places of $k$ that is disjoint from $S$ and satisfies that $\mathcal{O}_{K,S,T}^\times$ is $\ZZ$-torsion-free. Let $v_1,\ldots,v_{r'}$ be all places in $S$ which split in $L$. We assume $r' <|S|$.
Then, by \cite[Lem. 3.4 (i)]{R}, we see that $h_{k,S,T}:=|{\rm Cl}_S^T(k)|$ divides $h_{L,S,T}:=|{\rm Cl}_S^T(L)|$.
Take $u_1,\ldots,u_{r'}\in \mathcal{O}_{L,S,T}^\times$ such that $\{ u_1^{1-\tau},\ldots,u_{r'}^{1-\tau}\}$ is a basis of $(1-\tau)\mathcal{O}_{L,S,T}^\times$,
which is isomorphic to $ \ZZ^{\oplus r'}$, and $\det(-\log|u_i^{1-\tau}|_{w_j})_{1\leq i,j\leq r'}>0$,
where $w_j$ is a (fixed) place of $L$ lying above $v_j$. Put $W:=\{ v_1,\ldots,v_{r'} \}$. As in \S \ref{secmrs}, we define
\[J_{W} := \begin{cases} (\prod_{0<i \leq r'} I(G_{i})) \ZZ[H],
&\text{ if $W \neq \emptyset$},\\
 \ZZ[H], &\text{ if $W = \emptyset$,}\end{cases}\]
 where $G_i \subset H$ denotes the decomposition group of $v_i$, and $I(G_i)$ is the augmentation ideal of $\ZZ[G_i]$. Set
$$R_{S,T}:=\det({\rm{rec}}_{w_j}(u_i^{1-\tau})-1)_{1\leq i,j \leq r'} \in (J_{W})_H.$$
Let $\chi$ be the non-trivial character of $G/H$. The map
$$\ZZ[G]=\ZZ[H\times G/H] \longrightarrow \ZZ[H]$$
induced by $\chi$ is also denoted by $\chi$.

Gross's tori conjecture is formulated as follows.

\begin{conjecture} \label{tori conj}
$$\chi(\theta_{K/k,S,T}(0) )= 2^{|S|-1-r'}\frac{h_{L,S,T}}{h_{k,S,T}}R_{S,T} \quad in \quad (J_{W})_H. $$
\end{conjecture}

\begin{remark}
The statement that the equality of Conjecture \ref{tori conj} holds in $\ZZ[H]/I(H)^{r'+1}$ is equivalent to \cite[Conj. 7.4]{Hay} (if we neglect the sign). Indeed, we see that
$$R_{S,T}=((\mathcal{O}_{L,S,T}^\times)^- : (1-\tau)\mathcal{O}_{L,S,T}^\times) R_H^-,$$
where $R_H^-$ is the `minus-unit regulator' defined in \cite[\S 7.2]{Hay} (where our $H$ is denoted by $G$). Since there is a natural map $(J_{W})_H \rightarrow \ZZ[H]/I(H)^{r'+1}$, Conjecture \ref{tori conj} refines \cite[Conj. 7.4]{Hay}.
\end{remark}

\begin{theorem} \label{mrstori}
Conjecture \ref{mrsconj} implies Conjecture \ref{tori conj}. In particular, Conjecture \ref{tori conj} is valid if $K$ is an abelian extension over $\QQ$ or $k$ is a global function field.
\end{theorem}

\begin{proof}
First, note that the Rubin-Stark conjecture for $(K/k,S,T,\emptyset)$ and $(L/k,S,T,W)$ is true by Remark \ref{knownrs} (i) and (ii), respectively.
By Conjecture \ref{mrsconj}, we have
$$\theta_{K/k,S,T}(0)=\Rec_W(\epsilon_{L/k,S,T}^{W}) \quad in \quad (\mathcal{J}_{W})_H(=\ZZ[G/H]\otimes_\ZZ (J_{W})_H).$$
(Note that $\nu^{-1}(\mathcal{N}_H(\theta_{K/k,S,T}(0))) = \theta_{K/k,S,T}(0)$ in $(\mathcal{J}_{W})_H$ by \cite[Lem. 5.6 (iv)]{MR2}.)
Note that $\chi\circ \Rec_i={\rm{rec}}_{w_i}((1-\tau)(\cdot))-1$. So we have
$$\chi(\Rec_W(\epsilon_{L/k,S,T}^{W}))=(\bigwedge_{1\leq i \leq r'}({\rm{rec}}_{w_i}(\cdot)-1))((1-\tau)^{r'}\epsilon_{L/k,S,T}^{W}).$$
We know by the proof of \cite[Th. 3.5]{R} that
$$(1-\tau)^{r'}\epsilon_{L/k,S,T}^{W}=2^{|S|-1-r'}\frac{h_{L,S,T}}{h_{k,S,T}}u_1^{1-\tau} \wedge \cdots \wedge u_{r'}^{1-\tau}.$$
Hence we have
\begin{eqnarray}
\chi(\theta_{K/k,S,T}(0)) &=& \chi(\Rec_W(\epsilon_{L/k,S,T}^{W})) \nonumber \\
&=& (\bigwedge_{1\leq i \leq r'}({\rm{rec}}_{w_i}(\cdot)-1))((1-\tau)^{r'}
\epsilon_{L/k,S,T}^{W}) \nonumber \\
&=& 2^{|S|-1-r'}\frac{h_{L,S,T}}{h_{k,S,T}} (\bigwedge_{1\leq i \leq r'}({\rm{rec}}_{w_i}(\cdot)-1))(u_1^{1-\tau} \wedge \cdots \wedge u_{r'}^{1-\tau}) \nonumber \\
&=& 2^{|S|-1-r'} \frac{h_{L,S,T}}{h_{k,S,T}}R_{S,T}, \nonumber
\end{eqnarray}
as required.

Having now proved the first claim, the second claim follows directly from Corollary \ref{cormrs}.
\end{proof}

\begin{remark} \label{GreitherKucera}
\begin{rm} The strongest previous evidence in favour of the conjecture for tori is that obtained by Greither and Ku\v cera in \cite{GK0,GK}, in which it is referred to as the `Minus Conjecture' and studied in a slightly weaker form in order to remove any occurence of the auxiliary set $T$.
More precisely, by using rather different methods they were able to prove that this conjecture was valid in the case that $k=\QQ$, $K = FK^+$ where $F$ is imaginary quadratic of conductor $f$ and class number $h_F$ and $K^+/\QQ$ is tamely ramified, abelian of exponent equal to an odd prime $\ell$ and ramified at precisely $s$ primes $\{p_i\}_{1\le i\le s}$ each of which splits in $F/\QQ$; further, any of the following conditions are satisfied
\begin{itemize}
\item[$\bullet$] $s=1$ and $\ell \nmid f$ \cite[Th. 8.8]{GK0}, or
\item[$\bullet$] $s=2, \ell\nmid fh_F$ and either $K^+/\QQ$ is cyclic or $p_1$ is congruent to an $\ell$-th power modulo $p_2$ \cite[Th. 8.9]{GK0}, or
\item[$\bullet$] $\ell \ge 3(s+1)$ and $\ell\nmid h_F$ \cite[Th. 3.7]{GK}.
\end{itemize}

It is straightforward to show that the conjecture for tori implies their `Minus conjecture',
using \cite[Chap. IV, Lem. 1.1]{tate} to eliminate the dependence on `$T$'
(just as in the proof of Theorem \ref{darconj}). The validity of the `Minus conjecture' in the case $k = \QQ$ is thus also a consequence of Theorem \ref{mrstori}.
\end{rm}
\end{remark}

\section{Higher Fitting ideals of Selmer groups} \label{HigherFittH1}

In this section, we introduce a natural notion of `higher relative Fitting ideals' in \S \ref{DefinitionRelativeFitt},
and then study the higher Fitting ideals of
the transposed Selmer group $\mathcal{S}^{{\rm tr}}_{S,T}(\GG_{m/K})$.
In this way we prove Theorems \ref{MT2} and \ref{MC4} and Corollary \ref{CorRemark03}.

\subsection{Relative Fitting ideals} \label{DefinitionRelativeFitt}
In this subsection, we recall the definition of Fitting ideals and also introduce a natural notion of `higher relative Fitting ideals'.

Suppose that $R$ is a noetherian ring and $M$ is a finitely
generated $R$-module.
Take an exact sequence
\[ R^{\oplus m} \stackrel{f}{\rightarrow} R^{\oplus n}
\rightarrow M \rightarrow 0,\]
and denote by $A_{f}$
the matrix with $n$ rows and $m$ columns corresponding to $f$.
Then for $i \in \ZZ_{\geq 0}$
the $i$-th Fitting ideal of $M$, denoted by $\Fitt^{i}_{R}(M)$,
is defined to be the ideal generated by all $(n-i) \times (n-i)$
minors of $A_{f}$ if $0 \leq i <n$ and
$R$ if $i \geq n$.
In this situation we call $A_{f}$ a {\it relation matrix} of $M$.
These ideals do not depend on the choice of the above exact sequence
(see \cite[Chap. 3]{North}).
The usual notation is $\Fitt_{i,R}(M)$, but we use the above notation
which is consistent with the exterior power $\bigwedge_R^{i}M$.
If we can take a presentation
\[ R^{\oplus m} \stackrel{f}{\rightarrow} R^{\oplus n}
\rightarrow M \rightarrow 0\]
of $M$ with $m=n$, then we say
$M$ has a {\it quadratic presentation}.

We now fix a submodule $N$ of $M$, and non-negative integers $a$ and $b$. We write $\nu$ for the minimal number of generators of $N$.

If $b > \nu$, then we simply set
\[ \Fitt_{R}^{(a,b)}(M,N) := \Fitt^{a}_{R}(M/N).\]

However, if $b \leq \nu$ then we consider a relation matrix for $M$ of the form
$$A =
\left(
\begin{array}{cc}
A_{1} & A_{2} \\
0 & A_{3}
\end{array}
\right)$$
where $A_{1}$ is a relation matrix of $N$.
We suppose that $A_{1}$ is a matrix with $n_{1}$ rows and $m_{1}$ columns
and $A_{3}$ is a matrix with $n_{2}$ rows and $m_{2}$ columns.
We remove $b$ rows from among the first row to the $n_{1}$-th row of $A$ to get a
matrix $A'$, and remove $a$ rows from $A'$ to get $A''$.
We denote by $F(A'')$ the ideal generated by all $c \times c$ minors of
$A''$ where $c=n_{1}+n_{2}-a-b$ if $c>0$ and $F(A'')=R$ otherwise.
We consider all such $A''$ obtained from $A$ and then define the relative Fitting ideal by setting
\[ \Fitt_{R}^{(a,b)}(M,N):= \sum_{A''} F(A'').\]
By the standard method using the elementary operations of matrices
(see the proof of \cite[p.86, Th. 1]{North}), one can show that this sum does not depend on
the choice of relation matrix $A$.

The following result gives an alternative characterization of this ideal.

\begin{lemma} \label{LemmaRelFitt}
Let $X$ be an $R$-submodule of $M$ that is generated by $(a+b)$ elements
$x_{1},\ldots,x_{a+b}$ such that the elements $x_{1},\ldots,x_{b}$ belong to $N$.
Let $\mathcal{X}$ be the set of such $R$-submodules of $M$.
Then we have
$$\Fitt_{R}^{(a,b)}(M,N)=\sum_{X \in \mathcal{X}} \Fitt^{0}_{R}(M/X).$$
\end{lemma}

\begin{proof}
If $b > \nu$, both sides equal $\Fitt^{a}_{R}(M/N)$, so we may assume
$b \leq \nu$.
Let $e_{1},\ldots,e_{n}$ be the generators of $M$ corresponding to the
above matrix $A$ where $n=n_{1}+n_{2}$.
Suppose that $A''$ is obtained from $A$ by removing $(a+b)$ rows,
from the $i_{1}$-th row to
the $i_{a+b}$-th row with $1 \leq i_{1},\ldots,i_{b} \leq n_{1}$.
Let $X$ be a submodule of $M$ generated by $e_{i_{1}},\ldots,e_{i_{a+b}}$.
Then by definitions $X \in \mathcal{X}$ and $F(A'')=\Fitt^{0}_{R}(M/X)$.
This shows that the left hand side of the equation in Lemma \ref{LemmaRelFitt}
is in the right hand side.

On the other hand, suppose that $X$ is in $\mathcal{X}$ and
$x_{1},\ldots,x_{a+b}$ are generators of $X$.
Regarding $e_{1},\ldots,e_{n_{1}}$, $x_{1},\ldots,x_{b}$, $e_{n_{1}+1},\ldots,e_{n}$, $x_{b+1},\ldots,x_{a+b}$ as generators of
$M$, we have a relation matrix of $M$ of the form
$$B=\left(
\begin{array}{cccc}
A_{1} & B_{1} & A_{2} & B_{2}\\
0 & I_{b} & 0 & 0\\
0 & 0 & A_{3} & B_{3}\\
0 & 0 & 0 & I_{a}
\end{array}
\right)$$
where $I_{a}$, $I_{b}$ are the identity matrices of degree $a$, $b$,
respectively.
Then
$$C=\left(
\begin{array}{cccc}
A_{1} & B_{1} & A_{2} & B_{2}\\
0 & 0 & A_{3} & B_{3}\\
\end{array}
\right)$$
is a relation matrix of $M/X$.
Since $C$ is obtained from $B$ by removing $(a+b)$ rows in the way
of obtaining $A''$ from $A$,
it follows from the definition of the relative Fitting ideal that
$\Fitt^{0}_{R}(M/X) \subset \Fitt_{R}^{(a,b)}(M,N)$.
This shows the other inclusion.
\end{proof}

In the next result we record some useful properties of higher relative Fitting ideals.

\begin{lemma} \label{LemmaRelFitt2}\
\begin{itemize}
\item[(i)] {$\Fitt_{R}^{(a,b)}(M,N) \subset \Fitt_{R}^{a+b}(M)$.
}
\item[(ii)]{$\Fitt_{R}^{(a,0)}(M,N) = \Fitt_{R}^{a}(M)$.
}
\item[(iii)]{Suppose that there exists an exact sequence
$0 \rightarrow M' \rightarrow M \rightarrow R^{\oplus r} \rightarrow 0$ of $R$-modules
and that $N \subset M'$.
Then one has
\[ \Fitt_{R}^{(a,b)}(M,N) = \begin{cases} \Fitt_{R}^{(a-r,b)}(M',N), &\text{ if $a \geq r$,}\\
0, &\text{ otherwise.}\end{cases}\]
}
\item[(iv)]
{Assume that $M/N$ has a quadratic presentation.
Then one has
$$\Fitt_{R}^{(0,b)}(M,N) = \Fitt_{R}^{b}(N) \Fitt_{R}^{0}(M/N).$$
}
\end{itemize}
\end{lemma}

\begin{proof} Claims (i), (ii) and (iii) follow directly from the definition of the higher relative Fitting ideal.
To prove claim (iv), we consider a relation matrix
$$A =
\left(
\begin{array}{cc}
A_{1} & A_{2} \\
0 & A_{3}
\end{array}
\right)$$
as above, where $A_{1}$ is a matrix with $n_{1}$ rows and
$A_{3}$ is a square matrix of $n_{2}$ rows.
Put $n=n_{1}+n_{2}$.
Then a matrix $A''$ obtained from $A$ as above is of the form
$$A'' =
\left(
\begin{array}{cc}
A''_{1} & A''_{2} \\
0 & A_{3}
\end{array}
\right).$$
This is a matrix with $(n-b)$ rows and so a nonzero $(n-b) \times (n-b)$ minor of $A''$ must be
$\det (A_{3})$ times a $(n_{1}-b) \times (n_{1}-b)$ minor of
$A''_{1}$. This implies the required conclusion.
\end{proof}

\subsection{Statement of the conjecture}
Let $K/k,G,S,T,V$ be as in \S \ref{secrs}.
For the element $\epsilon_{K/k,S,T}^V$, the Rubin-Stark conjecture
asserts that $\Phi(\epsilon_{K/k,S,T}^V)$ belongs to $\ZZ[G]$ for every $\Phi$ in $\bigwedge_G^r \Hom_G(\mathcal{O}_{K,S,T}^\times,\ZZ[G])$.

We next formulate a much stronger conjecture which describes the
arithmetic significance of the ideal generated by the elements
$\Phi(\epsilon_{K/k,S,T}^V)$ when $\Phi$ runs over $\bigwedge_G^r \Hom_G(\mathcal{O}_{K,S,T}^\times,\ZZ[G])$.

\begin{conjecture} \label{fitconj} One has an equality
$$\Fit_G^r(\mathcal{S}_{S,T}(\GG_{m/K}))=\{  \Phi(\epsilon_{K/k,S,T}^V)^\#  :  \Phi \in \bigwedge_G^r \Hom_G(\mathcal{O}_{K,S,T}^\times,\ZZ[G])\},$$
or equivalently (by Lemma \ref{dual fitting}),
$$\Fit_G^r(\mathcal{S}^{{\rm tr}}_{S,T}(\GG_{m/K}))=\{  \Phi(\epsilon_{K/k,S,T}^V)  :  \Phi \in \bigwedge_G^r \Hom_G(\mathcal{O}_{K,S,T}^\times,\ZZ[G])\}.$$
\end{conjecture}


The following result shows that this conjecture refines the first half of the statement of Conjecture \ref{burnsconj}.

\begin{proposition}
For a finite set $\Sigma$ of places, we put
$\mathcal{J}_{\Sigma} = \prod_{v \in \Sigma} I(G_{v})\ZZ[G]$.
Assume Conjecture \ref{fitconj} is valid.
Then, for every
$\Phi \in \bigwedge_G^r\Hom_G(\mathcal{O}_{K,S,T}^\times,\ZZ[G])$ and
$v \in S \setminus V$, one has
$$\Phi (\epsilon_{K/k,S,T}^V) \in \mathcal{J}_{S \setminus (V \cup \{v\})}.$$
\end{proposition}

\begin{proof}
It is sufficient to show that $\Fit_G^r(\mathcal{S}^{{\rm tr}}_{S,T}(\GG_{m/K})) \subset  \mathcal{J}_{S \setminus (V \cup \{v\})}$.
Since there is a canonical surjective homomorphism
$$\mathcal{S}^{{\rm tr}}_{S,T}(\GG_{m/K}) \longrightarrow X_{K,S} \simeq \ZZ[G]^{\oplus r}
\oplus X_{K,S\setminus V},$$
we have
$$\Fit_G^r(\mathcal{S}^{{\rm tr}}_{S,T}(\GG_{m/K})) \subset \Fit_G^r(X_{K,S})
=\Fit_G^0(X_{K,S\setminus V}).$$
The existence of the surjective homomorphism
$X_{K,S\setminus V} \rightarrow Y_{K,S\setminus(V\cup\{v\})}$ implies
that $\Fit_G^0(X_{K,S\setminus V})\subset
\Fit_G^0(Y_{K,S\setminus(V\cup\{v\})}) = \mathcal{J}_{S\setminus(V\cup\{ v\})}$.
This completes the proof.
\end{proof}

\subsection{The leading term conjecture implies Conjecture \ref{fitconj}} \label{FittH1can}

The following result combines with Lemma \ref{dual fitting} to imply the statement of Theorem \ref{MT2}(i).

\begin{theorem} \label{ltcfit}
${\rm LTC}(K/k)$ implies Conjecture \ref{fitconj}. In particular, Conjecture \ref{fitconj} is valid if either $K$ is an abelian extension over $\QQ$ or $k$ is a function field or $[K:k]\leq 2$.
\end{theorem}

\begin{proof}
The second claim is a consequence of Remark \ref{knownltc}.

To prove the first claim we assume the validity of ${\rm LTC}(K/k)$.
Then the module $P$ that occurs in the exact sequence (\ref{tateseq}) is free of rank $d$, as we noted before. Hence we may assume $P=F$. Let $z_b \in \bigwedge_G^d F$ be as in \S \ref{secltcrs}.
By ${\rm LTC}(K/k)$, $z_b$ is a $G$-basis of $\bigwedge_G^d F$. Write $z_b$ as
$$z_b=x \bigwedge_{1\leq i \leq d} b_i$$
with some $x \in \ZZ[G]^\times$. By Theorem \ref{zetars} and Proposition \ref{propformula}, we have
$$\epsilon_{K/k,S,T}^V=\pm x \sum_{\sigma \in \mathfrak{S}_{d,r} }{\rm{sgn}}(\sigma) \det(\psi_i(b_{\sigma(j)}))_{r< i,j\leq d}b_{\sigma(1)}\wedge\cdots\wedge b_{\sigma(r)}.$$
Take $\Phi \in \bigwedge_G^r \Hom_G(\mathcal{O}_{K,S,T}^\times,\ZZ[G])$. Since $F/\mathcal{O}_{K,S,T}^\times \simeq \im (\psi) \subset F$ is $\ZZ$-torsion-free, we know by Lemma \ref{lemr}(ii) that the map
$$\Hom_G(F,\ZZ[G]) \longrightarrow \Hom_G(\mathcal{O}_{K,S,T}^\times,\ZZ[G])$$
induced by the inclusion $\mathcal{O}_{K,S,T}^\times \subset F$ is surjective. Hence, we can take a lift $\widetilde \Phi$ of $\Phi$ to $\bigwedge_G^r \Hom_G(F,\ZZ[G])$. We have
\begin{eqnarray}
\Phi(\epsilon_{K/k,S,T}^V)&=&\pm x \sum_{\sigma \in \mathfrak{S}_{d,r} }{\rm{sgn}}(\sigma) \det(\psi_i(b_{\sigma(j)}))_{r< i,j\leq d}\widetilde \Phi(b_{\sigma(1)}\wedge\cdots\wedge b_{\sigma(r)}) \nonumber \\
&\in& \langle  \det(\psi_i(b_{\sigma(j)}))_{r< i,j\leq d} : \sigma \in \mathfrak{S}_{d,r} \rangle_G. \nonumber
\end{eqnarray}

We consider the matrix $A$ corresponding to the presentation
\[ F \rightarrow F \rightarrow \mathcal{S}^{{\rm tr}}_{S,T}(\GG_{m/K}) \rightarrow 0\]
which
comes from the exact sequence (\ref{tateseq}).
By Lemma \ref{lempsi}, $\psi_i=0$ for $1\leq i \leq r$.
If we write elements in $F$ as column vectors, this implies that
the $i$-th row of $A$ is zero for all $i$ such that $1\leq i \leq r$.
Hence we have
$$\Fit_G^r(\mathcal{S}^{{\rm tr}}_{S,T}(\GG_{m/K}))= \langle\det(\psi_i(b_{\sigma(j)}))_{r< i,j\leq d} : \sigma \in \mathfrak{S}_{d,r}\rangle_G.$$
Therefore, we get an inclusion
$$\{  \Phi(\epsilon_{K/k,S,T}^V)  :  \Phi \in \bigwedge_G^r \Hom_G(\mathcal{O}_{K,S,T}^\times,\ZZ[G])\} \subset \Fit_G^r(\mathcal{S}^{{\rm tr}}_{S,T}(\GG_{m/K})).$$

We obtain the reverse inclusion from
$$(b_{\sigma(1)}^\ast \wedge\cdots\wedge b_{\sigma(r)}^\ast)(\epsilon_{K/k,S,T}^V)= \pm
x \det(\psi_i(b_{\sigma(j)}))_{r<i,j\leq d}$$
and the fact that $x$ is a unit in $\ZZ[G]$.
\end{proof}

\subsection{The proof of Theorem \ref{MC4}} \label{ProofofCorollaryMC4}
For any $G$-module $M$ we write $M^*$ for the linear dual $\Hom_\ZZ(M,\ZZ)$
endowed with the natural contragredient action of $G$.
We also set $V' := V\cup \{v\}$.

We start with a useful technical observation.

\begin{lemma}\label{tech2} For each integer $i$ with $1\le i\le r$ let $\varphi_i$ be an element of $(\mathcal{O}_{K,S,T}^\times)^*$. Then for any given integer $N$ there is a subset $\{\varphi_i':1\le i\le r\}$ of $(\mathcal{O}_{K,S,T}^\times)^*$ which satisfies the following properties.
\begin{itemize}
\item[(i)] For each $i$ one has $\varphi_i' \equiv \varphi_i$ modulo $N\cdot (\mathcal{O}_{K,S,T}^\times)^*$.
\item[(ii)] The image in $(\mathcal{O}^\times_{K,V',T})^*$ of the submodule of $(\mathcal{O}_{K,S,T}^\times)^*$ that is generated by the set $\{\varphi_i': 1\le i\le r\}$ is free of rank $r$.
    \end{itemize}
\end{lemma}
\begin{proof} Our choice of $V$ implies that we may choose a free $G$-submodule $\mathcal{F}$ of $(\mathcal{O}^\times_{K,V',T})^*$ of rank $r$. We then choose a subset $\{f_i:1\le i\le r\}$ of $(\mathcal{O}_{K,S,T}^\times)^*$ which the natural surjection  $\rho:(\mathcal{O}_{K,S,T}^\times)^*\to (\mathcal{O}_{K,V ',T}^\times)^*$ sends to a basis of $\mathcal{F}$. For any integer $m$ we set $\varphi_{i,m} := \varphi_i + mN f_i$ and note it suffices to show that for any sufficiently large $m$ the elements $\{\rho(\varphi_{i,m}): 1\le i\le r\}$ are linearly independent over $\QQ[G]$.

Consider the composite homomorphism of $G$-modules $\mathcal{F} \to \QQ(\mathcal{O}^\times_{K,V',T})^*\to \QQ \mathcal{F}$ where the first arrow sends each $\rho(f_i)$ to $\rho(\varphi_{i,m})$ and the second is induced by a choice of $\QQ[G]$-equivariant section to the projection $\QQ (\mathcal{O}^\times_{K,V',T})^*\to \QQ ((\mathcal{O}^\times_{K,V',T})^*/\mathcal{F})$. Then, with respect to the basis $\{\rho(f_i):1\le i\le r\}$, this linear map is represented by a matrix of the form $A + mN  I_r$ for a matrix $A$ in ${\rm M}_r(\QQ[G])$ that is independent of $m$. In particular, if $m$ is large enough to ensure that $-mN$ is not an eigenvalue of $e_\chi A$ for any $\chi$, then the composite homomorphism is injective and so the elements $\{\rho(\varphi_{i,m}): 1\le i\le r\}$ are linearly independent over $\QQ[G]$, as required.\end{proof}

For each integer $i$ with $1\le i\le r$ let $\varphi_i$ be an element of $(\mathcal{O}_{K,S,T}^\times)^*$. Then, for any non-zero integer $N$ which belongs to ${\rm Fitt}^0_{G}({\rm Cl}(K))$ we choose homomorphisms $\varphi_i'$ as in Lemma \ref{tech2}. Then the congruences in Lemma \ref{tech2}(i) imply that
\[ (\bigwedge_{1\le i \le r}\varphi_i)(\epsilon^V_{K/k,S,T}) \equiv (\bigwedge_{1\le i\le r}\varphi'_i)(\epsilon^V_{K/k,S,T})\,\,\,{\rm modulo} \,\,N\cdot\ZZ[G].\]
Given this, Lemma \ref{tech2}(ii) implies that Theorem \ref{MC4} is true provided that it is true for all $\Phi$ of the form ${\bigwedge}_{1\le i\le r}\varphi_i$ where the images in $(\mathcal{O}^\times_{K,V',T})^*$ of the homomorphisms $\varphi_i$ span a free module of rank $r$.

We shall therefore assume in the sequel that $\Phi$ is of this form.

For each index $i$ we now choose a lift $\widetilde\varphi_i$ of $\varphi_i$ to
$\mathcal{S}_{S,T}(\mathbb{G}_{m/K})$ and then write $\mathcal{E}_\Phi$ for the $G$-module that is generated by the set $\{\widetilde\varphi_i:1\le i\le r\}$.

\begin{proposition}\label{ltc prop} If ${\rm LTC}(K/k)$ is valid, then for every $\Phi$ as above one has
\begin{equation*}\label{reduction} \Phi(\epsilon^V_{K/k,S,T})^{\#} \in
{\rm Fitt}_{G}^0(\mathcal{S}_{S,T}(\mathbb{G}_{m/K})/\mathcal{E}_\Phi).\end{equation*}\end{proposition}

\begin{proof} We use the existence of an exact triangle in $D^{\rm p}(\ZZ[G])$ of the form
\begin{equation}\label{triangle} \ZZ[G]^{\oplus r,\bullet} \xrightarrow{\theta} R\Gamma_{c,T}((\mathcal{O}_{K,S})_\mathcal{W},\ZZ) \xrightarrow{\theta'} C^\bullet \to \ZZ[G]^{\oplus r,\bullet}[1]. \end{equation}
Here $\ZZ[G]^{\oplus r,\bullet}$ denotes the complex $\ZZ[G]^{\oplus r}[-1]\oplus \ZZ[G]^{\oplus r}[-2]$ and, after choosing an ordering $\{v_i:1\le i\le r\}$ of the places in $V$, the morphism $\theta$ is uniquely specified by the condition  that $H^1(\theta)$ sends each element $b_i$ of the canonical basis $\{b_i:1\le i\le r\}$ of $\ZZ[G]^{\oplus r}$ to $w_{i}^*$ in $(Y_{K,V})^*\subset (X_{K,S})^* = H^1_{c,T}((\mathcal{O}_{K,S})_\mathcal{W},\ZZ)$ and $H^2(\theta)$ sends each $b_i$ to $\widetilde\varphi_i$ in $\mathcal{S}_{S,T}(\mathbb{G}_{m/K})$.

Note that the long exact cohomology sequence of this triangle implies $C^\bullet$ is acyclic outside degrees one and two and identifies $H^1(C^\bullet)$ and $H^2(C^\bullet)$ with $(X_{K,S\setminus V})^*$ and $\mathcal{S}_{S,T}(\mathbb{G}_{m/K})/\mathcal{E}_\Phi$,
respectively.

In particular, if we now write $e_r$ for the idempotent of $\QQ[G]$ obtained as $\sum_{r_{\chi,S}=r} e_\chi$, then the space $e_r\QQ H^i(C^\bullet)$ vanishes for both $i = 1$ and $i = 2$. We may therefore choose a commutative diagram of $\mathbb{R}[G]$-modules
\begin{equation}\label{comm diag} \begin{CD}
0 @> >> \mathbb{R}[G]^{\oplus r} @>  H^1(\theta)>> \mathbb{R} H^1_{c,T}((\mathcal{O}_{K,S})_\mathcal{W},\ZZ) @> H^1(\theta')>> \RR H^1(C^\bullet) @> >> 0\\
@. @V \lambda_1 VV @V \lambda_2 VV @V\lambda_3VV\\
0 @> >> \mathbb{R}[G]^{\oplus r} @> H^2(\theta)>> \mathbb{R} H^2_{c,T}((\mathcal{O}_{K,S})_\mathcal{W},\ZZ) @> H^2(\theta')>> \RR H^2(C^\bullet) @> >> 0\end{CD}\end{equation}
such that $e_r \lambda_2  = e_r\lambda_{K,S}^* $.

This diagram combines with the triangle (\ref{triangle}) to imply that there is an equality of lattices
\begin{align}\label{intermediate} \vartheta_{\lambda_2}({\rm det}_G(R\Gamma_{c,T}((\mathcal{O}_{K,S})_\mathcal{W},\ZZ)))^{-1}
&= \det(\lambda_1)\cdot\vartheta_{\lambda_3}({\rm det}_G(C^\bullet))^{-1}.\end{align}

We now assume that the conjecture ${\rm LTC}(K/k)$ is valid. Then Proposition \ref{ltc prop2} implies that ${\rm det}_G(R\Gamma_{c,T}((\mathcal{O}_{K,S})_\mathcal{W},\ZZ))^{-1}$ is a free rank one $\ZZ[G]$-module and further that if we choose any basis $\xi$ for this module, then both $e_r \xi$ and $e_r\theta_{K/k,S,T}^*(0)^\#=\theta_{K/k,S,T}^{{(r)},\#}$ are bases of the $e_r\ZZ[G]$-module
\[ e_r \vartheta_{\lambda_2}({\rm det}_G(R\Gamma_{c,T}((\mathcal{O}_{K,S})_\mathcal{W},\ZZ)))^{-1} = e_r\vartheta_{\lambda_{K,S}^*}({\rm det}_G(R\Gamma_{c,T}((\mathcal{O}_{K,S})_\mathcal{W},\ZZ)))^{-1}.\]
Bass's Theorem  (cf. \cite[Chap. 7, (20.9)]{lam}) implies that for each prime $p$ the projection map $\ZZ_{(p)}[G]^\times \to e_r\ZZ_{(p)}[G]^\times$ is surjective. The above equality thus implies that the $\ZZ_{(p)}[G]$-module $\vartheta_{\lambda_2}({\rm det}_G(R\Gamma_{c,T}((\mathcal{O}_{K,S})_\mathcal{W},\ZZ)))^{-1}\otimes_\ZZ \ZZ_{(p)}$ has a basis $\xi_p$ for which one has $e_r \xi_p = e_r \theta_{K/k,S,T}^*(0)^\#=\theta_{K/k,S,T}^{{(r)},\#}$. For each prime $p$ the equality (\ref{intermediate}) therefore implies that
\begin{align}\label{contain} &e_r\vartheta_{\lambda_3}({\rm det}_G(C^\bullet))^{-1}\otimes_\ZZ \ZZ_{(p)}\\ =\, &{\rm det}(\lambda_1)^{-1} e_r\vartheta_{\lambda_2}({\rm det}_G(R\Gamma_{c,T}((\mathcal{O}_{K,S})_\mathcal{W},\ZZ)))^{-1}\otimes_\ZZ \ZZ_{(p)}\notag\\
=\, &\ZZ_{(p)}[G] \cdot e_r{\rm det}(\lambda_1)^{-1}\theta_{K/k,S,T}^{{(r)},\#}.\notag\end{align}

Now the commutativity of (\ref{comm diag}) implies that $e_r{\rm det}(\lambda_1)$ is equal to the determinant of the matrix which represents $e_r \lambda_{K,S}^*$ with respect to the bases $\{e_r w_i^\ast: 1\le i\le r\}$ and $\{e_r \varphi_i: 1\le i \le r\}$ and hence that
\[ e_r \bigwedge_{1\le i\le r}\lambda_{K,S}^*(w_i^*) = e_r {\rm det}(\lambda_1) \Phi.\]
Since the element $\epsilon_{K/k,S,T}^V$ is defined via the equality
\[ \theta_{K/k,S,T}^{(r)} \bigwedge_{1\le i\le r}\lambda_{K,S}^{-1}(w_i-w) = \epsilon_{K/k,S,T}^V,\]
one therefore has
\begin{eqnarray}\label{almost1}\Phi(\epsilon_{K/k,S,T}^V)^\#&=&
(e_r {\rm det}(\lambda_1))^{-1}\bigwedge_{1\le i\le r}
\lambda_{K,S}^*(w_i^*)(\theta_{K/k,S,T}^{{(r)},{\#}}  (\bigwedge_{1\le i\le r}
\lambda_{K,S}^{-1}(w_i-w)))\\
&=& (e_r {\rm det}(\lambda_1))^{-1}\theta_{K/k,S,T}^{(r),\#}\notag \\
&\in& e_r\vartheta_{\lambda_3}({\rm det}_G(C^\bullet))^{-1}\otimes_\ZZ \ZZ_{(p)}\notag
\end{eqnarray}
where the last containment follows from (\ref{contain}).

Now by the same reasoning as used in the proof of Lemma \ref{dual fitting}, we know that the $p$-localized complex $\ZZ_{(p)}\otimes C^\bullet$ is represented by a complex $P\xrightarrow{\delta} P$, where $P$ is a finitely generated free $\ZZ_{(p)}[G]$-module and the first term is placed in degree one. In particular, since for any character $\chi$ of $G$ the space $e_\chi\CC H^1(C^\bullet) = e_\chi\CC \ker(\delta)$ does not vanish if $e_\chi e_r = 0$,
one has
\begin{align}\label{there} e_r\vartheta_{\lambda_3}({\rm det}_G(C^\bullet))^{-1}_{(p)} &= \ZZ_{(p)}[G] e_r{\rm det}(\delta) \\
&= \ZZ_{(p)}[G] {\rm det}(\delta)\notag\\
&\subset {\rm Fitt}_{\ZZ_{(p)}[G]}^0((H^2_{c,T}((\mathcal{O}_{K,S})_\mathcal{W},\ZZ)/\mathcal{E}_\Phi)\otimes_\ZZ \ZZ_{(p)})\notag\\
&= {\rm Fitt}_{G}^0(H^2_{c,T}((\mathcal{O}_{K,S})_\mathcal{W},\ZZ)/\mathcal{E}_\Phi)\otimes_\ZZ \ZZ_{(p)}.\notag\end{align}
%
The inclusion here follows from the tautological exact sequence
\[ P \stackrel{\delta}{\longrightarrow} P \longrightarrow H^2(\ZZ_{(p)}\otimes C^\bullet) \longrightarrow 0\]
and the identification $H^2(\ZZ_{(p)}\otimes C^\bullet)
= (H^2_{c,T}((\mathcal{O}_{K,S})_\mathcal{W},\ZZ)/\mathcal{E}_\Phi)\otimes_\ZZ \ZZ_{(p)}.$

The claimed result now follows from (\ref{almost1}) and (\ref{there}).
\end{proof}

Now we proceed to the proof of Theorem \ref{MC4}.
The existence of a surjective homomorphism of $G$-modules $f:\mathcal{S}_{S,T}(\GG_{m/K}) \rightarrow
\mathcal{S}_{V' \cup S_\infty,T}(\GG_{m/K})$ (see Proposition \ref{new one}(ii)) combines with Proposition \ref{ltc prop} to imply that
\begin{equation}\label{contain2}\Phi(\epsilon^V_{K/k,S,T})^{\#} \in
{\rm Fitt}_{G}^0(\mathcal{S}_{V'\cup S_\infty,T}(\GG_{m/K})/f(\mathcal{E}_\Phi)).\end{equation}

This implies the first assertion of Theorem \ref{MC4} since the natural map
${\rm Cl}_{V'}^T(K)^\vee \rightarrow \mathcal{S}_{V'\cup S_\infty,T}(\GG_{m/K})$ induces an injection
\[ {\rm Cl}_{V'}^T(K)^\vee \rightarrow
\mathcal{S}_{V'\cup S_\infty,T}(\GG_{m/K})/f(\mathcal{E}_\Phi).\]

In addition, if $G$ is cyclic, then the latter injection combines with (\ref{contain2}) to imply that
$$\Phi(\epsilon^V_{K/k,S,T}) \in \Fitt_{G}^{0}({\rm Cl}_{V'}^T(K)^\vee)^{\#}
=\Fitt_{G}^{0}({\rm Cl}_{V'}^T(K)),$$
as claimed by the second assertion of Theorem \ref{MC4}.

This completes the proof of Theorem \ref{MC4}.

\subsection{The proof of Corollary \ref{CorRemark03}} \label{ProofofRemark}
Let $K/k$ be a CM-extension, $S=S_{\infty}(k)$, and $p$ an odd prime.
For a $\ZZ_{p}[G]$-module $M$, we denote by $M^{-}$ the
submodule on which the complex conjugation acts as $-1$.

Then, since complex conjugation acts trivially on
$\Hom_{\ZZ}(\mathcal{O}^\times_{K,S,T},\ZZ)\otimes \ZZ_p$, the exact sequence
$$
 0 \longrightarrow {\rm Cl}^T(K)^\vee \longrightarrow
\mathcal{S}_{S,T}(\GG_{m/K})
\longrightarrow \Hom_\ZZ(\mathcal{O}^\times_{K,S,T},\ZZ) \longrightarrow 0,$$
implies that in this case there is an equality
$$(({\rm Cl}^T(K) \otimes \ZZ_{p})^\vee)^-=
(\mathcal{S}_{S,T}(\GG_{m/K}) \otimes \ZZ_{p})^-.$$

In addition, in this case the containment of Proposition \ref{ltc prop} applies with $V$ empty (so $r=0$ and $\mathcal{E}_\Phi$ vanishes) to imply that
$$\theta_{K/k,S,T}(0)^{\#} \in {\rm Fitt}_{G}^0(
\mathcal{S}_{S,T}(\GG_{m/K})),$$
and hence one has
$$\theta_{K/k,S,T}(0)^{\#} \in \Fitt_{\ZZ_p[G]}^0((({\rm Cl}^T(K) \otimes \ZZ_{p})^{\vee})^-).$$
Since $\theta_{K/k,S,T}(0)$ lies in the minus component of $\ZZ_p[G]$, this is in turn equivalent to the required equality
$$\theta_{K/k,S,T}(0)^{\#} \in \Fitt_{\ZZ_p[G]}^0(({\rm Cl}^T(K) \otimes \ZZ_{p})^{\vee}).$$

This completes the proof of Corollary \ref{CorRemark03}.

\subsection{The higher relative Fitting ideals of the dual Selmer group} \label{HigherRelativeFitt}

We write $M_{\rm tors}$ for the $\ZZ$-torsion submodule of a $G$-module $M$ and abbreviate the higher relative Fitting ideal $\Fitt_{\ZZ[G]}^{(a,b)}(M,M_{{\rm tors}})$ to $\Fitt_{G}^{(a,b)}(M)$.

In this subsection, we study the ideals $\Fitt_{G}^{(r,i)}(\mathcal{S}^{{\rm tr}}_{S,T}(\GG_{m/K}))$ and, in particular, prove Theorem \ref{MT2}(ii). We note that the exact sequence (\ref{dual ses}) identifies $\mathcal{S}^{{\rm tr}}_{S,T}(\GG_{m/K})_{\rm tors}$ with the group ${\rm Cl}_S^T(K)$.

For each non-negative integer $i$ we define the set $\mathcal{V}_{i}$ as in Theorem \ref{MT2}(ii).

\begin{conjecture} \label{ConjectureRelativeFitt} For each non-negative integer $i$ one has an equality
$$\Fitt_{G}^{(r,i)}(\mathcal{S}^{{\rm tr}}_{S,T}(\GG_{m/K}))
=
\{  \Phi(\epsilon_{K/k,S \cup V',T}^{V \cup V'})  :  V' \in \mathcal{V}_{i} \
\mbox{and} \
\Phi \in \bigwedge_G^{r+i} \Hom_G(\mathcal{O}_{K,S \cup V',T}^\times,\ZZ[G])\}.$$
\end{conjecture}


The following result is a generalization of Theorem \ref{ltcfit} in \S \ref{FittH1can}.

\begin{theorem} \label{ThRelativeHigherFitt} If ${\rm LTC}(K/k)$ is valid, then so is Conjecture \ref{ConjectureRelativeFitt}.
\end{theorem}

\begin{proof}
We consider the composition of the two canonical homomorphisms
\[ \mathcal{S}^{{\rm tr}}_{S,T}(\GG_{m/K}) \rightarrow X_{K,S} \rightarrow Y_{K,V},\]
and denote its kernel by $\mathcal{S}^{{\rm tr}}_{S,T}(\GG_{m/K})'$.
By Lemma \ref{LemmaRelFitt2} (iii), we have
\begin{equation} \label{ProofThRelativeHigherFitt1}
\Fitt_{G}^{(r,i)}(\mathcal{S}^{{\rm tr}}_{S,T}(\GG_{m/K}))
=\Fitt_{G}^{(0,i)}(\mathcal{S}^{{\rm tr}}_{S,T}(\GG_{m/K})').
\end{equation}
We also note that the sequence (\ref{dual ses}) gives rise to an exact sequence of $G$-modules
\begin{equation} \label{ProofThRelativeHigherFitt11}0 \longrightarrow {\rm Cl}_S^T(K) \longrightarrow \mathcal{S}^{{\rm tr}}_{S,T}(\GG_{m/K})'
\longrightarrow X_{K,S \setminus V} \longrightarrow 0.\end{equation}

For $V' \in \mathcal{V}_{i}$, we denote by $\mathcal{S}^{{\rm tr}}_{S \cup V',T}(\GG_{m/K})'$ the kernel of the natural composition
\[ \mathcal{S}^{{\rm tr}}_{S \cup V',T}(\GG_{m/K}) \rightarrow X_{K,S \cup V'}
\rightarrow Y_{K,V \cup V'}\]
so that the following sequence is exact
$$0 \longrightarrow {\rm Cl}_{S \cup V'}^T(K) \longrightarrow
\mathcal{S}^{{\rm tr}}_{S \cup V',T}(\GG_{m/K})'
\longrightarrow X_{K,S \setminus V} \longrightarrow 0.$$

Let $X_{V'}$ be the subgroup of ${\rm Cl}_S^T(K)$
generated by the classes of
places of $K$ above $V'$ in ${\rm Cl}_S^T(K)$.
Since ${\rm Cl}_S^T(K)/X_{V'}={\rm Cl}_{S \cup V'}^T(K)$,
there is an isomorphism
$\mathcal{S}^{{\rm tr}}_{S,T}(\GG_{m/K})'/X_{V'} \simeq
\mathcal{S}^{{\rm tr}}_{S \cup V',T}(\GG_{m/K})'$.
By Chebotarev density theorem and Lemma \ref{LemmaRelFitt}, we obtain
\begin{eqnarray} \label{ProofThRelativeHigherFitt2}
\Fitt_{G}^{(0,i)}(\mathcal{S}^{{\rm tr}}_{S,T}(\GG_{m/K})')
&=&\sum_{V' \in  \mathcal{V'}_{i}}
\Fitt_{G}^{0}(\mathcal{S}^{{\rm tr}}_{S \cup V',T}(\GG_{m/K})') \nonumber\\
&=&\sum_{V' \in  \mathcal{V'}_{i}} \Fitt_{G}^{r+i}(\mathcal{S}^{{\rm tr}}_{S \cup V',T}(\GG_{m/K}))
\end{eqnarray}
where we used Lemma \ref{LemmaRelFitt2} (iii) again to get the last equality.

Now Theorem \ref{ThRelativeHigherFitt} follows from
(\ref{ProofThRelativeHigherFitt1}), (\ref{ProofThRelativeHigherFitt2}) and
Theorem \ref{ltcfit}.
\end{proof}

\begin{corollary} \label{MC11}
We assume that ${\rm LTC}(K/k)$ is valid and that the group $G=\Gal(K/k)$ is cyclic.
Then for each non-negative integer $i$ one has an equality

\begin{multline*}\Fitt_{G}^{i}({\rm Cl}_S^T(K))\Fitt_{G}^{0}(X_{K,S\setminus V})\\
=
\{  \Phi(\epsilon_{K/k,S \cup V',T}^{V \cup V'})  :  V' \in \mathcal{V}_{i} \
\mbox{and} \
\Phi \in \bigwedge_G^{r+i} \Hom_G(\mathcal{O}_{K,S \cup V',T}^\times,\ZZ[G])\}.\end{multline*}
\end{corollary}

\begin{proof}
Since $G$ is cyclic, the $G$-module $X_{K,S\setminus V}$ has a quadratic presentation. We may therefore apply Lemma \ref{LemmaRelFitt2}(iv) to the exact sequence (\ref{ProofThRelativeHigherFitt11}) to obtain an equality
\[ \Fitt_{G}^{i}({\rm Cl}_S^T(K))\Fitt_{G}^{0}(X_{K,S\setminus V}) = \Fitt_{G}^{(0,i)}(\mathcal{S}^{{\rm tr}}_{S,T}(\GG_{m/K})').\]

Given this equality, the claimed result follows from Theorem \ref{ThRelativeHigherFitt} and the equality (\ref{ProofThRelativeHigherFitt1}).
\end{proof}

An application of Theorem \ref{ThRelativeHigherFitt} to
character components of ideal class groups will be given in \S \ref{HigherFittChar}.

\section{Gross-Stark conjecture and the cyclotomic $\ZZ_{p}$-extension} \label{grost}
In this section we prove Theorem \ref{MT4} and Corollary \ref{IntroConsequencesOfVentullo}. We thus assume throughout that $k$ is totally real, that $K$ is a CM-field, and that $K/k$ is a finite abelian extension with Galois group $G$.

We fix an odd prime number $p$  and write $j$ for the complex conjugation in $G$.
We denote by $K_{\infty}$ the cyclotomic $\ZZ_{p}$-extension of $K$ and for each natural number $n$ write
$K_{n}$ for the $n$-th layer of this extension.

\subsection{$\mathcal{L}$-invariants for arbitrary rank}
Let $S, T, V,r$ be as in \S \ref{secrs}. We use the convention in \S \ref{convention}, so we fix a labeling $S=\{v_0,v_1,\ldots,v_n\}$ so that $V=\{v_1,\ldots,v_r\}$.
We assume that $S$ contains all primes above $p$.
We also assume that each prime in $V$ is a prime above $p$.
Let $\log_{p}:\QQ_{p}^{\times} \rightarrow \ZZ_{p}$
be Iwasawa's logarithm, namely the logarithm normalized as
$\log_{p}(p)=0$.
Recall that we fixed a place $w_{i}$ of $K$ above $v_{i}$.
For $x \in K^{\times}$, we define
$$\Log_{{i}}(x)=-\sum_{\tau \in G} \log_{p}
(\N_{K_{w_{i}}/{\QQ_{p}}}(\tau x)) \tau^{-1} \in
\ZZ_{p}[G].$$
Using the normalized additive valuation $\ord_{w_{i}}$ of $w_{i}$, we define
$$\Ord_{{i}}(x)=\sum_{\tau \in G} \ord_{w_{i}}(\tau x) \tau^{-1} \in
\ZZ[G].$$
We define $\Log_{V}$ and $\Ord_{V}$ by
$$\Log_{V}=\bigwedge_{1\le i \le r} \Log_{{i}}:
\QQ\bigwedge^r_{G} \mathcal{O}_{K,S,T}^\times
\longrightarrow \QQ_{p}[G],$$
$$\Ord_{V}=\bigwedge_{1\le i\le r} \Ord_{{i}}:
\QQ\bigwedge^r_{G} \mathcal{O}_{K,S,T}^\times
\longrightarrow \QQ[G].$$

For any $\ZZ[G]$-module $M$, we define submodules $M^+$ and $M^-$ by setting
\[ M^{\pm} :=\{x \in M : j(x)=\pm x \}.\]
Since $p$ is odd, if $M$ is a $\ZZ_{p}[G]$-module or a $\QQ[G]$-module,
$M$ is decomposed as $M=M^{+} \oplus M^{-}$.
In this case, for $x \in M$ we write $x^{\pm}$ for the component of $x$
in $M^{\pm}$.

We consider $\mathcal{O}_{K,V,T}^{\times}$ which is a subgroup of
$\mathcal{O}_{K,S,T}^{\times}$, and
${\bigwedge}^r_{G} \mathcal{O}_{K,V,T}^\times$.
Since $(\QQ\mathcal{O}_{K,\{v\},T}^\times)^{-}$
is a free $\QQ[G]^{-}$-module of rank $1$ for each $v \in V$,
it is easy to check that
$(\QQ{\bigwedge}^r_{G} \mathcal{O}_{K,V,T}^\times)^{-}$
is a free $\QQ[G]^{-}$-module of rank $1$.
We see that for a nonzero divisor $x \in (\QQ{\bigwedge}^r_{G}
\mathcal{O}_{K,V,T}^\times)^{-}$,
$\Ord_{V}(x)$ is not a zero divisor.
We define $\mathcal{L}_{V}$ by
$$\mathcal{L}_{V}=\frac{\Log_{V}(x)}{\Ord_{V}(x)} \in \QQ_{p}[G]^-.$$
Then $\mathcal{L}_{V}$ does not depend on the choice of
$x$ since $(\QQ{\bigwedge}^r_{G}
\mathcal{O}_{K,V,T}^\times)^{-}$ is free of rank $1$
over $\QQ[G]^{-}$.
We note that if $r=1$, then this invariant coincides with the $\mathcal{L}$-invariant used by
Darmon, Dasgupta and Pollack in \cite{DDP}.

\subsection{}
We denote by $K_{\infty}/K$ the cyclotomic $\ZZ_{p}$-extension
and $K_{n}$ its $n$-th layer.
Put $\Gamma=\Gal(K_{\infty}/K)$.

For sufficiently large $n$, no place in $S$ splits completely in $K_{n}$.
We consider $\theta_{K_{n}/k,S,T}(0) \in \ZZ[\Gal(K_{n}/k)]$.
The system $\{\theta_{K_{n}/k,S,T}(0)\}_{n \gg 0}$ becomes a
projective system, and defines an element
$\theta_{K_{\infty}/k,S,T} \in \ZZ_{p}[[\Gal(K_{\infty}/k)]]$.
Since $\Gal(K_{\infty}/k) \otimes \ZZ_{p}$ is a $\ZZ_{p}$-module
of rank $1$, there is an intermediate field $F$ such that
$F_{\infty}=K_{\infty}$ and $F \cap k_{\infty}=k$, namely
$\Gal(K_{\infty}/k)$ is the product of $\Gal(F_{\infty}/F) \simeq \ZZ_{p}$
and $\Gal(F/k)$.

Let $\gamma_{F}$ be a generator of $\Gal(F_{\infty}/F)$.
By the correspondence $\gamma_{F} \leftrightarrow 1+t$, we
identify $\ZZ_{p}[[\Gal(K_{\infty}/k)]]$ with a power series ring
$\ZZ_{p}[\Gal(F/k)][[t]]$.

We regard $\theta_{K_{\infty}/k,S,T}$ as a power series in
$\ZZ_{p}[\Gal(F/k)][[t]]$, and then define a function
$L_{p, K/k, S, T}(s)$ on the $s$-plane by setting
$$L_{p, K/k, S, T}(s) = c_{K_{\infty}/K}((\theta_{K_{\infty}/k,S,T})
|_{ t=\gamma_{F} \kappa(\gamma_{F})^{s}-1})$$
for $s \in \ZZ_{p}$
where
$c_{K_{\infty}/K}: \ZZ_{p}[[\Gal(K_{\infty}/k)]] \rightarrow
\ZZ_{p}[\Gal(K/k)]$ is the restriction map and
$\kappa: \Gal(K_{\infty}/k) \rightarrow \ZZ_{p}^{\times}$
is the cyclotomic character. It is easy to check that $L_{p, K/k, S, T}(s)$ does not
depend on the choice of $F$ and $\gamma_{F}$.

This function $L_{p, K/k, S, T}(s)$ is the $p$-adic $L$-function of Deligne-Ribet.
More precisely, suppose that $\chi$ is a ($\overline \QQ_p$-valued) $1$-dimensional odd character of
the absolute Galois group of $k$.
We can decompose $\chi=\chi_{1} \chi_{2}$ such that
the field $F(\chi_{1})$ corresponding to $\chi_{1}$
(namely, $\chi_{1}$ induces a faithful character of $\Gal(F(\chi_{1})/k)$)
satisfies $F({\chi_{1}}) \cap k_{\infty} = k$ and
$\chi_{2}$ is a character of $\Gal(k_{\infty}/k)$.
Let $K$ be the field corresponding to the kernel of $\chi$.
We have $K_{\infty}=F(\chi_{1})_{\infty}$.
Consider the $\chi_{1}$-component
\[ \theta_{K_{\infty}/k,S,T}^{\chi_{1}}(t)=
\theta_{F(\chi_{1})_{\infty}/k,S,T}^{\chi_{1}}(t) \in \ZZ_{p}[\im (\chi_{1})][[t]]\]
of $\theta_{K_{\infty}/k,S,T}$.
Then the $p$-adic $L$-function of Deligne-Ribet $L_{p, S, T}(\chi^{-1}\omega, s)$
coincides with $\theta_{K_{\infty}/k,S,T}^{\chi_{1}}
(\chi_{2}(\gamma_{F})\kappa(\gamma_{F})^{s}-1)$.
In this sense, the `$\chi$-component' of $L_{p, K/k, S, T}(s)$
is the $p$-adic $L$-function of Deligne-Ribet.

The Gross-Stark conjecture predicts firstly that the order of vanishing at $s=0$ of
$L_{p, S, T}(\chi^{-1}\omega, s)$ is equal to
\[ r_\chi:=r_{\chi,S}=|\{v \in S : \chi(G_v)=1\}|.\]

It is known that the main conjecture for totally real fields proved by Wiles in \cite{Wiles} implies that
$s^{r_{\chi}}$ divides
$L_{p, S, T}(\chi^{-1}\omega, s)$.
Concerning the `leading term' at $s=0$ of $L_{p, S, T}(\chi^{-1}\omega, s)$, the Gross-Stark
conjecture then asserts the following.

\begin{conjecture} [The Gross-Stark conjecture; ${\rm GS}(K/k,S,T)$] \label{GrossStark}
We have
$$\lim_{s \rightarrow 0} L^{(r)}_{p, K/k, S, T}(s) = \mathcal{L}_{V} \theta_{K/k,S \setminus V,T}(0)$$
where $L^{(r)}_{p, K/k, S, T}(s)=s^{-r}L_{p, K/k, S, T}(s)$.
\end{conjecture}

\begin{remark} \label{Rem71}
\begin{rm}
Let $\epsilon_{K/k,S,T}^V \in \bigcap_G^r \mathcal{O}_{K,S,T}^\times$ be
the Rubin-Stark element.
For a character $\chi$  of $G$, consider the $\chi$-component of
the above conjecture.
If $r<r_{\chi}$,
the both sides of the $\chi$-component of the above equation are zero.
So we assume that $r=r_{\chi}$.
Then the $\chi$-component $(\QQ\mathcal{O}_{K,V,T}^\times )^{\chi}$
coincides with the $\chi$-component
$(\QQ \mathcal{O}_{K,S,T}^\times )^{\chi}$,
so the $\chi$-component $\epsilon_{K/k,S,T}^{V,\chi}$ of
$\epsilon_{K/k,S,T}^V$ is in
$(\bigcap_G^r \mathcal{O}_{K,V,T}^\times)^{\chi}$ if we assume
the Rubin-Stark conjecture.
We know by \cite[Prop. 5.2]{R} that
$$\Ord_{V}(\epsilon_{K/k,S,T}^V)=\theta_{K/k,S \setminus V,T}(0).$$
Therefore, it follows from the definition of $ \mathcal{L}_{V}$ that
Conjecture \ref{GrossStark} is equivalent to
$$\Log_{V}(\epsilon_{K/k,S,T}^V)=
\lim_{s \rightarrow 0} L^{(r)}_{p, K/k, S, T}(s)$$
if we assume Conjecture \ref{rsconj}.
\end{rm}
\end{remark}

In the following, we consider the minus part of the $p$-part of our
conjectures.

By the minus part of the $p$-part of the leading term conjecture
${\rm LTC}(K/k)$ we mean the equality
\[ (\ZZ_{p}[G] z_{K/k,S,T})^- ={\det}_{\ZZ_p[G]}(D_{K,S,T}^\bullet \otimes \ZZ_{p})^-\]
and we denote it by ${\rm LTC}(K/k)_{p}^-$.

In the same way, we write ${\rm MRS}(K/L/k,S,T)_{p}^-$
to mean that $\mathcal{N}_H(\epsilon_{K/k,S,T}^V)^-$ belongs to $(\im (\nu) \otimes \ZZ_{p})^-$
and that the equality of
Conjecture \ref{mrsconj} is valid after projection to the group
$(\bigcap_{G}^r \mathcal{O}_{K,S,T}^\times \otimes_{\ZZ} \ZZ[H]/I(H)J_{W}
\otimes \ZZ_{p})^-$.

\begin{proposition} \label{MRStoGS}
Let $K_{\infty}/K$ be the cyclotomic $\ZZ_{p}$-extension and
$K_{n}$ the $n$-th layer.
If the conjecture ${\rm MRS}(K_{n}/K/k,S,T)^-_p$ is valid for
all sufficiently large $n$, then the conjecture
${\rm GS}(K/k,S,T)$ is also valid.
\end{proposition}

\begin{proof}
For sufficiently large $n$, no place in $S$ splits completely, so
$r$ for $K_{n}/k$ should be zero and
$\epsilon_{K_{n}/k,S,T}^\emptyset = \theta_{K_{n}/k,S,T}(0)$.
We put $\Gamma=\Gal(K_{\infty}/K)$, and take a generator $\gamma_{K}$
of $\Gal(K_{\infty}/K)$.
We also put $\Lambda_{K_{\infty}}=\ZZ_{p}[[\Gal(K_{\infty}/k)]]$, and
$t_{K}=\gamma_{K}-1 \in \Lambda_{K_{\infty}}$.

Then ${\rm MRS}(K_{n}/K/k,S,T)^-_p$ implies that
$\theta_{K_{n}/k,S,T}(0)$ belongs to $t_{K}^{r} \ZZ_{p}[\Gal(K_{n}/k)]$ and satisfies
\[ \theta_{K_{n}/k,S,T}(0)= \Rec_V(\epsilon_{K/k,S,T}^{V}) \in \ZZ_{p}[\Gal(K_{n}/k)]/t_{K}^{r+1}\ZZ_{p}[\Gal(K_{n}/k)].\]
Taking the limit of the above equality, we obtain
$$\theta_{K_{\infty}/k,S,T}=
\Rec_V(\epsilon_{K/k,S,T}^{V}) \ \ \mbox{in} \ \
\Lambda_{K_{\infty}}/(t_{K}^{r+1}).$$
Put $G=\Gal(K/k)$.
For $v_{i} \in V$, let
\[ \Rec_{i}:\mathcal{O}_{K,S,T}^{\times}
\rightarrow t_{K}\Lambda_{K_{\infty}}/t_{K}^{2}\Lambda_{K_{\infty}}\]
be the map
$\Rec_{i}(a)=\sum_{\tau \in G}\tau^{-1}(\rec_{w_{i}}(\tau a)-1)$
defined in \S \ref{secmrs} where
$w_{i}$ is the fixed place of $K$ above $v_{i}$.
For an element $\sigma=\gamma_{K}^{c} \in \Gamma$ with $c \in \ZZ_{p}$, we
define $\log_{\gamma_{K}}(\sigma)=c$.
We have
\begin{eqnarray*}
\Rec_{i}(a) &=& \sum_{\tau \in G}\tau^{-1}
\log_{\gamma_{K}}(\rec_{w_{i}}(\tau a))t_{K} \ {\rm mod} \ t_{K}^{2}\\
&=& \sum_{\tau \in G}\tau^{-1}
\log_{p}\kappa(\rec_{w_{i}}(\tau a))t_{K}(\log_{p}\kappa(\gamma_{K}))^{-1}
\ {\rm mod} \ t_{K}^{2}
\end{eqnarray*}
where $\kappa: \Gal(K_{\infty}/k) \rightarrow \ZZ_{p}^{\times}$
is the cyclotomic character as above.

By local class field theory, we know that
$\kappa(\rec_{w_{i}}(x))=\N_{K_{w_{i}}/{\QQ_{p}}}(x)^{-1}$ for any
$x \in K_{w_{i}}^{\times}$.
Therefore, $$\Rec_{i}(a)=-\sum_{\tau \in G}\tau^{-1}
\log_{p}\N_{K_{w_{i}}/{\QQ_{p}}}(\tau a)t_{K}(\log_{p}\kappa(\gamma_{K}))^{-1}
=\Log_{i}(a)t_{K}(\log_{p}\kappa(\gamma_{K}))^{-1}.$$
This implies that
$$\lim_{t_{K} \rightarrow 0} t_K^{-r}\theta_{K_{\infty}/k,S,T}
=\lim_{t_{K} \rightarrow 0} t_K^{-r}\Rec_V(\epsilon_{K/k,S,T}^{V})
=\frac{\Log_{V}(\epsilon_{K/k,S,T}^{V})}{(\log_{p}\kappa(\gamma_{K}))^{r}}
 \in \ZZ_{p}[G].$$

We take $F$ and $t$ in the same manner when we
defined $L_{p,K/k,S,T}(s)$,
and write $t_{K}=\gamma_{K}-1=g(t) \in \Lambda_{K_{\infty}}=\ZZ_{p}[\Gal(F/k)][[t]]$.

Suppose that $K_{i}=F_{n}$ for some $n$ and $i$.
Then we can take $\gamma_{F}^{p^{n}}=\gamma_{K}^{p^{i}}$.
Therefore, there is $\sigma \in \Gal(K_{\infty}/k)$ with finite order
such that $\gamma_{K}=\sigma \gamma_{F}^{p^{n-i}}$.
Thus, we have $g(t)=\sigma(1+t)^{p^{n-i}}-1$ and
$g(t_{F}\kappa(\gamma_{F})^{s}-1)=\gamma_{K} \kappa(\gamma_{K})^{s}-1$.
We regard $t_K^{-r}\theta_{K_{\infty}/k,S,T}$ as an element of
$\ZZ_{p}[\Gal(F/k)][[t]]$, then we have
$$c_{K_{\infty}/K}((t_K^{-r}\theta_{K_{\infty}/k,S,T})|_{
t=\gamma_{F}\kappa(\gamma_{F})^{s}-1})=
\frac{L_{p, K/k, S, T}(s)}
{(\kappa(\gamma_{K})^{s}-1)^{r}}.$$
Therefore, we obtain
$$
\lim_{t_{K} \rightarrow 0} t_K^{-r}\theta_{K_{\infty}/k,S,T}
=\lim_{s \rightarrow 0} \frac{L_{p, K/k, S, T}(s)}
{(\kappa(\gamma_{K})^{s}-1)^{r}}
=\frac{L^{(r)}_{p, K/k, S, T}(0)}{(\log_{p}\kappa(\gamma_{K}))^{r}}.$$
Comparing the above two equations, we have
$$L^{(r)}_{p, K/k, S, T}(0)=\Log_{V}(\epsilon_{K/k,S,T}^V).$$
By Remark \ref{Rem71}, we get the conclusion.
\end{proof}

We can now prove both Theorem \ref{MT4} and Corollary \ref{IntroConsequencesOfVentullo}.

\subsection{The proof of Theorem \ref{MT4}} The fact that claim (i) implies claim (ii) follows directly from Theorem \ref{ltcmrs}.

We next recall (from Remark \ref{remltc}) that the validity of ${\rm LTC}(K_{n}/k)^-_p$ implies that of ${\rm LTC}(F/k)^-_p$ for
each intermediate field $F$ of $K_{n}/k$. The fact that claim (ii) implies claim (iii) therefore follows from
Proposition \ref{MRStoGS}.

Finally, the fact that claim (iii) implies claim (i) was shown by the first author in \cite[Cor. 3.9]{burns2} (and see \cite{aze2} for a more general result in this direction).

This completes the proof of Theorem \ref{MT4}.

\subsection{The proof of Corollary \ref{IntroConsequencesOfVentullo}} Under the hypothesis that at most one $p$-adic place of $k$ splits in $K/K^+$, Gross has used Brumer's $p$-adic analogue of Baker's Theorem to prove that the order of vanishing at $s=0$ of
$L_{p, S, T}(\chi^{-1}\omega, s)$ is equal to the predicted value $|\{v \in S : \chi(G_v)=1\}|$ (see \cite[Prop. 2.13]{Gp}).

In addition, under the same hypothesis on $K/K^+$ Ventullo has recently proved the validity of the conjecture ${\rm GS}(K/k,S,T)$ in \cite{ventullo}.

Given these facts, the result of Corollary \ref{IntroConsequencesOfVentullo} follows directly from Theorem \ref{MT4}.

\section{Higher Fitting ideals of character components of class groups} \label{HigherFittChar}

In this section, as an application of
Theorem \ref{ThRelativeHigherFitt}, we study the higher Fitting ideals of
character components of class groups.

\subsection{General abelian extensions}
We suppose that $K/k$ is a finite abelian extension as in \S \ref{HigherFittH1}.
We take and fix an odd prime $p$ in this section.
We put $A_{S}^{T}(K)={\rm Cl}_{S}^{T}(K) \otimes \ZZ_{p}$,
$A^{T}(K)={\rm Cl}^T(K) \otimes \ZZ_{p}$, and
$A(K)={\rm Cl}(K) \otimes \ZZ_{p}$.

We take a character $\chi$ of $G=\Gal(K/k)$.
Throughout this section, we assume that {\it the order of $\chi$ is prime to $p$}.

We decompose $G=\Delta_{K} \times \Gamma_{K}$ where $| \Delta_{K}|$ is prime
to $p$ and $\Gamma_{K}$ is a $p$-group.
By our assumption $\chi$ is regarded as a character of $\Delta_{K}$.
For any $\ZZ_{p}[\Delta_{K}]$-module $M$, we define the $\chi$-component $M^{\chi}$ by setting
\[ M^{\chi} :=M \otimes_{\ZZ_{p}[\Delta_{K}]} \mathcal{O}_{\chi}\]
where
$\mathcal{O}_{\chi}=\ZZ_{p}[\im (\chi)]$ on which $\Delta_{K}$ acts via $\chi$.
This is an exact functor from the category of $\ZZ_{p}[\Delta_{K}]$-modules to
that of $\mathcal{O}_{\chi}$-modules.

Let $k_{\chi}$ be the subfield of $K$ corresponding to the kernel of $\chi$,
namely, $\chi$ induces a faithful character of $\Gal(k_{\chi}/k)$.
Put $K_{(\Delta)}:=K^{\Gamma_{K}}$, then $k_{\chi} \subset K_{(\Delta)}$.
We also put
$\Delta_{K,\chi}:=\Gal(K_{(\Delta)}/k_{\chi})$ which is a subgroup
of $\Delta_{K}$.
We consider $K(\chi):=K^{\Delta_{K,\chi}}$, then
$\Gal(K(\chi)/k_{\chi})=\Gamma_{K}$.
We consider $A_{S}^{T}(K)^{\chi}$ which we regard as an
$\mathcal{O}_{\chi}[\Gamma_{K}]$-module.
By the standard norm argument, we know the natural map
$A_{S}^{T}(K(\chi))^{\chi} \to
A_{S}^{T}(K)^{\chi}$ is bijective, so when we consider
the $\chi$-component $A_{S}^{T}(K)^{\chi}$, we may assume that
$\chi$ is a faithful character of $\Delta_K$ by replacing $K$ with $K(\chi)$.
In the following, we assume this.
We write $\chi(v) \neq 1$ if the decomposition group of $\Delta_{K}$ at $v$ is
non-trivial.

We denote the $\chi$-component of $\epsilon_{K/k,S,T}^{V}$
by $\epsilon_{K/k,S,T}^{V, \chi} \in
((\bigcap_{G}^r
\mathcal{O}_{K,S,T}^\times) \otimes \ZZ_{p})^{\chi}$.
Let $\mathcal{V}_{i}$ be the set as in Theorem \ref{MT2}(ii)
for $i \geq 0$.

Finally we assume that the following condition is satisfied
\begin{itemize}
\item[$(*)$] any ramifying place $v$ of $k$ in $K$ does not split completely in
$K_{(\Delta)}$.
\end{itemize}

\begin{theorem} \label{ThCharacterComponent1}
Let $r$ be the number of the infinite places of $k$ that split completely in $K$.
We assume that $\chi \neq 1$ is a faithful character of $\Delta_{K}$, and
consider the $\chi$-component of the class group $A^{T}(K)^{\chi}$ which
is an $\mathcal{O}_{\chi}[\Gamma_{K}]$-module.
We assume that the $\chi$-component of ${\rm LTC}(K/k)$ is valid and that the condition $({{*}})$ is satisfied.

Then for any non-negative integer $i$ one has an equality
$$\Fitt_{\mathcal{O}_{\chi}[\Gamma_{K}]}^{i}(A^T(K)^{\chi})
=
\{  \Phi(\epsilon_{K/k,S \cup V',T}^{V \cup V', \chi})  :  V' \in \mathcal{V}_{i} \
\mbox{and} \
\Phi \in \bigwedge_{\mathcal{O}_\chi[\Gamma_K]}^{r+i} \mathcal{H}_{\chi} \}$$
where $S=S_{\infty}(k) \cup S_{\rm ram}(K/k)$ and $\mathcal{H}_{\chi}=
\Hom_{\mathcal{O}_{\chi}[\Gamma_{K}]}((\mathcal{O}_{K,S \cup V',T}^\times
\otimes \ZZ_{p})^{\chi},\mathcal{O}_{\chi}[\Gamma_{K}])$.
\end{theorem}

\begin{proof}
Since $v \in S_{\rm ram}(K/k)$ does not split completely in $K_{(\Delta)}$, one has
$\chi(v) \neq 1$ and hence $(Y_{K,S_{\rm ram}}\otimes \ZZ_{p})^{\chi}=0$.
As $\chi \neq 1$, we therefore also have $(X_{K,S_{\rm ram}}\otimes \ZZ_{p})^{\chi}
=(Y_{K,S_{\rm ram}}\otimes \ZZ_{p})^{\chi}=0$.
Hence $(X_{K,S} \otimes \ZZ_{p})^{\chi} =
(Y_{K,S_{\infty}}\otimes \ZZ_{p})^{\chi}$ is isomorphic to $\mathcal{O}_{\chi}[\Gamma_{K}]^{\oplus r}$.
This implies that
\[ \Fitt_{\mathcal{O}_{\chi}[\Gamma_{K}]}^{(r,i)}
((\mathcal{S}^{{\rm tr}}_{S,T}(\GG_{m/K}) \otimes \ZZ_{p})^{\chi}), A_S^T(K)^{\chi})
=\Fitt_{\mathcal{O}_{\chi}[\Gamma_{K}]}^{i}(A_S^T(K)^{\chi})\]
and so the claim follows from Theorem \ref{ThRelativeHigherFitt}. \end{proof}

In the case $K=k_{\chi}$, the condition $({{*}})$ is automatically satisfied.
We denote the group $A^T(k_{\chi})^{\chi}$ by $(A^T)^{\chi}$, which is determined only
by $\chi$.

\begin{corollary} \label{CorCharacterComponent1}
Let $\chi$ be a non-trivial linear character of $k$ of order prime to $p$,
and $r$ the number of the archimedean places of $k$ that split completely
in $k_{\chi}$.
We assume the $\chi$-component of ${\rm LTC}(k_{\chi}/k)$ to be valid.
Then for any non-negative integer $i$ one has an equality
$$\Fitt_{\mathcal{O}_{\chi}}^{i}((A^T)^{\chi})
=
\{  \Phi(\epsilon_{k_{\chi}/k,S \cup V',T}^{V \cup V', \chi})  :  V' \in \mathcal{V}_{i} \
\mbox{and} \
\Phi \in \bigwedge_{\mathcal{O}_\chi}^{r+i} \mathcal{H}_{\chi} \}$$
where $S=S_{\infty}(k) \cup S_{\rm ram}(k_\chi/k)$ and
$\mathcal{H}_{\chi}=\Hom_{\mathcal{O}_{\chi}}
((\mathcal{O}_{k_{\chi},S \cup V',T}^\times \otimes \ZZ_{p})^{\chi},
\mathcal{O}_{\chi})$.
\end{corollary}

\subsection{The order of a character component of CM abelian extensions}
 In this subsection, we assume that $k$ is totally real, $K$ is a CM-field,
and $\chi$ is an odd character. In this case, we can compute
the right hand side of Theorem \ref{ThCharacterComponent1} more explicitly.
First of all, note that $r=0$ in this case.

We first consider the case $K=k_{\chi}$ and $i=0$.
When $S=S_{\infty}(k) \cup S_{\rm ram}(k_\chi/k)$, we denote the $L$-function
$L_{k,S,T}(\chi^{-1},s)$ by $L_{k}^T(\chi^{-1},s)$.
When $T$ is empty, we denote
$L_{k}^T(\chi^{-1},s)$ by $L_{k}(\chi^{-1},s)$.
In this case, we know
$$\epsilon_{k_{\chi}/k,S,T}^{\emptyset, \chi}=\theta_{k_{\chi}/k,S,T}(0)^{\chi}
=L_{k}^T(\chi^{-1},0)$$
(see \S \ref{secmrs}).
Therefore, Corollary \ref{CorCharacterComponent1} with $i=0$ implies

\begin{corollary} \label{CorCharacterComponent2}
Let $k$ be totally real, and $\chi$ a one dimensional odd character of $k$
of order prime to $p$.
We assume the $\chi$-component of ${\rm LTC}(k_{\chi}/k)$ to be valid.
\begin{itemize}
\item[(i)]{ One has
$| (A^T)^{\chi}| = |\mathcal{O}_{\chi}/L_{k}^T(\chi^{-1},0)|$. }
\item[(ii)]{ Let ${\rm Cl}(k_{\chi})$ be the ideal class group of $k_{\chi}$,
$A(k_{\chi})={\rm Cl}(k_{\chi}) \otimes \ZZ_{p}$, and $A^{\chi}=A(k_{\chi})^{\chi}$.
We denote by
$\omega$ the Teichm\"{u}ller character giving the Galois
action on $\mu_{p}$, the group of $p$-th roots of unity,
and by $\mu_{p^{\infty}}(k(\mu_{p}))$
the group of roots of unity of $p$-power order in $k(\mu_{p})$.
Then one has
$$|A^{\chi} |=
\left\{
\begin{array}{ll}
|\mathcal{O}_{\chi}/L_{k}(\chi^{-1},0)|  & \mbox{if} \ \chi \neq \omega, \\
|\mathcal{O}_{\chi}/(|\mu_{p^{\infty}}(k(\mu_{p}))|L_{k}(\chi^{-1},0))| &
\mbox{if} \ \chi=\omega.
\end{array}\right.
$$}
\end{itemize}
\end{corollary}

\begin{proof} Claim (i) is an immediate consequence of Corollary \ref{CorCharacterComponent1} and
a remark before this corollary. We shall now prove claim (ii).

When $\chi \neq \omega$, we take a finite place $v$ such that
$v$ is prime to $p$ and
${\N}v \not \equiv \chi({\rm Fr}_v)$ (mod $p$).
We put $T=\{v\}$.
Then $(A^T)^{\chi}=A^{\chi}$ and
$\ord_{p}L_{k}^T(\chi^{-1},0)=\ord_{p}L_{k}(\chi^{-1},0)$.
Therefore, claim (i) implies the equality in claim (ii).

When $\chi=\omega$, using Chebotarev density theorem
we take a finite place $v$ such that
$v$ splits completely in $k_{\chi}=k_{\omega}=k(\mu_{p})$ and
$\ord_{p}|\mu_{p^{\infty}}(k_{\chi})|=\ord_{p}({\N}v-1)$.
We take $T=\{v\}$, then
we also have $(A^T)^{\chi}=A^{\chi}$ from the exact sequence
$$\mu_{p^{\infty}}(k_{\chi}) \longrightarrow (\bigoplus_{w \mid v}
\kappa(w)^{\times} \otimes \ZZ_{p} )^{\chi}
\longrightarrow (A^T)^{\chi} \longrightarrow A^{\chi} \longrightarrow 0$$
where $w$ runs over all places of $k_{\chi}$ above $v$.
Therefore, claim (ii) follows from claim (i) in this case, too.
\end{proof}

\begin{corollary} \label{CorCharacterComponent3}
Assume that at most one $p$-adic place $\mathfrak{p}$ of $k$
satisfies $\chi(\mathfrak p)=1$.
Then the same conclusion as Corollary \ref{CorCharacterComponent2} holds.
\end{corollary}

\begin{remark}
\begin{rm}
We note that the formula on $A^{\chi}$ in Corollary \ref{CorCharacterComponent2} has
not yet been proved in general even in such semi-simple case (namely the case
that the order of $\chi$ is prime to $p$).
If no $p$-adic place $\mathfrak{p}$ satisfies $\chi(\mathfrak p)=1$,
this is an immediate consequence of the main conjecture proved by Wiles \cite{Wiles}.
Corollary \ref{CorCharacterComponent3} claims that this holds
even if $|\{\mathfrak{p}: \text{$p$-adic place such that }\chi(\mathfrak p)=1\}|=1$.
\end{rm}
\end{remark}

\subsection{The structure of the class group of a CM field} \label{strclassgroup}

Now we consider a general CM-field $K$ over a totally real number field $k$
(in particular, we do not assume that $K=k_{\chi}$).

We assume the condition $({{*}})$ stated just prior to Theorem \ref{ThCharacterComponent1}.

We fix a strictly positive integer $N$. Suppose that $v$ is a place of $k$ such that $v$ is prime to $p$,
$v$ splits completely in $K$
and there is a cyclic extension $F(v)/k$ of degree $p^{N}$,
which is unramified outside $v$ and in which
$v$ is totally ramified.
(Note that $F(v)$ is not unique.)
We denote by $\mathcal{S}(K)$ the set of such places $v$ and recall that $\mathcal{S}(K)$ is infinite (see \cite[Lem. 3.1]{Ku5}).

Suppose now that $V=\{v_{1},\ldots,v_{t}\}$ is a subset of $\mathcal{S}(K)$
consisting of $t$ distinct places.
We take a cyclic extension $F(v_{j})/k$ as above, and put
$F=F(v_{1}) \cdots  F(v_{t})$ the compositum of
fields $F(v_{j})$.
In particular, $F$ is totally real. We denote by $\mathcal{F}_{t,N}$ the set of all fields $F$ constructed
in this way.
When $t=0$, we define $\mathcal{F}_{0,N}=\{k\}$.

We set
\[ H:=\Gal(KF/K)\cong \Gal(F/k)\cong \prod_{j=1}^{t} \Gal(F(v_{j})/k),\]
where the first (restriction) isomorphism is due to the fact that $K \cap F=k$ and the second to the fact that each extension $F(v_{j})/k$ is totally ramified at $v_j$ and unramified at all other places.

We fix a generator $\sigma_{j}$ of $\Gal(F(v_{j})/k)$ and
set $S_{j}:=\sigma_{j}-1 \in \ZZ[\Gal(KF/k)]$. Noting that $\Gal(KF/k)=G \times H$ where
$G=\Gal(K/k)$, for each element $x$ of $\ZZ[\Gal(KF/k)]=\ZZ[G][H]$ we write $x=\sum x_{n_{1},\ldots,n_{t}} S_{1}^{n_{1}}\cdots S_{t}^{n_{t}}$
where each $x_{n_{1},\ldots,n_{t}}$ belongs to $\ZZ[G]$. We then define a map
\[ \varphi_{V}: \ZZ[\Gal(KF/k)] \rightarrow \ZZ/p^{N}[G]\]
by sending $x$ to $x_{1,\ldots,1}$ modulo $p^{N}$ and we note that this map is a well-defined homomorphism of $G$-modules.

We consider $\theta_{KF/k,S \cup V,T}(0) \in \ZZ[\Gal(KF/k)]$.
We define $\Theta^{i}_{N,S,T}(K/k)$ to be the ideal of
$\ZZ/p^{N}[G]$ generated by all
$\varphi_{V}(\theta_{KF/k,S \cup V,T}(0)) \in \ZZ/p^{N}[G]$
where $F$ runs over $\mathcal{F}_{t,N}$ such that $t \leq i$.
We note that we can compute $\theta_{KF/k,S \cup V,T}(0)$, and hence also
$\varphi_{V}(\theta_{KF/k,S \cup V,T}(0))$, numerically. Taking $F=k$, we know that $\theta_{K/k,S,T}(0)$ mod $p^{N}$ is in
$\Theta^{i}_{N,S,T}(K/k)$ for any $i \geq 0$.

We set $\mathcal{F}_{N}:=\bigcup_{t \geq 0} \mathcal{F}_{t,N}$.

For any abelian extension $M/k$, if $S=S_{\infty}(k) \cup S_{\rm ram}(M/k)$
and $T$ is the empty set, we write
$\theta_{M/k}(0)$ for $\theta_{M/k,S,T}(0)$.

We take a character $\chi$ of $\Delta_{K}$
such that $\chi \neq \omega$, at first.
We take $S=S_{\infty}(k) \cup S_{\rm ram}(K/k)$ and $T=\emptyset$.
In this case, we know that the $\chi$-component
$\theta_{KF/k}(0)^{\chi}$ is integral, namely is
in $O_{\chi}[\Gamma_{K} \times H]$.
We simply denote
the $\chi$-component
$\Theta^{i}_{N,S,\emptyset}(K/k)^{\chi}$
by $\Theta^{i}_{N}(K/k)^{\chi}$ ($\subset O_{\chi}[\Gamma_{K}]$).
This ideal $\Theta^{i}_{N}(K/k)^{\chi}$ coincides with the
higher Stickelberger ideal
$\Theta_{i,K}^{(\delta,N),\chi}$ defined in \cite[\S 8.1]{Ku5}.

When $\chi=\omega$, we assume that $K=k(\mu_{p^{m}})$ for some $m \geq 1$.
By using the Chebotarev density theorem we can choose a place $v$ which satisfies all of the following conditions
\begin{itemize}
\item[(i)]{ $v$ splits completely in $k(\mu_{p})/k$, }
\item[(ii)]{ each place above $v$ of $k(\mu_{p})$ is inert in $K/k(\mu_{p})$, and }
\item[(iii)]{ each place $w$ of $K$ above $v$ satisfies
$\ord_{p}|\mu_{p^{\infty}}(K)|=\ord_{p}({\N}w-1)$.}
\end{itemize}
We set $T:=\{v\}$.
We consider the $\omega$-component
$\Theta^{i}_{N,S,\{v\}}(K/k)^{\omega}$, which
we denoted by $\Theta^{i}_{N,\{v\}}(K/k)^{\omega}$

\begin{theorem} \label{ThHigherFittCM}
Let $K/k$ be a finite abelian extension, $K$ a CM-field, and
$k$ totally real.
Suppose that $\chi$ is an odd faithful character of $\Delta_{K}$, and
consider the $\chi$-component of the class group $A(K)^{\chi}$ which
is an $\mathcal{O}_{\chi}[\Gamma_{K}]$-module.
We assume the condition $({{*}})$ stated just prior to Theorem \ref{ThCharacterComponent1} and the validity of
the $\chi$-component of ${\rm LTC}(FK/k)$ for every field $F$ in $\mathcal{F}_{N}$.
\begin{itemize}
\item[(i)]{ Suppose that $\chi \neq \omega$.
For any integer $i \geq 0$, we have
$$\Fitt_{\mathcal{O}_{\chi}[\Gamma_{K}]/p^{N}}^{i}(A(K)^{\chi} \otimes \ZZ/p^{N})
= \Theta^{i}_{N}(K/k)^{\chi}.$$}
\item[(ii)]{ We assume that $K=k(\mu_{p^{m}})$ for some $m \geq 1$.
For $\chi=\omega$, using a place $v$ as above, we have
$$\Fitt_{\mathcal{O}_{\omega}[\Gamma_{K}]/p^{N}}^{i}(A(K)^{\omega} \otimes \ZZ/p^{N})
= \Theta^{i}_{N,\{v\}}(K/k)^{\omega}$$
for any $i \geq 0$. }
\end{itemize}
\end{theorem}

\begin{proof}
 We first prove claim (i). Since the image of
$\theta_{KF/k}(0)$ in $\ZZ[G]$
is a multiple of $\theta_{K/k}(0)$,
$\Theta^{0}_{N}(K/k)^{\chi}$ is a principal ideal
generated by $\theta_{K/k}(0)^{\chi}$.
Therefore, this theorem for $i=0$ follows from Theorem \ref{ThCharacterComponent1}.

Now suppose that $i>0$.
For a place $v \in \mathcal{S}(K)$, we take
a place $w$ of $K$ above $v$.
Put $H(v)=\Gal(F(v)K/K)=\Gal(F(v)/k) \simeq \ZZ/p^{N}$.
We take a generator $\sigma_{v}$ of $H(v)$ and fix it.
We define $\phi_{v}$ by
\begin{eqnarray*}
\phi_{v}: K^{\times}
&\stackrel{\Rec_{v}}{\longrightarrow} &(I(H(v))\ZZ[\Gal(F(v)K/k)]/
I(H(v))^{2}\ZZ[\Gal(F(v)K/k)])\\
&=&\ZZ[G] \otimes_\ZZ
I(H(v))/I(H(v))^{2} \simeq \ZZ/p^{N}[G]
\end{eqnarray*}
Here, the last isomorphism is defined by $\sigma_{v}-1 \mapsto 1$,
and $\Rec_{v}$ is defined by
$$\Rec_{v}(a)=\sum_{\tau \in G}\tau^{-1}({\rm{rec}}_{w}(\tau a)-1)$$
as in \S \ref{secmrs} by using the reciprocity map
${\rm rec}_w: K_{w}^{\times} \rightarrow H(v)$ at $w$.
Taking the $\chi$-component of $\phi_{v}$, we obtain
$$\phi_{v}: (K^{\times} \otimes \ZZ/p^{N})^{\chi}
\longrightarrow \mathcal{O}_{\chi}[\Gamma_{K}]/p^{N},$$
which we also denote by $\phi_{v}$.

We take $S=S_{\infty}(k) \cup S_{\rm ram}(K/k)$, $T=\emptyset$,
and $V=\{v_{1},\ldots,v_{i}\} \in \mathcal{V}_{i}$.
Suppose that $\Phi=\varphi_{1} \wedge\cdots\wedge
\varphi_{i}$ where
$$\varphi_{j} \in \mathcal{H}_{\chi}=
\Hom_{\mathcal{O}_{\chi}[\Gamma_{K}]}((\mathcal{O}_{K,S \cup V}^\times
\otimes \ZZ_{p})^{\chi},\mathcal{O}_{\chi}[\Gamma_{K}]/p^{N})$$
for $j=1,\ldots,i$.
We take a place $w_{j}$ of $K$ above $v_{j}$ for $j$
such that $1 \leq j \leq i$.
We denote by $[w_{j}]$ the class of $w_{j}$ in $A(K)^{\chi}$.

By \cite[Lem. 10.1]{Ku5}, for each integer $j=1$,...,$i$ we can choose a place $v_{j}' \in \mathcal{S}(K)$
 that satisfies all of the following conditions;
\begin{itemize}
\item[(a)]{ $[w_{j}']=[w_{j}]$ in $A(K)^{\chi}$ where $w_{j}'$ is a place of $K$
above $v_{j}'$,}
\item[(b)]{ $\varphi_{j}(x)=\phi_{v_{j}'}(x)$ for any $x \in
(\mathcal{O}_{K,S \cup V}^\times \otimes \ZZ/p^{N})^{\chi}$,}
\item[(c)]{ $\phi_{v_{m}'}(x_{j})=0$ for any $j$ and $m$ such that $j \neq m$
where $x_{j}$ is the element in $(\mathcal{O}_{K,S \cup V \cup V'}^\times \otimes \ZZ_{p})^\chi$
whose prime decomposition is $(x_{j})=w_{j}(w_{j}')^{-1}$ (the existence of $x_{j}$
follows from (a)).}
\end{itemize}
Here, we used the fact that the natural map
$(\mathcal{O}_{K,S \cup V}^\times \otimes \ZZ/p^{N})^{\chi} \rightarrow
(K^{\times} \otimes \ZZ/p^{N})^{\chi}$ is injective.

Set $V'=\{v_{1}',\ldots,v_{i}'\}$.
By property (b), we have
$$\Phi(\epsilon_{K/k,S \cup V,\emptyset}^{V, \chi})=
(\phi_{v_{1}'}\wedge\cdots\wedge \phi_{v_{i}'})
(\epsilon_{K/k,S \cup V,\emptyset}^{V, \chi}).$$
It follows from property (c) that if $i>1$, then
$$(\phi_{v_{1}'}\wedge\cdots\wedge \phi_{v_{i}'})
(\epsilon_{K/k,S \cup V,\emptyset}^{V, \chi}) =
(\phi_{v_{1}'}\wedge\cdots\wedge \phi_{v_{i}'})
(\epsilon_{K/k,S \cup V',\emptyset}^{V', \chi}) \in \mathcal{O}_{\chi}[\Gamma_{K}]/p^{N},$$
whilst for $i=1$ one has
\[ \phi_{v_{1}'}(\epsilon_{K/k,S \cup V,\emptyset}^{V, \chi})
\equiv \phi_{v_{1}'}(\epsilon_{K/k,S \cup V',\emptyset}^{V', \chi})
\,\,({\rm mod }\,\,\theta_{K/k}(0)^{\chi}).\]

Set $F=F(v_{1}')\cdots F(v_{i}')$ and $H=\Gal(FK/K)=\Gal(F/k)$.
Then as in \S\ref{secmrs} we can define
$\Rec_{V'}(\epsilon_{K/k,S \cup V',\emptyset}^{V', \chi})
\in \ZZ[G] \otimes (J_{V'})_{H}$.
Let $\varphi_{V'}:\ZZ[G \times H] \rightarrow
\ZZ/p^{N}[G]$ be the
homomorphism defined before Theorem \ref{ThHigherFittCM} by
using the generators $\sigma_{v_{i}'}$ we fixed.
This $\varphi_{V'}$ induces a homomorphism
\[ \ZZ[G \times H]/I(H)^{i+1}\ZZ[G \times H]
= \ZZ[G] \otimes \ZZ[H]/I(H)^{i+1}
\rightarrow \ZZ/p^{N}[G]\]
and we also denote the composite homomorphism
\[ \ZZ[G] \otimes (J_{V'})_{H} \rightarrow
\ZZ[G] \otimes \ZZ[H]/I(H)^{i+1}
\stackrel{\varphi_{V'}}{\rightarrow} \ZZ/p^{N}[G]\]
by $\varphi_{V'}$.

Then by the definitions of these homomorphisms, we have
$$(\phi_{v_{1}'}\wedge\cdots\wedge \phi_{v_{i}'})
(\epsilon_{K/k,S \cup V',\emptyset}^{V', \chi})=
\varphi_{V'}(\Rec_{V'}(\epsilon_{K/k,S \cup V',\emptyset}^{V', \chi})).$$
By Conjecture \ref{mrsconj} which is a theorem under our assumptions
(Theorem \ref{ltcmrs}), we get
$$\varphi_{V'}(\Rec_{V'}(\epsilon_{K/k,S \cup V',\emptyset}^{V', \chi}))
=\varphi_{V'}(\theta_{KF/k}(0)^{\chi}).$$
Combining the above equations, we get
$$\Phi(\epsilon_{K/k,S \cup V,\emptyset}^{V, \chi}) \equiv
\varphi_{V'}(\theta_{KF/k}(0)^{\chi}) \ (\mbox{mod} \
\theta_{K/k}(0)^{\chi}).$$
Since $\varphi_{V'}(\theta_{KF/k}(0)^{\chi})$, $\theta_{K/k}(0)^{\chi}$
are in $\Theta^{i}_{N}(K/k)^{\chi}$, we get
$\Phi(\epsilon_{K/k,S \cup V,\emptyset}^{V, \chi}) \in
\Theta^{i}_{N}(K/k)^{\chi}$.
It follows from
Theorem \ref{ThCharacterComponent1} that the left hand side of the equation in
Theorem \ref{ThHigherFittCM} (i) is in the right hand side.

On the other hand, suppose that $F$ is in $\mathcal{F}_{t,N}$
with $t \leq i$, and that
$V=\{v_{1},\ldots,v_{t}\}$ is the set of ramifying place in $F/k$.
As above, by Theorem \ref{ltcmrs} we have
$$\varphi_{V}(\theta_{KF/k}(0)^{\chi})=
\varphi_{V}(\Rec_{V}(\epsilon_{K/k,S \cup V,\emptyset}^{V, \chi}))
=(\phi_{v_{1}}\wedge\cdots\wedge \phi_{v_{t}})
(\epsilon_{K/k,S \cup V,\emptyset}^{V, \chi}).$$
Therefore, by Theorem \ref{ThCharacterComponent1} we have
${\varphi_V}(\theta_{KF/k}(0)^{\chi}) \in
\Fitt_{\mathcal{O}_{\chi}[\Gamma_{K}]/p^{N}}^{t}(A(K)^{\chi} \otimes \ZZ/p^{N})$.
Since $\Fitt_{\mathcal{O}_{\chi}[\Gamma_{K}]/p^{N}}^{t}(A(K)^{\chi} \otimes \ZZ/p^{N})
\subset \Fitt_{\mathcal{O}_{\chi}[\Gamma_{K}]/p^{N}}^{i}
(A(K)^{\chi} \otimes \ZZ/p^{N})$,
we get
$$\varphi_{V}(\theta_{KF/k}(0)^{\chi}) \in
\Fitt_{\mathcal{O}_{\chi}[\Gamma_{K}]/p^{N}}^{i}(A(K)^{\chi} \otimes \ZZ/p^{N}).$$
Thus, the right hand side of the equation in
Theorem \ref{ThHigherFittCM} (i) is in the left hand side.

We can prove claim (ii) by the same method.
The condition on $v$ is used to show the injectivity of
the natural homomorphism
$(\mathcal{O}_{K,S \cup V,T}^\times \otimes \ZZ/p^{N})^{\omega} \rightarrow
(K^{\times} \otimes \ZZ/p^{N})^{\omega}$
with $T=\{v\}$.
\end{proof}

\begin{corollary} \label{CorHigherFittCM1}
Let $K/k$ and $\chi$ be as in Theorem \ref{ThHigherFittCM}.
We assume the condition $({{*}})$ stated just prior to Theorem \ref{ThCharacterComponent1}
and that there is at most one place $\mathfrak{p}$ of $k$
above $p$ such that $\chi(\mathfrak p)=1$.
Then the same conclusion as Theorem \ref{ThHigherFittCM} holds.
\end{corollary}

\begin{proof} It suffices to note that, under the stated conditions, the $\chi$-component of ${\rm LTC}(FK/k)$ is valid
for each $F$ in $\mathcal{F}_{N}$. This follows by combining Theorem \ref{MT4} with an argument of Gross and the main result of
Ventullo in \cite{ventullo} (as in the proof of Corollary \ref{IntroConsequencesOfVentullo} given above).
\end{proof}

To give an example of Corollary \ref{CorHigherFittCM1} we suppose that $K$ is the $m$-th layer of the
 cyclotomic $\ZZ_{p}$-extension of $K_{(\Delta)}$ for some strictly positive integer $m$, and assume that $\chi(\mathfrak p) \neq 1$
for any $\mathfrak{p} \mid p$.

Then this assumption implies that the condition $({{*}})$ is satisfied and so all of the assumptions in Corollary \ref{CorHigherFittCM1}
are satisfied. Therefore, by taking the projective limit of the conclusion,
Corollary \ref{CorHigherFittCM1} implies the result of the second author in \cite[Th. 2.1]{Ku5}.

In this sense, Corollary \ref{CorHigherFittCM1} is a natural generalization
of the main result in \cite{Ku5}.

\vspace{5mm}

To state our final result we now set
\[ \Theta^{i}(K/k)^{\chi}={\lim\limits_{\longleftarrow}}_{N}
\Theta^{i}_{N}(K/k)^{\chi} \subseteq \mathcal{O}_{\chi}[\Gamma_{K}].\]

Then Theorem \ref{ThHigherFittCM} implies that ${\rm Fitt}^{i}_{O_{\chi}[\Gamma]}(A(K)^{\chi})=\Theta^{i}(K/k)^{\chi}$.

Let $k_{\chi}$ be the field corresponding to the kernel of $\chi$ as in
Corollary \ref{CorCharacterComponent2}.
We denote $\Theta^{i}(k_{\chi}/k)^{\chi}$ by $\Theta^{i,\chi}$, which
is an ideal of $\mathcal{O}_{\chi}$.
For $\chi=\omega$, we denote
${\lim\limits_{\longleftarrow}}_{N} \Theta^{i}_{N,\{v\}}(k_{\chi}/k)^{\omega}$
by $\Theta^{i,\omega}$.

Then Corollary \ref{CorHigherFittCM1} implies the following result, which
is a generalization of the main result of the second author in \cite{Ku3}.

\begin{corollary} \label{CorHigherFittCM2}
Set $A^{\chi}:=({\rm Cl}(k_{\chi}) \otimes \ZZ_{p})^{\chi}$ as in Corollary \ref{CorCharacterComponent2}.
 Assume that there is at most one $p$-adic place $\mathfrak{p}$ of $k$
such that $\chi(\mathfrak p)=1$ and that the $p$-adic Iwasawa $\mu$-invariant of
$K$ vanishes.

Then there is an isomorphism of $\mathcal{O}_{\chi}$-modules of the form $A^{\chi} \simeq \bigoplus_{i \geq 1} \Theta^{i,\chi}/\Theta^{i-1,\chi}.$
\end{corollary}
\

\noindent{}{\bf Acknowledgements} It is a pleasure for the first author to thank Dick Gross for much encouragement at an early stage of this general project and, in addition, to thank Cornelius Greither for stimulating discussions.
The second author would like to thank Cornelius Greither for discussions with him on various topics related to the subjects in this paper.
He also thanks Kazuya Kato and Karl Rubin for their beautiful ideas to define several zeta elements, and for stimulating conversations and discussions with them.
The third author would like to thank Kazuya Kato for his interest in the works of the third author and for encouragement.

\end{document}